\documentclass[12pt]{amsart}
\usepackage{amsmath, amssymb, latexsym, amsthm, amscd, mathrsfs, stmaryrd}
\usepackage{amsfonts, mathrsfs}
\usepackage{mathtools}
\usepackage[title, titletoc, toc]{appendix}
\usepackage[all]{xy}
\usepackage{soul}
\usepackage{tikz}
\usetikzlibrary{arrows}
\usetikzlibrary{calc}
\usepackage{pgfplots}
\pgfplotsset{compat=1.17}
\usepackage{cite}
\usepackage{enumerate}
\usepackage{tikz-cd}
\usepackage{yhmath}
\usepackage{faktor}

\newtheoremstyle{noparens}%
{}{}%
{\itshape}{}%
{\bfseries}{.}%
{ }%
{\thmname{#1}\thmnumber{ #2}\mdseries\thmnote{ #3}}
\theoremstyle{noparens}
\newtheorem{Def}[subsubsection]{Definition}%[section]
\newtheorem{example}[subsubsection]{Example}%[section]
\newtheorem{rem}[subsubsection]{Remark}%[section]
\newtheorem{prop}[subsubsection]{Proposition}
\newtheorem{thm}[subsubsection]{Theorem}
\newtheorem{lem}[subsubsection]{Lemma}
\newtheorem{cor}[subsubsection]{Corollary}

\newcommand{\mrm}{\mathrm}

\newcommand{\Q}{\mathbb Q}

\newcommand{\p}{ {\rm p}}
\newcommand{\ro}{\mrm{ro}}
\newcommand{\co}{\mrm{co}}

\newenvironment{eq}
{
	\begin{equation*}
		\begin{aligned}
		}
		{
		\end{aligned}
	\end{equation*}
}

\newcommand{\End}{\mrm{End}}

\usepackage{color}
\newcommand{\nc}{\newcommand}
\nc{\redtext}[1]{\textcolor{red}{#1}}
\nc{\bluetext}[1]{\textcolor{blue}{#1}}
\nc{\greentext}[1]{\textcolor{green}{#1}}
\nc{\yl}[1]{\redtext{From yq: #1}}
\nc{\zb}[1]{\redtext{From zb: #1}}
\newcommand\enc[1]{
	\tikz[baseline=(X.base)] 
	\node (X) [draw, shape=circle, inner sep=0] {$ #1 $};}

\newcommand\senc[1]{
	\tikz[baseline=(X.base)] 
	\node (X) [scale=0.7, draw, shape=circle, inner sep=0] {$ #1 $};}

\allowdisplaybreaks

\title[Geometric approach to Mirabolic Schur-Weyl Duality of type $A$]{Geometric approach to Mirabolic Schur-Weyl Duality of type A}
\author{Zhaobing Fan}
\address{College of Mathematical Sciences\\ Harbin Engineering University, Harbin, China 150001}
\email{fanzhaobing@hrbeu.edu.cn}
\author{Zhicheng Zhang}
\address{College of Mathematical Sciences\\ Harbin Engineering University, Harbin, China 150001}
\email{zhichengzhang@hrbeu.edu.cn}
\author{Haitao Ma}
\address{College of Mathematical Sciences\\ Harbin Engineering University, Harbin, China 150001}
\email{hmamath@hrbeu.edu.cn}

\date{}

\begin{document}
	
	\begin{abstract}
		We commence by constructing the mirabolic quantum Schur algebra, utilizing the convolution algebra defined on the variety of triples of two $n$-step partial flags and a vector. Subsequently, we employ a stabilization procedure to derive the mirabolic quantum $\mathfrak{gl}_n$. Then we present the geometric approach of the mirabolic Schur-Weyl duality of type $A$.
	\end{abstract}

	\maketitle
	\tableofcontents
	
	\section{Introduction}
	For the general linear algebraic group $G=GL_d(\mathbb{F}_q)$, Beilinson, Lusztig and MacPherson introduced a geometric realization of the quantum group $U_q(\mathfrak{gl}_n)$, utilizing the convolution algebra on the variety of pairs of partial flags in a vector space $V$ over finite field\cite{BLM}. This construction is known as the BLM realization.
	As a consequence, Grojnowski and Lusztig realized the quantum Schur-Weyl duality geometrically\cite{GL}, while the algebraic perspective was elucidated by Jimbo \cite{Jim}.
	
	A natural extension in the exploration of quantum groups involves venturing into the `mirabolic' setting, where the term `mirabolic' originates from the mirabolic subgroup. This subgroup $M$ is defined as a subgroup of $GL_{d}(\mathbb{C})$ that fixes a non-zero vector in $V$. Broadly, for a $GL_{d}(\mathbb{C})$-variety $X$, the $M$-orbits on $X$ exhibit a one-to-one correspondence with the $GL_{d}(\mathbb{C})$-orbits on $X\times V\setminus\{0\}$. 
	
	One compelling reason for the prominence of the mirabolic analogy lies in the category of mirabolic $\mathscr{D}$-modules, a category of $\mathscr{D}$-modules on $GL_{d}(\mathbb{C})\times V$, which bears a close relationship with the spherical trigonometric Cherednik algebra\cite{FG}. 
	Moreover, the study of the category of mirabolic $\mathscr{D}$-modules is intertwined with the category of $\mathscr{D}$-modules on $\mathscr{Y}\times \mathscr{Y} \times \mathbb{P}^{d}$, where $\mathscr{Y}$ is the complete flag variety and $\mathbb{P}^{d}$ is the projective space \cite{FG1}. 
	
	Travkin delved into the convolution product on $\mathscr{Y}\times \mathscr{Y}\times V$ to yield a bimodule $\mathcal{MH}$ of Hecke algebra and defined its Kazhdan-Lusztig basis\cite{T}. Moreover, Rosso studied $\mathcal{MH}$ as a convolution algebra, providing a geometric realization of the mirabolic Hecke algebra\cite{R14}, which was originally defined by Solomon in \cite{S}. This geometric realization also serves to recover the classification of irreducible representations of the mirabolic Hecke algebra as outlined in \cite{Si}. 
	
	Extending the scope from $\mathscr{Y}$ to the $n$-step partial flag varieties $\mathscr{X}$, one can construct a convolution algebra called mirabolic quantum Schur algebra $\mathcal{MS}_{n,d}$. Rosso's work on the algebra $\mathcal{MS}_{2,d}$ in \cite{R18} led to the definition of the mirabolic quantum $\mathfrak{sl}_2$, denoted as $MU_2$. He further classified finite dimension irreducible representations of $MU_2$ and provided an algebraic proof of the mirabolic Schur-Weyl duality.
	
	The aim of this paper is to construct the algebra $\mathcal{MS}_{n,d}$, featuring more intricate structure. The orbits of $G$-action of $\mathscr{X}\times \mathscr{X}\times V$ are parametrized by the decorated matrices $(A,\Delta)$ (see Definition \ref{def 3.1.1}). For $n>2$, there are some $\Delta$ that never appear when $n=2$ and the elements in $\Delta$ are interwoven in calculations. Both of them cause the multiplication formulas are more complicated. See Remark \ref{rem 4.1.2} for more details. To our best knowledge, this phenomenon never appears before in BLM type approach for various quantum type algebra. Besides, this phenomenon also causes a more complicated way to generate arbitrary element in $\mathcal{MS}_{n,d}$, see Proposition \ref{prop 4.2.1}. Subsequently, we proceed to construct the mirabolic quantum group. Differing from Rosso's work, we employ the stabilization procedure to derive the mirabolic quantum group. In our setting, the procedure will involve the new generator of $\mathcal{MS}_{n,d}$, then we will obtain a new algebra $\mathcal{K}$ and a surjective algebra homomorphism form $\mathcal{K}$ to $\mathcal{MS}_{n,d}$. In addition, we present a geometric approach to the mirabolic Schur-Weyl duality, while Rosso established the algebraic perspective of it. Unlike the classical works, the straightforward verification of double centralizer property through the theorem \cite[Theorem 2.1]{P} is unattainable, since the theorem only involves two flag varieties while mirabolic setting has an additional vector space. To establish the double centralizer theorem in the mirabolic setting, we introduce the algebra $MS$, defined as the space of $V \rtimes G$-invariant $\mathbb{Z}[v,v^{-1}]$-valued functions on $(\mathscr{X}\times V) \times (\mathscr{X} \times V)$ with the convolution product as multiplication. By replacing $\mathscr{X}$ with $\mathscr{Y}$, we construct the algebra $MH$. We demonstrate that $\mathscr{X}\times V$ and $\mathscr{Y}\times V$ satisfy the conditions outlined in the theorem in \cite[Theorem 2.1]{P} as $V\rtimes G$-set. Continuing, we establish algebra isomorphisms $\mathcal{MS}_{n,d}\cong MS$ and $\mathcal{MH}\cong MH$. Finally, we prove the double centralizer theorem for $\mathcal{MS}_{n,d}$ and $\mathcal{MH}$.
	
	The foundational works \cite{BLM} and \cite{GL} also served as wellspring of inspiration for various generalizations. 
	For instance, 
	Bao and Wang generalized the Schur-Jimbo duality to establish a new duality 
	between $\imath$quantum group and Iwahori-Hecke algebra of type $B/C$ in\cite{BW13}. 
	Their subsequent work in \cite{BKLW} presents a geometric construction of the $\imath$quantum groups $\mathbf{U}^{\imath}$ and $\mathbf{U}^{\jmath}$, providing a geometric realization of this new duality. 
	In a related vein, Fan and Li leveraged the geometry of partial flag varieties of type $D$ 
	to construct the quantum algebras $\mathcal{K}$ and $\mathcal{K}^m$, subsequently verifying the type-$D$ duality \cite{FL}. Additional insights into the BLM realization of affine quantum groups can be found in \cite{DF15}, \cite{DF19}, \cite{DF22}, \cite{FLLLW}, \cite{GV}, \cite{Lu99} and \cite{Lu00}. In our forthcoming work, we aim to generalize the mirabolic setting to the $\imath$quantum group and the affine case.
	
	The paper is organized as follows. In Section 2, we present the main result in our paper.
	In Section 3, we describe the $G$-orbits of $\mathscr{X}\times \mathscr{X}\times V$ and compute the dimension of the orbits. Concurrently, we construct mirabolic quantum Schur algebra $\mathcal{MS}_{n,d}$. Additionally, we provide a review of the geometric realization of the mirabolic Hecke algebra $\mathbf{MH}$ and the space $\mathcal{MV}$ defined in \cite{R18}.
	In Section 4, we delve into the computation of the multiplication formulas for the convolution product on $\mathcal{MS}_{n,d}$ and demonstrate the generators of the algebra. We then present explicit formulas for both the left $\mathcal{MS}_{n,d}$-action and the right $\mathcal{MH}$-cation on $\mathcal{MV}$. 
	In Section 5, our focus shifts to the construction of the algebra $\mathcal{K}$ 
	through the stabilization procedure, and gain the mirabolic quantum $\mathfrak{gl}_n$. In Section 6, we give the geometric approach of the mirabolic Schur-Weyl duality of type $A$. 
	\subsection*{Acknowledgement} Z. Fan was partially supported by the NSF of China grant 12271120 and the Fundamental Research Funds for the central universities.
	
	\section{Mirabolic Schur-Weyl Duality}
	\subsection{Mirabolic setting}
	We first recall the mirabolic Hecke algebra, which is introduced in \cite{S}. 
	\begin{Def}
		The mirabolic Hecke algebra $\mathbf{MH}$ is an associative $\mathbb{Q}(v)$-algebra with generators $\{\tau_i\ |\ i=0, \cdots, d-1\}$ and the following relations.
		\begin{align*}
			\tau^2_0&=(v^2-2)\tau_0+(v^2-1);\\
			\tau^2_i&=(v^2-1)\tau_i+v^{2} \quad i\geq 1;\\
			\tau_i \tau_{i+1}\tau_i&=\tau_{i+1}\tau_{i}\tau_{i+1} \quad i\geq 1;\\
			\tau_0\tau_{1}\tau_0\tau_1&=(v^{2}-1)(\tau_{1}\tau_{0}\tau_{1}+\tau_1\tau_0)-\tau_0\tau_1\tau_0;\\
			\tau_1\tau_{0}\tau_1\tau_0&=(v^{2}-1)(\tau_{1}\tau_{0}\tau_{1}+\tau_0\tau_1)-\tau_0\tau_1\tau_0;\\
			\tau_i \tau_{j}&=\tau_{j}\tau_{i} \quad |i-j|\geq 2.
		\end{align*}
	\end{Def}
	Let $[a,b]$ be the set of integers between $a$ and $b$. 
	\begin{Def}\label{def 2.1.2}
		The mirabolic quantum group $\mathbf{MU}$ is a $\mathbb{Q}(v)$-algebra generated by $L$, $E_i$, $F_i$ and $H^{\pm}_a$ for $i\in [1, n-1]$ and $a\in [1, n]$, with the following relations.
		\begin{align*}
			& H_{a}H^{-1}_{a}=1, H_{a}H_{b}=H_{b} H_{a};\\
			& E_{i}^2E_{j}+E_{j}E_{i}^2=(v+v^{-1})E_{i}E_{j}E_{i} \quad |i-j|=1;\\
			& F_{i}^{2}F_{j}+F_{j}F_{i}^2=(v+v^{-1})F_{i}F_{j}F_{i} \quad |i-j|=1;\\
			& E_{i}E_{j}=E_{j}E_{i},  F_{i}F_{j}=F_{j}F_{i} \quad |i-j|>1;\\
			& H_{a}E_{i}=v^{\delta_{a i}-\delta_{a,i+1}}E_{i}H_{a};\\
			& H_{a}F_{i}=v^{-\delta_{a i}+\delta_{a,i+1}}F_{i}H_{a};\\
			& E_{i}F_{j}-F_{j}E_{i}=\delta_{i j}\frac{H_{i}H^{-1}_{i+1}-H^{-1}_{i}H_{i+1}}{(v-v^{-1})};\\
			& H_{a}L=LH_{a}, L^2=L;\\
			& LE_{i}=LE_{i}L, LF_{i}=LF_{i}L;\\
			& (v+v^{-1}) E_{i}LE_{i}=v^{-1}E_{i}^{2}L+vLE_{i}^{2};\\
			& (v+v^{-1}) F_{i}LF_{i}=vF_{i}^{2}L+v^{-1}LF_{i}^{2}.
		\end{align*}
	\end{Def}
	
	\subsection{The duality}
	For fixed positive integers $n$ and $d$, we define a $\mathbb{Q}(v)$-linear space $\mathbf{MV}$ with basis 
	$$\{m_{\mathbf{r}, \mathbf{j}} \ |\ \mathbf{r}=(r_1,\ldots, r_d)\in S_{row}, \mathbf{j}\in J_{\mathbf{r}}\},$$
	where 
	$$S_{row}=\{ \mathbf{r}=(r_1,\ldots, r_d)\ |\ 1 \leq r_i\leq n \text{ for all }i\in [1,d]\},$$
	$$ J_{\mathbf{r}} = \{(j_1, j_2, \cdots, j_k)\ | \ 0\leq k\leq n, 1\leq j_k< \cdots< j_1\leq d, r_{j_1}< r_{j_2}< \cdots< r_{j_k}\}.$$ 
	Moreover, by\cite[Lemma 6.3]{R18} the dimension of this space is 
	$$\sum^{\min \{n,d\}}_{k=1} \binom{n}{k}\binom{d}{k}n^{d-k}.$$
	For $\mathbf{j}=(j_1, \cdots, j_k)$, let $\mathbf{j}\setminus(j_l)=(j_1, \cdots, j_{l-1}, j_{l+1}, \cdots, j_k)$. For a sequence $\mathbf{r}=(r_1, \cdots, r_d)$ and a fixed $p\in [1, n]$, we define sequences $\mathbf{r}'_p$ and $\mathbf{r}''_p$ by 
	\begin{align*}
		(\mathbf{r}'_p)_j=\left\{
		\begin{array}{cc}
			r_j  & j\neq p, \\
			r_j+1 & j=p,
		\end{array}
		\right.
		\text{ and }
		(\mathbf{r}''_p)_j=\left\{
		\begin{array}{cc}
			r_j  & j\neq p, \\
			r_j-1 & j=p.
		\end{array}
		\right.
	\end{align*}
	\begin{lem}\label{lem 2.3.1}
		For any $m_{\mathbf{r}, \mathbf{j}}\in \mathbf{MV}$, where $\mathbf{r}=(r_1, \cdots, r_d)$ and $\mathbf{j}=(j_1, \cdots, j_k)$, the following formulas give a well-defined left $\mathbf{MU}$-action on $\mathbf{MV}$.
		
		\noindent$(a)$ For $a\in [1,n]$
		\begin{align*}
			H^{\pm}_{a}m_{\mathbf{r}, \mathbf{j}}=v^{\mp\sum_{1\leq j \leq d} \delta_{a r_j}}m_{\mathbf{r}, \mathbf{j}}.
		\end{align*}
		
		\noindent$(b)$ If $\mathbf{j}=0$, we have 
		\begin{eq}
			\relax L m_{\mathbf{r}, \mathbf{j}}=& 
			v^{-\sum_{j \in [1, d]}\delta_{1 r_{j}}}(m_{\mathbf{r}, \mathbf{j}} +\sum_{p \in [1, d], r_p>0}m_{\mathbf{r}, (p)}).
		\end{eq}
		
		\noindent$(c)$ If $r_{j_1}=1$, then 
		\begin{eq}
			\relax L m_{\mathbf{r}, \mathbf{j}}=& 
			v^{-2\sum_{j \in [1, d]}\delta_{1 r_{j}}+2\sum_{j<j_1}\delta_{1 r_j}}(v^{2}-1)m_{\mathbf{r}, \mathbf{j}\setminus{\{j_1\}}}\\
			& + v^{-2\sum_{j \in [1, d]}\delta_{1 r_{j}}}\sum_{p>j_1, \mathbf{r}_p>0}v^{2\sum_{j<j_1}\delta_{1 r_j}}(v^{2}-1)m_{\mathbf{r},\mathbf{j}_p} \\
			& + v^{-2\sum_{j \in [1, d]}\delta_{1 r_{j}}}(v^{2\sum_{j \leq j_1}\delta_{1 r_j}}-1)m_{\mathbf{r}, \mathbf{j}}, 
		\end{eq}
		where $\mathbf{j}_p=(p, j_2, \cdots, j_k)$.
		
		\noindent$(d)$ If $r_{j_1}>1$, we have
		\begin{eq}
			\relax L m_{\mathbf{r}, \mathbf{j}}= &
			v^{-\sum_{j \in [1, d]}\delta_{1 r_{j}}}v^{2\sum_{j \leq j_1}\delta_{1 r_j}}m_{\mathbf{r}, \mathbf{j}} \\
			& +v^{-\sum_{j \in [1, d]}\delta_{1 r_{j}}}\sum_{p>j_1, \mathbf{r}_{p}>0}v^{2\sum_{j \leq j_1}\delta_{1 r_j}}m_{\mathbf{r}, \mathbf{j}_p},
		\end{eq}
		where $\mathbf{j}_p=(p, j_1, j_2, \cdots, j_k)$.
		
		\noindent$(e)$ If for any $j_m$, $r_{j_m}\neq i, i+1$, then
		\begin{align*}
			E_{i}m_{\mathbf{r}, \mathbf{j}} = & v^{-\sum_{j \in [1, d]}\delta_{i r_j}}\sum\limits_{p \in [1, d], p\neq j_m r_p=i+1}v^{2\sum_{j>p}\delta_{i r_{j}}}m_{\mathbf{r}''_p, \mathbf{j}};\\
			F_{i}m_{\mathbf{r}, \mathbf{j}} = & v^{-\sum_{j \in [1, d]}\delta_{i+1, r_j}}\sum\limits_{p \in [1, d], r_p=i}v^{2\sum_{j<p}\delta_{i+1, r_{j}}}m_{\mathbf{r}'_p, \mathbf{j}}.\\
		\end{align*}
		
		\noindent$(f)$ If there is $j_m\in \mathbf{j}$ such that $r_{j_m}=i$ and $r_{j_{m+1}}\neq i+1$, then
		\begin{align*}
			E_{i}m_{\mathbf{r}, \mathbf{j}}= & v^{-\sum_{j \in [1, d]}\delta_{i r_j}}\sum\limits_{p \in [1, j_{m+1}]\cup [j_m+1, d], r_p=i+1}v^{2\sum_{j>p}\delta_{i r_{j}}}m_{\mathbf{r}''_p, \mathbf{j}}\\
			& + v^{-\sum_{j \in [1, d]}\delta_{i r_j}}\sum\limits_{p \in [j_{m+1}+1, j_{m}-1], r_p=i+1}v^{2\sum_{j>p}\delta_{i, r_{j}}-1}m_{\mathbf{r}''_p, \mathbf{j}};\\
			F_{i}m_{\mathbf{r}, \mathbf{j}}= & v^{-\sum_{j \in [1, d]}\delta_{i+1, r_j}}\sum\limits_{p \in [1, d], p \neq j_m r_p=i}v^{2\sum_{j<p}\delta_{i+1, r_{j}}}m_{\mathbf{r}'_p, \mathbf{j}}\\
			& + v^{-\sum_{j \in [1, d]}\delta_{i+1, r_j}}\sum\limits_{p= j_m}v^{2\sum_{j<p}\delta_{i+1, r_{j}}}m_{\mathbf{r}'_p, \mathbf{j}\setminus(j_m)}\\
			& + v^{-\sum_{j \in [1, d]}\delta_{i+1, r_j}}\sum\limits_{t \in [j_{m+1}+1, j_m-1], r_t=i} v^{2\sum_{j\leq j_{m+1}}\delta_{i+1,r_{j}}}m_{\mathbf{r}'_t, \mathbf{j}_t}\\
			& + v^{-\sum_{j \in [1, d]}\delta_{i+1, r_j}}v^{2\sum_{j\leq j_{m+1}}\delta_{i+1,r_{j}}}m_{\mathbf{r}'_{j_m}, \mathbf{j}},
		\end{align*}
		where $\mathbf{j}_t= (j_1, \cdots, j_m, t, j_{m+1} , \cdots, j_{k-1}, j_k)$.
		
		\noindent$(g)$ If there is $j_m\in \mathbf{j}$ such that $r_{j_m}=i+1$ and $r_{j_{m-1}}\neq i$, then
		\begin{align*}
			E_{i}m_{\mathbf{r}, \mathbf{j}} = & v^{-\sum_{j \in [1, d]}\delta_{i r_j}}\sum\limits_{p \in [1, d], p \neq j_m,  r_p=i+1}v^{2\sum_{j>p}\delta_{i r_{j}}}m_{\mathbf{r}''_p, \mathbf{j}}\\
			& + v^{-\sum_{j \in [1, d]}\delta_{i r_j}}\sum\limits_{p=j_m}v^{2\sum_{j>p}\delta_{i r_{j}}}m_{\mathbf{r}''_p, \mathbf{j}\setminus(j_m)}\\
			& + v^{-\sum_{j \in [1, d]}\delta_{i r_j}}\sum\limits_{p = j_m}v^{2\sum_{j>p}\delta_{i r_{j}}}m_{\mathbf{r}''_p, \mathbf{j}}\\
			& + v^{-\sum_{j \in [1, d]}\delta_{i r_j}}\sum\limits_{\substack{p = j_m,  r_t=i+1,\\t \in [j_{m+1}+1, j_m-1]}} v^{2\sum_{j>p}\delta_{i r_{j}}}m_{\mathbf{r}''_{j_{m+1}}, \mathbf{j}_t},
		\end{align*}
		where $\mathbf{j}_t= (j_1, \cdots, j_m, t, j_{m+1} , \cdots, j_{k-1}, j_k)$;
		\begin{align*}
			F_{i}m_{\mathbf{r}, \mathbf{j}} =v^{-\sum_{j \in [1, d]}\delta_{i+1, r_j}}\sum\limits_{p \neq j_m, r_p=i}v^{2\sum_{j<p}\delta_{i+1, r_{j}}}m_{\mathbf{r}'_p, \mathbf{j}}.
		\end{align*}
		
		\noindent$(h)$ If there are $j_{m}, j_{m+1}\in \mathbf{j}$ such that $r_{j_m}=i$ and $r_{j_{m+1}}=i+1$, then
		\begin{align*}
			E_{i}m_{\mathbf{r}, \mathbf{j}} = & v^{-\sum_{j \in [1, d]}\delta_{i r_j}}\sum\limits_{p \in [1, j_{m+1}]\cup [j_m+1, d], r_p=i+1}v^{2\sum_{j>p}\delta_{i r_{j}}}m_{\mathbf{r}''_p, \mathbf{j}}\\
			& + v^{-\sum_{j \in [1, d]}\delta_{i r_j}}\sum\limits_{p \in [j_{m+1}+1, j_{m}-1], r_p=i+1}v^{2\sum_{j>p}\delta_{i r_{j}}-1}m_{\mathbf{r}''_p, \mathbf{j}}\\
			& + v^{-\sum_{j \in [1, d]} \delta_{i r_j}}\sum\limits_{p = j_{m+1}}v^{2(\sum_{j<j_{m+1}}\delta_{i+1, r_{j}}-1)} (v^2-1)m_{\mathbf{r}''_p, \mathbf{j}\setminus{j_{m+1}}}\\
			& + v^{-\sum_{j \in [1, d]} \delta_{i r_j}}\sum\limits_{\substack{p = j_{m+1}, \\ j_{m+2}<t<j_{m+1}, r_{t}=h+1}} v^{2(\sum_{j<j_{m+1}}\delta_{i+1, r_{j}}-1)} (v^2-1)m_{\mathbf{r}''_p, \mathbf{j}_t},
		\end{align*}
		where $\mathbf{j}_t=(j_1, \cdots, j_{m+2}, t, j_{m}, \cdots, j_k)$;
		\begin{align*}
			F_{i}m_{\mathbf{r}, \mathbf{j}} = & v^{-\sum_{j \in [1, d]}\delta_{i+1, r_j}}\sum\limits_{p \neq j_m+1, r_p=i}v^{2\sum_{j<p}\delta_{i+1, r_{j}}} m_{\mathbf{r}'_p, \mathbf{j}}\\
			& + v^{-\sum_{j \in [1, d]} \delta_{i+1, r_j}}\sum\limits_{p = j_m}v^{2\sum_{j<j_{m+1}}\delta_{i+1, r_{j}}}(v^2-1)m_{\mathbf{r}'_p, \mathbf{j}_{j_m}}\\
			& + v^{-\sum_{j \in [1, d]} \delta_{i+1, r_j}}\sum\limits_{t \in [j_{m+1}+1, j_m-1], r_t=i} v^{2\sum_{j<j_{m+1}}\delta_{i+1, r_{j}}} (v^2-1)m_{\mathbf{r}'_t, \mathbf{j}_t},
		\end{align*}
		where $\mathbf{j}_t=(j_1, \cdots, j_{m+2}, t, j_{m}, \cdots, j_k)$.
	\end{lem}
	\begin{proof}
		The lemma follows from Proposition \ref{prop 4.3.1} and Proposition \ref{prop 4.2.6}.
	\end{proof}
	
	We define the symmetric group $S_d$ actions on $S_{row}$ and $J_{\mathbf{r}}$ as $\sigma\mathbf{r}$$=(r_{\sigma(1)}, \cdots, r_{\sigma(d)})$ and $\sigma\mathbf{j}$$=\{\sigma(j_1), \cdots, \sigma(j_k)\}$.
	\begin{lem}
		Assume $1\leq j \leq d-1$. Let $\mathbf{r}=(r_1, \cdots, r_d)$ and $\mathbf{j}=(j_1, \cdots, j_k)$, the following formulas give a well-defined right $\mathbf{MH}$-action on $\mathbf{MV}$.
		
		\noindent$(a)$ If $j,j+1\notin \mathbf{j}$, then
		\begin{align*}
			m_{\mathbf{r}, \mathbf{j}} \tau_j=\left\{\begin{array}{ll}
				m_{s_j\mathbf{r}, \mathbf{j}} & \text { if } r_j<r_{j+1}, \\
				v^2 m_{\mathbf{r},\mathbf{j}} & \text { if } r_j=r_{j+1}, \\
				\left(v^2-1\right) m_{\mathbf{r}, \mathbf{j}}+v^2 m_{s_j\mathbf{r},\mathbf{j}} & \text { if } r_j>r_{j+1}.
			\end{array}\right.
		\end{align*}
		
		\noindent$(b)$ If $j\in \mathbf{j}$ and $j+1\notin \mathbf{j}$, then
		\begin{align*}
			m_{\mathbf{r}, \mathbf{j}} \tau_j=\left\{\begin{array}{ll}
				m_{s_j\mathbf{r}, s_j\mathbf{j}}+m_{s_j\mathbf{r}, \mathbf{j}\setminus (j)} & \text { if } r_j<r_{j+1}, \\
				m_{\mathbf{r}, s_j\mathbf{j}} & \text { if } r_j=r_{j+1}, \\
				m_{\mathbf{r}, \mathbf{j}\cup \{j+1\}}+m_{s_j\mathbf{r}, s_j\mathbf{j}}+ m_{s_j\mathbf{r}, \mathbf{j}\setminus (j)} & \text { if } r_j>r_{j+1}.
			\end{array}\right.
		\end{align*}
		
		\noindent$(c)$ If $j\notin \mathbf{j}$ and $j+1\in \mathbf{j}$, then
		\begin{align*}
			m_{\mathbf{r}, \mathbf{j}} \tau_j=\left\{\begin{array}{ll}
				2m_{s_j\mathbf{r}, \mathbf{j}\setminus (j+1)}+m_{s_j\mathbf{r}, s_j\mathbf{j}}+m_{s_j\mathbf{r}, \mathbf{j}\cup\{j\}}+m_{\mathbf{r}, \mathbf{j}\setminus (j+1)} & \text { if } r_j<r_{j+1}, \\
				v^2m_{\mathbf{r}, s_j\mathbf{j}}+v^2m_{\mathbf{r}, \mathbf{j}\setminus (j+1)}+(v^2-1)m_{\mathbf{r}, \mathbf{j}} & \text { if } r_j=r_{j+1}, \\
				(v^2-1)m_{\mathbf{r}, \mathbf{j}}+ v^2(2m_{\mathbf{r}, \mathbf{j}\setminus (j+1)}+m_{s_j\mathbf{r}, s_j\mathbf{j}}+m_{s_j\mathbf{r}, \mathbf{j}\setminus (j+1)}) & \text { if } r_j>r_{j+1}.
			\end{array}\right.
		\end{align*}
		
		\noindent$(d)$ If $j, j+1\in \mathbf{j}$, then 
		\begin{align*}
			m_{\mathbf{r}, \mathbf{j}} \tau_j & =
			2(v^2-1)m_{\mathbf{r}, \mathbf{j}}+(2v^2-1)m_{\mathbf{r}, \mathbf{j}\setminus (j+1)}\\
			& + (2v^2-1)m_{s_j\mathbf{r}, \mathbf{j}\setminus (j, j+1)}+(2v^2-1)m_{r_1 \ldots r_{j+1} r_j\ldots r_d, \mathbf{j}\setminus (j)}.
		\end{align*}
		
		\noindent$(e)$ 
		\begin{align*}
			m_{\mathbf{r}, \mathbf{j}} \tau_0=\left\{\begin{array}{ll}
				(v^{2(\sum_{i\leq i_k} \delta_{r_{1} i})}-1)m_{\mathbf{r}, \mathbf{j}}+ \sum_{i'>i_k}\delta_{r_1 i'}m_{\mathbf{r}, \mathbf{j}'},  &  \text{ if } j_k>1,\\
				(v^{2}-2)m_{\mathbf{r}, \mathbf{j}}  &  \text{ if } j_k=1,
			\end{array}\right.
		\end{align*}
		where $\mathbf{j}'=(j_1, \cdots, j_k, 1)$.
	\end{lem}
	\begin{proof}
		This lemma follows from Theorem \ref{mh} and  Proposition \ref{prop 4.4.1}.
	\end{proof}
	Then we conclude the main theorem.
	\begin{thm}\label{thm 2.3.3}
		The left $\mathbf{MU}$-action and the right $\mathbf{MH}$-action on $\mathbf{MV}$ are commuting. 
		They form a double centralizer for $n\geq d$, i.e., 
		$$\End_{\mathbf{MU}}(\mathbf{MV}) \cong \mathbf{MH}, \quad
		\mathbf{MU}\to \End_{\mathbf{MH}}(\mathbf{MV}) \text{ is surjective}.$$ 
	\end{thm}
	\begin{proof}
		The theorem follows from the previous two lemmas, Proposition \ref{prop 4.2.5}, Proposition \ref{prop 4.2.6} and Theorem \ref{thm 6.2.8}.
	\end{proof}
	
	\section{Geometric Setting}
	\subsection{$GL_d(\mathbb{F}_q)$-orbits on mirabolic flag varieties}\ 
	Let $v$ be an indeterminate and $\mathcal{A}=\mathbb{Z}[v,v^{-1}]$. Take $\mathbb{F}_q$ as a finite field with $q$ elements, 
	and $V=\mathbb{F}^d_q$ as a vector space of dimension $d$. 
	For the fixed positive integers $n$ and $d$, define $\mathscr{X}$ as the set of all $n$-step partial flags in $V$ 
	$$
	\mathscr{X}=\{f=(0=V_0\subset V_1\subset \cdots \subset V_{n-1}\subset V_n=V)\}.
	$$
	Similarly, let $\mathscr{Y}$ be the set of all complete flags in $V$
	$$\mathscr{Y}=\{f=(0=V_0\subset V_1\subset \cdots \subset V_{d-1}\subset V_d=V) \ |\ \dim{V_{i}/V_{i-1}}=1 \}.$$ 
	The group $G=GL_d(\mathbb{F}_q)$ acts naturally on $\mathscr{X}$ and $\mathscr{Y}$. Thus $\mathscr{X} \times \mathscr{X}\times V$, $\mathscr{Y} \times \mathscr{Y}\times V$ and $\mathscr{X} \times \mathscr{Y}\times V$ are endowed with the diagonal action of $G$.
	
	Let $\mathcal{MH}=\mathcal{A}_{G}(\mathscr{Y}\times \mathscr{Y}\times V)$, the set of $G$-invariant $\mathcal{A}$-valued functions over $\mathscr{Y}\times \mathscr{Y}\times V$. When $v$ is specialized to $\sqrt{q}$, Rosso defined the convolution product on $\mathcal{MH}$ as follows.
	For any $g, h \in \mathcal{MH}$
	$$
	g\ast h(f, f', \omega)=\sum\limits_{f''\in \mathscr{Y}, \ \mu\in V}g(f, f'', \mu)h(f'', f', \omega-\mu)
	.$$
	Therefore, $\mathcal{MH}$ is an associative $\mathcal{A}$-algebra. The $G$-orbits of $\mathscr{Y} \times \mathscr{Y} \times V$ are parameterized by pairs $(\sigma, \beta)$ where 
	$\sigma$ is an element in the symmetric group $S_d$ and $\beta\in \{(j_1, j_2, \cdots, j_k)\ |\ 1\leq j_i\leq d\text{ for } i\in [1, k], j_1> j_2> \cdots> j_k\}$ such that $\sigma(j_1)< \cdots< \sigma(j_k)$.
	We denote the characteristic functions of these orbits by $\tau_{\sigma, \beta}$.
	\begin{thm}[\cite{R14}]\label{mh}
		The mirabolic Hecke algebra $\mathbf{MH}$ is isomorphic to the $\Q(v)\otimes_{\mathcal{A}}\mathcal{MH}$. 
		The isomorphism is given by $\tau_0 \mapsto \tau_{id, \{1\}}$ and $\tau_i\mapsto \tau_{s_i, \emptyset}$ where $s_i$ denotes the permutation $(i, i+1)$.
	\end{thm} 
	
	\begin{Def}\label{def 3.1.1}
		A decorated matrix is a pair $(A, \Delta)$, where $A=(a_{i j})\in Mat_{n\times n}(\mathbb{N})$ 
		and $\Delta =\{(i_1, j_1), \cdots, (i_k, j_k)\}$ is a set that satisfies
		$$
		1\leq  i_1< \cdots < i_k \leq  n , \ 1\leq j_k< \cdots< j_1 \leq n, 
		$$
		with the additional condition that for all $(i, j)\in \Delta$, the entry $a_{i j}>0$. We denote by $\Xi_{n|n}$ the set of $n\times n$ decorated matrices and 
		$\Xi_{n|n,d}$ is defined as $\{(A, \Delta)\in \Xi_{n|n} | \sum_{i j}a_{i j}=d \}$.
	\end{Def}
	
	We can write down a decorated matrix by circling the entries of the matrix corresponding to $\Delta$.
	\begin{example}
		Let
		$
		A=
		\begin{pmatrix}
			1 & 4 & 0 \\
			3 & 7 & 1\\
			4 & 5 & 2\\
		\end{pmatrix}
		$
		and $\Delta=\{(2, 2), (3, 1)\}$, we can write $(A, \Delta)$ as
		$$
		(A, \Delta)=
		\begin{pmatrix}
			1 & 4 & 0 \\
			3 & \enc{7} & 1\\
			\enc{4} & 5 & 2\\
		\end{pmatrix}
		.$$
	\end{example}
	
	By \cite{MWZ}, we have the bijection between the set of $G$-orbits of $\mathscr{X} \times \mathscr{X}\times V$ and $\Xi_{n|n,d}$, i.e., 
	$$
	G \backslash  \mathscr{X}\times \mathscr{X}\times V \longleftrightarrow \Xi_{n|n,d}
	.$$
	The orbit corresponding to $(A, \Delta)$ is denoted by $\mathcal{O}_{A, \Delta}$. We define $\mathcal{MS}_{n,d}=\mathcal{A}_{G}(\mathscr{X}\times \mathscr{X}\times V)$ to be the space of $G$-invariant $\mathcal{A}$-valued functions on $\mathscr{X}\times \mathscr{X}\times V$. The convolution product makes $\mathcal{MS}_{n,d}$ an associative algebra over $\mathcal{A}$. Its basis is given by $\{e_{A, \Delta}\ |\ (A, \Delta)\in \Xi_{n|n,d} \}$, where $e_{A, \Delta}$ is the characteristic function of the orbit $\mathcal{O}_{A, \Delta}$. The unit element of $\mathcal{MS}_{n,d}$ is
	$$
	\sum\limits_{D}e_{D, \emptyset},
	$$
	where $D$ runs over all diagonal matrices in $Mat_{n\times n}(\mathbb{N})$ such that $(D,\emptyset)\in \Xi_{n|n,d}$.
	
	\begin{Def}
		The algebra $\mathcal{MS}_{n,d}$ is called mirabolic quantum Schur algebra. 
	\end{Def}
	Let $\mathcal{MV}= \mathcal{A}_{G}(\mathscr{X}\times \mathscr{Y}\times V)$ be the space of $G$-invariant $\mathcal{A}$-valued functions on $\mathscr{X}\times \mathscr{Y}\times V$. The space $\mathcal{MV}$ is endowed with a left $\mathcal{MS}_{n,d}$-action and a right $\mathcal{MH}$-action by convolution.
	
	Let
	$$\Xi^1 =\{(A,\Delta)\in \Xi_{n|m,d}\ |\ \text{ for any } j\in [1,d], \sum_{i\in [1,n]}a_{ij}=1\}.$$ 
	By \cite{R18}, there is a bijection 
	\begin{align*}
		G\backslash \mathscr{X}\times \mathscr{Y}\times V \leftrightarrow \Xi^1.
	\end{align*}
	For any $(A,\Delta)\in \Xi^{1}$ where $\Delta=\{(i_1,j_1), \cdots, (i_k,j_k)\}$, we associate $(A,\Delta)$ a pair of sequence
	$(\mathbf{r}=(r_1,\cdots, r_d),\mathbf{j})$, where $r_c$ is the unique number in $[1, n]$ such that $a_{r_c c}=1$ for any $c\in [1, d]$ and $ \mathbf{j}=(j_1, \cdots, j_k)$. Then Rosso showed that there is a bijection 
	$$
	\Xi^{1}\leftrightarrow \{(\mathbf{r},\mathbf{j})\ |\ \mathbf{r}=(r_1, r_2, \ldots ,r_d), \mathbf{j}\in J_{\mathbf{r}}\}.
	$$
	Let $e_{\mathbf{r}, \mathbf{j}}=e_{A,\Delta}$ be the characteristic function of the $\mathcal{O}_{A, \Delta}$. Then $\{e_{\mathbf{r},\mathbf{j}}| \mathbf{r}=(r_1,\cdots, r_d), \mathbf{j}\in J_{\mathbf{r}}\}$ form a basis of $\mathcal{MV}$.
	\begin{rem}
		By the two bijections above, we have isomorphism $\mathbf{MV} \cong \mathbb{Q}(v)\otimes_{\mathcal{A}}\mathcal{MV}$.
	\end{rem}
	\subsection{Dimension of orbits}
	In this subsection, we assume the ground field is an algebraic closure $\overline{\mathbb{F}}_q$ of $\mathbb{F}_q$ whenever we refer to the dimension of a $G$-orbit or its stabilizer.
	We first define a partial order on $\Delta$.
	\begin{Def}
		We say $\Delta' \leq \Delta$ if for any $(i, j)\in \Delta'$, there is a pair $(k,l)\in \Delta$ such that $i\leq k$ and $j\leq l$.
		We say $\Delta' < \Delta$ if $\Delta'\neq \Delta$.
	\end{Def}
	\begin{lem}[\cite{MWZ}]
		Let $(f, f', f'')\in \mathscr{X}\times \mathscr{X}\times \mathbb{P}(V)$, where $\mathbb{P}(V)$ is the projective space of $V$. 
		The dimension of $\mathcal{O}_{(f,f',f'')}$, the $G$-orbit of $(f,f',f'')$, is 
		$$\sum\limits_{i<k\  or\  j< l}a_{i j}a_{kl}+\sum\limits_{\{(i, j)\}\leq \Delta}  a_{i j}-1.$$
	\end{lem}
	\begin{cor}
		The dimension of the orbit $\mathcal{O}_{A,\Delta}$ denoted by $d(A,\Delta)$ is 
		$$\sum\limits_{i<k\  or\  j< l}a_{i j}a_{kl}+\sum\limits_{\{(i, j)\}\leq \Delta}  a_{i j}.$$
	\end{cor}
	\begin{proof}
		When $\Delta=\emptyset$, $d(A,\Delta)=\sum\limits_{i<k\  or\  j< l}a_{i j}a_{kl}$ has been shown in the classical work\cite{BLM}. For $\Delta \neq \emptyset$, let $(f, f', \nu)\in \mathcal{O}_{A,\Delta}$. We consider $\overline{\mathbb{F}}_q^{*}$ acts on $\mathcal{O}_{A,\Delta}$ by scalar multiplication. It is easy to see that this $\overline{\mathbb{F}}_q^{*}$-action is free. The orbit of $(f, f', \nu)$ under this action is just $(f, f', \overline{\mathbb{F}}_q^{*} \nu)$. 
		Then by the previous lemma we have the following identity.
		\begin{align*}
			\dim {\mathcal{O}_{A,\Delta}}=\dim{\mathcal{O}_{(f,f',\overline{\mathbb{F}}_q^{*} \nu)}}+1. 
		\end{align*}
		Thus the corollary follows.
	\end{proof}
	
	Let $r(A, \Delta)$ denote the dimension of the image of $\mathcal{O}_{A, \Delta}$ 
	under the first projection $pr_1: \mathscr{X}\times \mathscr{X}\times V\to \mathscr{X}$. Let $B$ be a diagonal matrix with $b_{ii}=\sum_{j}a_{ij}$.
	We can see that $r(A, \Delta)=\dim \mathcal{O}_{B,\emptyset}=\sum\limits_{i<k}b_{ii}b_{kk}=\sum\limits_{i<k}a_{i j}a_{k l}$ and 
	\begin{equation*}
		d(A, \Delta)-r(A, \Delta)=
		\sum\limits_{i\geq k, j< l}a_{i j}a_{k l}+\sum\limits_{\{(i, j)\}\leq \Delta}a_{i j}.
	\end{equation*}
	
	\section{Calculus of the algebra $\mathcal{MS}_{n,d}$}
	\subsection{Multiplication formulas}\ 
	For any two integers $N, t$, let $$[N, t]_v=\prod_{i\in [1, t]}\frac{v^{2(N-i+1)}-1}{v^{2i}-1}.$$
	For any $(A,\Delta)\in \Xi_{n|n,d}$, we define
	$$
	\ro(A)=(\sum\limits_{1\leq j \leq n}a_{1j}, \cdots, \sum\limits_{1\leq j \leq n}a_{nj}),
	\ 
	\co(A)=(\sum\limits_{1\leq i \leq n}a_{i1}, \cdots, \sum\limits_{1\leq i \leq n}a_{in}).
	$$
	
	\begin{prop}\label{prop 4.1.1}
		Assume that $h\in [1, n-1]$. For $(B,\emptyset)$, $(A, \Delta)\in \Xi_{n|n,d}$ satisfied $\ro(A) = \co(B)$ where $B - E_{h, h+1}$ is a diagonal matrix
		and $\Delta = \{(i_1, j_1), \cdots, (i_k, j_k)\}$.
		
		\noindent$(a)$ If for any $t\in [1, k]$, we have $i_t \neq h, h+1$, then 
		$$e_{B, \emptyset}\ast e_{A, \Delta} = \sum\limits_{p\in[1, n], a_{h+1, p} \geq 1} v^{2\sum_{j>p}a_{h, j}}\frac{v^{2(a_{h, p}+1)}-1}{v^2 - 1}e_{A+E_{h, p}-E_{h+1, p}, \Delta}.$$
		
		\noindent$(b)$ If there exists $m$ such that $i_m = h$ and $i_{m+1} \neq h+1$, then
		\begin{eq}
			e_{B, \emptyset}\ast e_{A, \Delta}  = & \sum\limits_{p \leq j_{m+1}, a_{h+1, p}\geq 1}v^{2\sum_{j> p}a_{h, j}}\frac{v^{2(a_{h, p}+1)}-1}{v^2 - 1}e_{A+E_{h, p}-E_{h+1, p}, \Delta}\\
			& +\sum\limits_{p \in[j_{m+1}+1, j_m - 1], a_{h+1, p} \geq 1} v^{2(\sum_{j> p}a_{h, j}-1)}\frac{v^{2(a_{h, p}+1)}-1}{v^2 - 1}e_{A+E_{h, p}-E_{h+1, p}, \Delta} \\
			& + \sum\limits_{p = j_m, a_{h+1, p} \geq 1} v^{2\sum_{j> p}a_{h, j}}\frac{v^{2a_{h, p}}-1}{v^2 - 1}e_{A+E_{h, p}-E_{h+1, p}, \Delta} \\
			& + \sum\limits_{p \in[j_m+1, n], a_{h+1, p} \geq 1} v^{2\sum_{j>p}a_{h, j}}\frac{v^{2(a_{h, p}+1)}-1}{v^2 - 1}e_{A+E_{h, p}-E_{h+1, p}, \Delta}.\\
		\end{eq}
		
		\noindent$(c)$ If there exists $m$ such that $i_{m-1} \neq h$ and $i_{m} = h+1$, then
		\begin{align*}
			e_{B, \emptyset}\ast e_{A, \Delta} = & \sum\limits_{p \in[1, n], a_{h+1, p} \geq 1} v^{2\sum_{j>p}a_{h, j}}\frac{v^{2(a_{h, p}+1)}-1}{v^2 - 1}e_{A+E_{h, p}-E_{h+1, p}, \Delta} \\
			& + \sum\limits_{p=j_m, a_{h+1, p} \geq 1} v^{2\sum_{j\geq p}a_{h, j}}e_{A+E_{h, p}-E_{h+1, p}, \Delta_c}\\
			& +\sum\limits_{p= j_m, a_{h+1, p} \geq 1, t \in [j_{m+1}+1, j_{m}-1]} v^{2\sum_{j\geq p}a_{h, j}}e_{A+E_{h, p}-E_{h+1, p}, \Delta_t}, 
		\end{align*}
		where
		\begin{eq}
			& \Delta_c = \{(i_1, j_1), \cdots, (i_{m-1}, j_{m-1}), (h, j_m), (i_{m+1}, j_{m+1}), \cdots, (i_k, j_k)\},\\
			& \Delta_t = \{(i_1, j_1), \cdots, (i_{m-1}, j_{m-1}), (h, j_m), (h+1, t), (i_{m+1}, j_{m+1}), \cdots, (i_k, j_k)\}.
		\end{eq}
		
		\noindent$(d)$ If there exists $m$ such that $i_m=  h$ and $i_{m+1} = h+1$, then
		\begin{align*}
			e_{B, \emptyset}\ast e_{A, \Delta}  = & \sum\limits_{p \in[1, j_{m+1}]\cup[j_{m}+1, n], a_{h+1, p} \geq 1} v^{2\sum_{j>p}a_{h, j}}\frac{v^{2(a_{h, p}+1)}-1}{v^2 - 1}e_{A+E_{h, p}-E_{h+1, p}, \Delta} \\
			& + \sum\limits_{p \in[j_{m+1} + 1, j_{m}-1], a_{h+1, p} \geq 1} v^{2(\sum_{j>p}a_{h, j} - 1)}\frac{v^{2(a_{h, p}+1)}-1}{v^2 - 1}e_{A+E_{h, p}-E_{h+1, p}, \Delta} \\
			& +  \sum\limits_{p=j_m, a_{h+1, p} \geq 1} v^{2\sum_{j>p}a_{h, j}}\frac{v^{2a_{h, p}}-1}{v^2 - 1}e_{A+E_{h, p}-E_{h+1, p}, \Delta} \\
			& + \sum\limits_{p=j_{m+1}, a_{h+1, p} \geq 1} (v^{2\sum_{j\geq p}a_{h, j}} - v^{2(\sum_{j > p}a_{h, j}- 1)})e_{A+E_{h, p}-E_{h+1, p}, \Delta_s}\\
			& +\sum\limits_{\substack{p=j_{m+1}, a_{h+1, p} \geq 1,\\ t \in [j_{m+2}+1, j_{m+1}-1]}} (v^{2\sum_{j\geq p}a_{h, j}} - v^{2(\sum_{j > p}a_{h, j}- 1)})e_{A+E_{h, p}-E_{h+1, p}, \Delta_t},
		\end{align*}
		where
		\begin{eq}
			& \Delta_s = \{(i_1, j_1), \cdots, (i_{m-1}, j_{m-1}), (h, j_m), (i_{m+2}, j_{m+2}), \cdots, (i_k, j_k)\}, \\
			& \Delta_t = \{(i_1, j_1), \cdots, (i_{m-1}, j_{m-1}), (h, j_m), (h+1, t), (i_{m+2}, j_{m+2}), \cdots, (i_k, j_k)\}.
		\end{eq}
	\end{prop}
	
	\begin{proof}\
		We show the results at the specialization of $v$ to $\sqrt{q}$. By an argument similar to \cite[Proposition 3.1]{R18}, we have the following identity.
		\begin{align*}
			e_{B, \emptyset}\ast e_{A, \Delta}=\sum\limits_{p\in [1,n], a_{h+,p}>0}\sharp Z_p e_{A+E_{h,p}-E_{h+1,p},\Delta'},
		\end{align*}
		where $Z_p$ is the set of all subspaces $U\subset V$ satisfying the following conditions.
		\begin{enumerate}
			\item $(f,f',w)$ is a fixed triple in $\mathscr{X}$ lying in the orbit $\mathcal{O}_{A+E_{h,p}-E_{h+1,p},\Delta'}$,
			\item $U\subset V_h$ and $\dim{U}=\dim{V_h}-1$,
			\item $ V_h\cap V'_j=U\cap V'_j (j<p), \ V_h\cap V'_{j}\neq U\cap V'_j (j\geq p)$,
			\item $\omega\in \sum_{(i, j)\in \Delta} V_{i}\cap V'_{j}\setminus V_{i-1}\cap V_{j}+V_{i}\cap V'_{j-1}$.
		\end{enumerate}
		So we only need to compute the number $\sharp Z_p$.
		
		For any $p\in [1, n]$ and $U\in Z_p$, there exists a vector $y_U\in V_h$ such that $V_h=U\oplus \mathbb{F}_q y_U$. 
		Let $\omega=\sum_{(i, j)\in \Delta'} \omega_{i j}$ where $\omega_{i j}\in V_i\cap V'_{j}\setminus ((V_{i-1}\cap V'_{j})+(V_{i}\cap V'_{j-1}))$.
		
		$(a)$ In this case, $i_t \neq h, h+1$ for any $(i_t, j_t)\in \Delta$. We see $\Delta'=\Delta$, then for any $p\in [1, n]$,  
		\begin{equation*}
			\begin{aligned}
				\sharp Z_p & =  \sharp \{U|V_{h-1}+(V_{h}\cap V'_{p-1})\subset U \subset V_{h}\}
				-\sharp \{U|V_{h-1}+(V_{h}\cap V'_{p})\subset U \subset V_{h}\}\\
				& =  \frac{q^{\sum_{j\geq p}a'_{hj}}-1}{q-1}-\frac{q^{\sum_{j> p}a'_{hj}}-1}{q-1}\\
				& =  q^{\sum_{j>p}a_{hj}} \frac{q^{a_{hp}+1}-1}{q-1},
			\end{aligned}
		\end{equation*}
		so $(a)$ follows.
		
		$(b)$ In this case, there exists $(i_m, j_m)\in \Delta $ such that $i_m=h$ and $i_{m+1}\neq h+1$. For any $U$ such that $V_{h-1}\subset U\subset V_h$ and $\dim U+1=\dim V_h$, we have $U\in Z$ if and only if 
		$\omega\in \sum_{(i, j)\in \Delta, i\neq h}, V_{i}\cap V'_{j}+U\cap V'_{j_m}$. 
		We have $\Delta'=\Delta$ due to $U\cap V'_{j_m}\subset V_h\cap V'_{j_m}$. 
		
		For $p>j_m$, we have $U\cap V'_{j_m}=V_{h}\cap V'_{j_m}$, which implies $\sharp Z_p$ is the same as $(a)$.
		
		For $j_{m+1} < p \leq j_m$, we have $U\cap V'_{j_m}\neq V_{h}\cap V'_{j_m}$, then $U\in Z_p$ if $\omega_{h j_m}\in U$. If $p=j_m$, then
		\begin{align*}
			\sharp Z_{j_m}= & \sharp {\{U|V_{h-1}+(V_{h}\cap V'_{j_m-1})\subset U \subset V_{h}, \omega_{h j_m}\in U\}}\\
			& -\sharp {\{U|V_{h-1}+(V_{h}\cap V'_{j_m})\subset U \subset V_{h}, \omega_{h j_m}\in U\}}\\
			= & \frac{q^{\sum_{j\geq j_m}a_{h, j}}-1}{q-1}-\frac{q^{\sum_{j> j_m}a_{h, j}}-1}{q-1}\\
			= & q^{\sum_{j>j_m}a_{h, j}} \frac{q^{a_{h, j_m}}-1}{q-1},
		\end{align*}
		and if $j_{m+1}<p<j_m$, then
		\begin{align*}
			\sharp Z_p  = & \sharp {\{U|V_{h-1}+(V_{h}\cap V'_{p-1})\subset U \subset V_{h}, \omega_{h j_m}\in U\}}\\
			& -\sharp {\{U|V_{h-1}+(V_{h}\cap V'_{p})\subset U \subset V_{h}, \omega_{h j_m} \in U\}}\\
			= & \frac{q^{\sum_{j\geq p}a_{h, j}}}{q-1}-\frac{q^{\sum_{j> p}a_{h, j}-1}-1}{q-1}\\
			= & q^{\sum_{j>p}a_{h, j}-1} \frac{q^{a_{hp}+1}-1}{q-1}.
		\end{align*}
		
		For $p\leq j_{m+1}$ and $U\in Z_p$, we have $y_U\in V_{i_{m+1}}\cap V'_{j_{m+1}}$.
		Since $\omega_{h j_m}\in \mathbb{F}_q y_U + U\cap V'_{j_m}$, we see  
		\begin{align*}
			\omega \in \sum\limits_{\substack{(i, j)\in \Delta\\i \neq h}} V_i\cap V'_{j}+U \cap V'_{j_m}.
		\end{align*}
		Therefore, we conclude $\sharp Z_p$ is the same as $(a)$.
		Then $(b)$ follows.
		
		$(c)$ In this case, there exists $(i_m, j_m)\in \Delta$ such that $i_m=h+1$ and $i_{m-1}\neq h$. 
		For $p\neq j_m$, we find $\Delta'=\Delta$. For $p=j_m$, since $(h+1,j_m)\in \Delta$ and 
		$$V_{h+1}\cap V'_{j_m} \setminus V_h\cap V'_{j_m}+V_{h+1}\cap V'_{j_m-1}\subset V_{h+1}\cap V'_{j_m} \setminus U\cap V'_{j_m}+V_{h+1}\cap V'_{j_m-1},$$
		we have the following choices of $\Delta'$m
		\begin{align*}
			\Delta' & =\Delta,\\
			\Delta' & =\Delta_c = \{(i_1, j_1), \cdots, (h, j_m), \cdots, (i_k, j_k)\}, \\
			\Delta' & =\Delta_t = \Delta_c\cup \{(h+1, t)\}, t \in [j_{m+1}+1, j_{m}-1].
		\end{align*}
		
		For $\Delta' = \Delta$, 
		We see $(U\cap V'_{j_m})+(V_{h+1}\cap V'_{j_{m}-1}) \subset (V_h \cap V'_{j_m})+(V_{h+1}\cap V'_{j_{m}-1})$. Therefore, for any $p\in [1, n]$, we conclude $\sharp Z_p$ is the same as $(a)$.
		
		For $\Delta'=\Delta_c$ and $\Delta'=\Delta_t$, 
		let $\omega_{h j_m}=\omega_1+\omega_2$ where $\omega_1 \in U\cap V'_{j_m}$ and $\omega_{2}\in V_{h}\cap V'_{j_{m}}$. The subspace $U\in Z_{j_m}$ 
		if and only if $\omega_2\notin U\cap V'_{j_m}$.
		It follows that
		\begin{align*}
			\sharp {Z_p} = & \sharp {\{U|V_{h-1}+(V_{h}\cap V'_{p-1})\subset U \subset V_{h}\}}-\sharp {\{U|V_{h-1}+(V_{h}\cap V'_{p})\subset U \subset V_{h}\}}\\
			& - \sharp {\{U|V_{h-1}+(V_{h}\cap V'_{p-1})\subset U \subset V_{h}, \omega_{2}\in U\}}\\
			& + \sharp {\{U|V_{h-1}+(V_{h}\cap V'_{p})\subset U \subset V_{h}, \omega_2\in U\}}\\
			= & \frac{q^{\sum_{j\geq p}a_{h, j}+1}-1}{q-1}-\frac{q^{\sum_{j\geq p}a_{h, j}}-1}{q-1}\\
			= & q^{\sum_{j\geq p}a_{h, j}}, 
		\end{align*}
		then $(c)$ follows.
		
		$(d)$ In this case, there exists $(i_m, j_m), (i_{m+1}, j_{m+1})\in \Delta$ such that $i_m=h, i_{m+1}=h+1$. 
		The $\Delta'$ is similar to the case $(c)$, i.e., 
		for $p\neq j_{m+1}$, we have $\Delta'=\Delta$ and for $p=j_{m+1}$ there are other $\Delta'$, 
		\begin{align*}
			\Delta'& =\Delta,\\
			\Delta' & =\Delta_s = \Delta\setminus \{(h+1, j_{m+1})\}, \\
			\Delta' & =\Delta_t= \{(i_1, j_1), \cdots, (h, j_m), (h+1, t), \cdots, (i_k, j_k)\}, t \in [j_{m+2}+1, j_{m+1}-1].
		\end{align*}
		
		For $\Delta'=\Delta$ and any $U$ such that $V_{h-1}\subset U\subset V_h$ and $\dim U+1=\dim V_h$, 
		we know $U\cap V'_{j_{m+1}}\subset V_h\cap V'_{j_{m+1}}$, so
		$$\omega_{h+1, j_{m+1}}\in V_{h+1}\cap V'_{j_{m+1}}\setminus{(U\cap V'_{j_{m+1}})+(V_{h+1}\cap V'_{j_{m+1}-1}}).$$ 
		Then $U\in Z$ if $\omega_{h j_m} \in U\cap V'_{j_m}$, hence the calculation is the same as $(b)$.
		
		For $\Delta'=\Delta_s$ and $\Delta=\Delta_t$, we have a decomposition $\omega_{h j_m}=\omega_{1}+\omega_{2}$, where $\omega_{1}\in U\cap V'_{j_{m}}$ and $\omega_{2}\in V_{h}\cap V'_{j_{m}}$. Since $p=j_{m+1}$, we see $y_U\in V_{h+1}\cap V'_{j_{m+1}}\setminus U\cap V'_{j_{m+1}}$. Then $U\in Z$ if $\omega_{2}\notin U$. So 
		\begin{align*}
			\sharp  Z_p = & \sharp  \{U|V_{h-1}+(V_{h}\cap V'_{p-1})\subset U \subset V_{h}\}-\sharp \{U|V_{h-1}+(V_{h}\cap V'_{p})\subset U \subset V_{h}\}\\
			& - \sharp  \{U|V_{h-1}+(V_{h}\cap V'_{p-1})\subset U \subset V_{h}, \omega_2\in U\}\\
			& + \sharp  \{U|V_{h-1}+(V_{h}\cap V'_{p})\subset U \subset V_{h}, \omega_2 \in U\}\\
			= & \frac{q^{\sum_{j\geq p}a_{h, j}+1}-q^{\sum_{j>p}a_{h, j}}}{q-1}-\frac{q^{\sum_{j\geq p}a_{h, j}}-q^{\sum_{j>p}a_{h, j}-1}}{q-1}\\
			= & q^{\sum_{j\geq p}a_{h, j}}-q^{\sum_{j>p}a_{h, j}-1},
		\end{align*}
		then $(d)$ follows.
	\end{proof}
	\begin{rem}\label{rem 4.1.2}
		When $n>2$, we need to consider more complicated $\Delta$, so the multiplication formulas are more complicated than those for $n=2$. For example, let 
		$$(A,\Delta)=
		\begin{pmatrix}
			0 & 0 & \enc{1}\\
			\enc{2} & 1 & 1 \\
			1 & 1 & 0 \\
		\end{pmatrix},
		(A,\Delta')=
		\begin{pmatrix}
			0 & 0 & 1\\
			2 & 1 & \enc{1} \\
			1 & \enc{1} & 0 \\
		\end{pmatrix},
		(B,\emptyset)=    
		\begin{pmatrix}
			1 & 1 &0 \\
			0 & 3 &0 \\
			0 & 0 & 2 \\
		\end{pmatrix}.$$
		\begin{align*}
			e_{B,\emptyset}\ast e_{A,\Delta}=& v^2e_{A+E_{1,1}-E_{2,1},\Delta}+e_{A+E_{1,2}-E_{2,2},\Delta}\\
			&+e_{A+E_{1,3}-E_{2,3},\Delta} +v^2e_{A+E_{1,1}-E_{2,1},\Delta\setminus\{(2,1)\}}.
		\end{align*}
		\begin{align*}
			e_{B,\emptyset}\ast e_{A,\Delta}=& v^2e_{A+E_{1,1}-E_{2,1},\Delta}+v^2e_{A+E_{1,2}-E_{2,2},\Delta}\\
			&+(v^2+1)e_{A+E_{1,3}-E_{2,3},\Delta} +v^2e_{A+E_{1,1}-E_{2,1},\Delta''},
		\end{align*}
		where $\Delta''=\{(1,3).(3,2)\}$. We shall see this phenomenon a few times in the later calculation.
	\end{rem}
	
	\begin{prop}\label{prop 4.1.2} 
		Assume $h\in [1,n-1]$. For $(C,\emptyset), (A,\Delta)\in \Xi_{n|n,d}$ satisfied $\ro(A) = \co(C)$ where $C - E_{h+1, h}$ is a diagonal matrix, 
		and $\Delta = \{(i_1, j_1), \cdots, (i_k, j_k)\}$.
		
		\noindent$(a)$ If for any $t$, we have $i_t \neq h, h+1$, then
		$$e_{C, \emptyset}\ast e_{A, \Delta} = \sum\limits_{p\in[1, n], a_{h, p} \geq 1} v^{2\sum_{j<p}a_{h+1, j}}\frac{v^{2(a_{h+1, p}+1)}-1}{v^2 - 1}e_{A - E_{h, p}+E_{h+1, p}, \Delta}.$$
		
		\noindent$(b)$ If there exists  $m$ such that $i_{m} = h$ and $i_{m+1} \neq h$, then
		\begin{align*}
			e_{C, \emptyset}\ast e_{A, \Delta} = & \sum\limits_{p \in[1, n], a_{h, p} \geq 1} v^{2\sum_{j<p}a_{h+1, j}}\frac{v^{2(a_{h+1, p}+1)}-1}{v^2 - 1}e_{A-E_{h, p}+E_{h+1, p}, \Delta} \\  
			& +\sum\limits_{t \in [j_{m+1}+1, j_m-1]} v^{2\sum_{j \leq j_{m+1}}a_{h+1,j}}e_{A-E_{h, t}+E_{h+1, t}, \Delta_t}\\
			& + v^{2\sum_{j \leq j_{m+1}}a_{h+1,j}}e_{A-E_{h, j_m}+E_{h+1, j_m}, \Delta_b},
		\end{align*}
		where
		\begin{eq}
			\Delta_b& = \{(i_1, j_1), \cdots, (i_{m-1}, j_{m-1}), (h+1, j_m), (i_{m+1}, j_{m+1}), \cdots, (i_k, j_k)\}, \\
			\Delta_t &= \{(i_1, j_1), \cdots, (i_{m-1}, j_{m-1}), (h, j_m), (h+1, t), (i_{m+1}, j_{m+1}), \cdots, (i_k, j_k)\}.
		\end{eq}
		
		\noindent$(c)$ If there exists  $m$ such that $i_{m-1} \neq h$ and $i_{m} = h+1$, then
		\begin{eq}
			e_{C, \emptyset}\ast e_{A, \Delta}  = & \sum\limits_{p \neq j_m, a_{h, p} \geq 1} v^{2\sum_{j<p}a_{h+1, j}}\frac{v^{2(a_{h+1, p}+1)}-1}{v^2 - 1}e_{A-E_{h, p}+E_{h+1, p}, \Delta} \\
			& + \sum\limits_{p = j_m, a_{h, p} \geq 1} v^{2(\sum_{j< p}a_{h + 1, j}+1)}\frac{v^{2(a_{h+1, p})}-1}{v^2 - 1}e_{A - E_{h, p} + E_{h+1, p}, \Delta}.
		\end{eq}
		
		\noindent$(d)$ If there exists $m$ such that $i_m= h$ and $i_{m+1} = h+1$, then
		\begin{align*}
			e_{C, \emptyset}\ast e_{A, \Delta}  = & \sum\limits_{p \in[1, j_{m+1}-1]\cup[j_{m+1}+1, n], a_{h, p} \geq 1} v^{2\sum_{j<p}a_{h+1, j}}\frac{v^{2(a_{h+1, p}+1)}-1}{v^2 - 1}e_{A-E_{h, p}+E_{h+1, p}, \Delta} \\
			& +  \sum\limits_{p=j_{m+1}, a_{h, p} \geq 1} v^{2(\sum_{j<p}a_{h+1, j} + 1)}\frac{v^{2a_{h+1, p}}-1}{v^2 - 1}e_{A-E_{h, p}+E_{h+1, p}, \Delta} \\
			& + \sum\limits_{p=j_{m}, a_{h, p} \geq 1} (v^{2\sum_{j \leq j_{m+1}}a_{h+1, j}} - v^{2\sum_{j < j_{m+1}}a_{h+1, j}})e_{A-E_{h, p}+E_{h+1, p}, \Delta_s}\\
			& + \sum\limits_{p=t\in[j_{m+1} + 1, j_m - 1], a_{h, p} \geq 1}(v^{2\sum_{j \leq j_{m+1}}a_{h+1, j}} - v^{2\sum_{j < j_{m+1}}a_{h+1, j}})e_{A-E_{h, p}+E_{h+1, p}, \Delta_t},
		\end{align*}
		where
		\begin{eq}
			\Delta_s &= \{(i_1, j_1), \cdots, (i_{m-1}, j_{m-1}), (h+1, j_m), (i_{m+2}, j_{m+2}), \cdots, (i_k, j_k)\}, \\
			\Delta_t &=\{(i_1, j_1), \cdots, (i_{m-1}, j_{m-1}), (h, j_m), (h+1, t), (i_{m+2}, j_{m+2}), \cdots, (i_k, j_k)\}.\\
		\end{eq}
	\end{prop}
	
	\begin{proof}\ 
		The proof is similar to Proposition \ref{prop 4.1.1}.
	\end{proof}
	Let $(A,\Delta), (A,\Delta')\in \Xi_{n|n,d}$, for any $(i_p,j_p)\in \Delta$, we define 
	\begin{align*}
		\theta_{A}(i_p, j_p)= \left\{ \begin{array}{ll}
			v^{2\sum_{i_{p-1}<i< i_p, j\leq j_p}a_{i, j}}(v^{2\sum_{j\leq j_p}a_{i_p, j}}-2)  & \text{ if } (i_p, j_p)\in \Delta', \\
			v^{2(\sum_{i_{p-1}<i< i_p,\ j\leq j_p}a_{i, j}+\sum_{j<j_p} a_{i_p, j})}(v^{2a_{i_p, j_p}}-1)  & \text{ otherwise }.\\
		\end{array}
		\right.
	\end{align*}
	Then we have the following proposition.
	\begin{prop}\label{prop 4.1.3}
		Assume $h\in [1, n]$. For $(D,\{(h, h)\}), (A, \Delta)\in \Xi_{n|n,d}$ such that $\ro(A)=\co(D)$, where $D$ is a diagonal matrix, and $\Delta = \{(i_1, j_1), \cdots, (i_k, j_k)\}$.
		
		\noindent$(a)$ If there is $(i_m, j_m)\in \Delta$ such that $i_{m-1}<h<i_{m}$, then
		\begin{align*}
			e_{D,\{(h, h)\}} \ast e_{A, \Delta}= \sum_{\Delta'}\prod^{m-1}_{p=1}\theta_{A}(i_p, j_p)\rho(A,\Delta') e_{A, \Delta'},
		\end{align*}
		where $\Delta'$ runs over the set
		$$
		\left\{ \Delta'=\{(i'_1, j'_1), \cdots, (i'_{l}, j'_{l})\}\ \bigg| \ \begin{array}{l}
			(1)  (i_{t}, j_{t})\in \Delta'\text{ for all } t \in [m, k] \\
			(2)  \text{ if } (i', j')\in \Delta'\setminus \Delta, \text{ then } i'\leq h
		\end{array}\right\},$$
		\begin{align*}
			\rho(A,\Delta')= \left\{ \begin{array}{ll}
				v^{2\sum_{i_{m-1}<i\leq h, j\leq j_m} a_{i, j}}  & \text{ if there exist } (h, j)\in \Delta', \\
				v^{2\sum_{i_{m-1}<i<h, j\leq j_m} a_{i, j}}(v^{2\sum_{j\leq j_m}a_{h, j}}-1)  & \text{ otherwise }.\\
			\end{array}
			\right.
		\end{align*}
		
		\noindent$(b)$ If there is $i_m=h$, we have
		\begin{align*}
			e_{D,\{(h, h)\}} \ast e_{A, \Delta}= \sum_{\Delta'} \prod^{m}_{k=1}\theta_{A}(i_p, j_p)e_{A, \Delta'},
		\end{align*}
		where $\Delta'$ runs over the set
		$$
		\left\{ \Delta'=\{(i'_1, j'_1), \cdots, (i'_{l}, j'_{l})\}\ \bigg| \ \begin{array}{l}
			(1)  (i_{t}, j_{t})\in \Delta'\text{ for all } t \in [m+1, k] \\
			(2)  \text{ if } (i', j')\in \Delta'\setminus \Delta, \text{ then } i'\leq h
		\end{array}\right\}.$$
		
		\noindent$(c)$ If $i_k<h$, then
		\begin{eq}
			e_{D, \{(h, h)\}} \ast e_{A, \Delta}= \sum_{\Delta'}\prod^{k}_{p=1}\theta_{A}(i_p, j_p) e_{A, \Delta'},
		\end{eq}
		where $\Delta'$ runs over the set $\{\Delta'=\{(i'_1, j'_1), \cdots, (i'_{l}, j'_{l})\}\ |\ i'_{l}\leq h\}$. 
		
		\noindent$(d)$ If $\Delta=\emptyset$, then
		\begin{align*}
			e_{D,\{(h, h)\}} \ast e_{A, \emptyset}=\sum_{t\in [1,n], a_{h, t}>0}e_{A,\{(h,t)\}}.
		\end{align*}
	\end{prop} 
	
	\begin{proof}
		The proof is similar to Proposition \ref{prop 4.1.1}. 
	\end{proof}
	
	Let $[A]_{\Delta}=v^{-d(A, \Delta)+r(A, \Delta)}e_{A, \Delta}$ and we define a bar involution on $\mathcal{A}$ by $\overline{v}=v^{-1}$. We have the following propositions.
	\begin{prop}\label{prop 4.1.4}
		Assume $h\in [1, n-1]$. For $(B, \emptyset), (A,\Delta) \in \Xi_{n|n,d}$ satisfied $\ro(A) = \co(B)$, where $B - E_{h, h+1}$ is a diagonal matrix and $\Delta = \{(i_1, j_1), \cdots, (i_k, j_k)\}$, let $\beta_{A}(p)=\sum_{j\geq p}a_{h, j}-\sum_{j>p}a_{h+1, j}$.
		
		\noindent$(a)$ If for any $t$, we have $i_t \neq h, h+1$, then
		$$[B]_{\emptyset}\ast [A]_\Delta = \sum\limits_{p\in[1, n], a_{h+1, p} \geq 1}
		v^{\beta_{A}(p)}\overline{[a_{h, p}+1, 1]}_v[A+E_{h, p}-E_{h+1, p}]_{\Delta}.
		$$
		
		\noindent$(b)$ If there exists $m$ such that $i_m = h$ and $i_{m+1} \neq h+1$, then
		\begin{eq}
			\relax [B]_\emptyset\ast [A]_{\Delta}  = &\sum\limits_{p \in [j_{m+1}+1, j_m-1], a_{h+1, p} \geq 1}v^{\beta_{A}(p)-1}\overline{[a_{h, p}+1, 1]}_v[A+E_{h, p}-E_{h+1, p}]_{\Delta} \\
			& + \sum\limits_{p = j_m, a_{h+1, p} \geq 1} v^{\beta_{A}(p)-1}\overline{[a_{h, p}, 1]}_v[A+E_{h, p}-E_{h+1, p}]_{\Delta}\\
			&+ \sum\limits_{p \notin [j_{m+1}+1, j_m], a_{h+1, p} \geq 1}v^{\beta_{A}(p)}\overline{[a_{h, p}+1, 1]}_v[A+E_{h, p}-E_{h+1, p}]_{\Delta}. \\
		\end{eq}
		
		\noindent$(c)$ If there exists $m$ such that $i_{m-1} \neq h$ and $i_{m} = h+1$, then
		\begin{eq}
			\relax [B]_\emptyset\ast [A]_{\Delta} = &\sum\limits_{p \in [1, n], a_{h+1, p} \geq 1} v^{\beta_{A}(p)}\overline{[a_{h, p}+1, 1]}_v[A+E_{h, p}-E_{h+1, p}]_{\Delta} \\
			& + \sum\limits_{p=j_m, a_{h+1, p} \geq 1} v^{\beta_{A}(p)-\sum_{j_{m+1}< j \leq p}a_{h+1, j}+1}[A+E_{h, p}-E_{h+1, p}]_{\Delta_c}\\
			& +\sum\limits_{p= j_m, a_{h+1, p}\geq 1, j_{m+1}<t< j_m}v^{\beta_{A}(p)-\sum_{t<j \leq p}a_{h+1, j}+1}[A+E_{h, p}-E_{h+1, p}]_{\Delta_t},\\
		\end{eq}
		where
		\begin{eq}
			\Delta_c&= \{(i_1, j_1), \cdots, (i_{m-1}, j_{m-1}), (h, j_m), (i_{m+1}, j_{m+1}), \cdots, (i_k, j_k)\},\\
			\Delta_t& = \{(i_1, j_1), \cdots, (i_{m-1}, j_{m-1}), (h, j_m), (h+1, t), (i_{m+1}, j_{m+1}), \cdots, (i_k, j_k)\}.
		\end{eq}
		
		\noindent$(d)$ If there exists $m$ such that $i_m=  h$ and $i_{m+1} = h+1$, then
		\begin{align*}
			&    \relax [B]_\emptyset\ast [A]_{\Delta} \\
			= & \sum\limits_{p \notin[j_{m+1}+1, j_m], a_{h+1, p} \geq 1}v^{\beta_{A}(p)}\overline{[a_{h, p}+1, 1]}_v[A+E_{h, p}-E_{h+1, p}]_{\Delta} \\
			& +  \sum\limits_{j_{m+1}<p<j_{m}, a_{h+1, p} \geq 1}v^{\beta_{A}(p)-1}\overline{[a_{h, p}+1, 1]}_v[A+E_{h, p}-E_{h+1, p}]_{\Delta} \\
			& +  \sum\limits_{p=j_m, a_{h+1, p} \geq 1} v^{\beta_{A}(p)-1}\overline{[a_{h, p}, 1]}_v[A+E_{h, p}-E_{h+1, p}]_{\Delta}\\
			& + \sum\limits_{\substack{p=j_{m+1},\\ a_{h+1, p} \geq 1}}v^{\beta_{A}(p)-\sum_{j_{m+2}<j \leq p}a_{h+1, j}+1}(1-v^{-2})\overline{[a_{h, p}+1, 1]}_v[A+E_{h, p}-E_{h+1, p}]_{\Delta_s}\\
			& + \sum\limits_{\substack{p=j_{m+1}, a_{h+1, p} \geq 1, \\j_{m+2}<t<j_{m+1}}} v^{\beta_{A}(p)-\sum_{t<j \leq p}a_{h+1, j}+1}(1-v^{-2})\overline{[a_{h, p}+1, 1]}_v[A+E_{h, p}-E_{h+1, p}]_{\Delta_t},
		\end{align*}
		where
		\begin{eq}
			\Delta_s&= \{(i_1, j_1), \cdots, (i_{m-1}, j_{m-1}), (h, j_m), (i_{m+2}, j_{m+2}), \cdots, (i_k, j_k)\}, \\
			\Delta_t &= \{(i_1, j_1), \cdots, (i_{m-1}, j_{m-1}), (h, j_m), (h+1, t), (i_{m+2}, j_{m+2}), \cdots, (i_k, j_k)\}.\\
		\end{eq}
	\end{prop}
	\begin{proof}
		We will show the statement $(b)$ is as specified, then the rest of formulas are similar. 
		For any $p\in[1, n]$, let $(X, \Delta')=(A+E_{h, p}-E_{h+1, p}, \Delta')$, we have the following identities.
		\begin{align*}
			& d(A, \Delta)-r(A, \Delta) = \sum \limits_{i\geq k, j<l} a_{i, j}a_{k, l}+\sum\limits_{\{(i, j)\}\leq \Delta}a_{i, j}, \\
			& d(B, \emptyset)-r(B, \emptyset)=\sum \limits_{j \in [1, n]} a_{h, j},\\
			& d(X, \Delta')-r(X, \Delta')=  \sum \limits_{i\geq k, j<l} a_{i, j}a_{k, l}+\sum\limits_{j<p}a_{h, p}- \sum\limits_{j>p}a_{h+1, p}\\
			& \qquad\qquad\qquad\qquad\quad + \sum\limits_{\{(i, j)\}\leq \Delta'}(a_{i, j}+\delta_{i h}\delta_{j p}-\delta_{i, h+1}\delta_{j p}),
		\end{align*}
		Hence
		\begin{eq}
			\quad & d(X, \Delta')-r(X, \Delta')-(d(A, \Delta)-r(A, \Delta))-(d(B, \emptyset)-r(B, \emptyset))+ 2\sum\limits_{j\geq p}a_{h, j}\\
			= & \sum\limits_{j\geq p}a_{h, j}-\sum\limits_{j> p}a_{h+1, j}
			+\sum\limits_{\{(i, j)\}\leq \Delta'}(a_{i, j}+\delta_{i h}\delta_{j p}-\delta_{i, h+1}\delta_{j p}))-\sum\limits_{\{(i, j)\}\leq \Delta}a_{i j}.\\
		\end{eq}
		
		In $(b)$, we have $\Delta'=\Delta$, so if $p\notin [j_{m+1}, j_m]$, then
		\begin{align*}
			&d(X, \Delta')-r(X, \Delta')-(d(A, \Delta)-r(A, \Delta)) -(d(B, \emptyset)-r(B, \emptyset))+ 2\sum\limits_{j\geq p}a_{h, j}\\
			= & \sum\limits_{j\geq p}a_{h, j}-\sum\limits_{j> p}a_{h+1, j}.
		\end{align*}
		Then the exponent of $v$ in this case is $\sum\limits_{j\geq p}a_{h, j}-\sum\limits_{j> p}a_{h+1, j}$.
		
		If $p\in [j_{m+1}+1, j_m]$, then
		\begin{align*}
			\quad & d(X, \Delta')-r(X, \Delta')-(d(A, \Delta)-r(A, \Delta))-(d(B, \emptyset)-r(B, \emptyset))+ 2\sum\limits_{j\geq p}a_{h, j}\\
			= & \sum\limits_{j\geq p}a_{h, j}-\sum\limits_{j> p}a_{h+1, j}+1.
		\end{align*}
		Therefore we have the exponent of $v$ in this case is $\sum\limits_{j\geq p}a_{h, j}-\sum\limits_{j> p}a_{h+1, j}-1$. Thus $(b)$ follows.
	\end{proof}
	
	\begin{prop}\label{prop 4.1.5}
		Assume $h\in [1, n-1]$. For $(C,\emptyset)\in \Xi_{d}, (A, \Delta) \in \Xi_{n|n, d}$ satisfied $\ro(A) = \co(C)$ and $C - E_{h+1, h}$ is a diagonal matrix, 
		where $\Delta = \{(i_1, j_1), \cdots, (i_k, j_k)\}$, let $\beta_{A}'(p)=\sum_{j \leq p}a_{h+1, j}-\sum_{j<p}a_{h, j}$ for $p\in [1, n]$.
		
		\noindent$(a)$ If for any $t$, we have $i_t \neq h, h+1$, then
		$$[C]_{\emptyset}\ast [A]_\Delta = \sum\limits_{p\in[1, n], a_{h, p} \geq 1} v^{\beta_{A}'(p)}\overline{[a_{h+1, p}+1, 1]}_v[A+E_{h+1, p}-E_{h, p}]_{\Delta}.$$
		
		\noindent$(b)$ If there exists $m$ such that $i_m= h$ and $i_{m+1} \neq h+1$, we have
		\begin{align*}
			\relax [C]_\emptyset\ast [A]_{\Delta}= & \sum\limits_{p \in[j_{m+1}+1, j_m], a_{h, p} \geq 1} v^{\beta_{A}'(p)-1}\overline{[a_{h+1, p}+1, 1]}_v[A+E_{h+1, p}-E_{h, p}]_{\Delta} \\
			& +\sum\limits_{p \notin [j_{m+1}+1, j_m], a_{h, p} \geq 1} v^{\beta_{A}'(p)}\overline{[a_{h+1, p}+1, 1]}_v[A+E_{h+1, p}-E_{h, p}]_{\Delta} \\  
			& + \sum\limits_{p = j_m, a_{h, p} \geq 1}v^{\sum_{j\leq j_{m+1}}a_{h+1, j}-\sum_{j< j_{m}}a_{h, j}}[A+E_{h+1, p}-E_{h, p}]_{\Delta_b} \\
			& + \sum\limits_{j_{m+1}<p = t<j_{m}, a_{h, p} \geq 1}v^{\sum_{j\leq j_{m+1}}a_{h+1, j}-\sum_{j< p}a_{h, j}}[A+E_{h+1, p}-E_{h, p}]_{\Delta_t},
		\end{align*}
		where
		\begin{eq}
			\Delta_b &= \{(i_1, j_1), \cdots, (i_{m-1}, j_{m-1}), (h+1, j_m), (i_{m+1}, j_{m+1}), \cdots, (i_k, j_k)\}, \\
			\Delta_t &= \{(i_1, j_1), \cdots, (i_{m-1}, j_{m-1}), (h, j_m), (h+1, t), (i_{m+1}, j_{m+1}), \cdots, (i_k, j_k)\}.
		\end{eq}
		
		\noindent$(c)$ If there exists $m$ such that $i_{m-1} \neq h$ and $i_{m} = h+1$, then
		\begin{eq}
			\relax [C]_\emptyset\ast [A]_{\Delta} = & \sum\limits_{p \neq j_m, a_{h, p} \geq 1} v^{\beta_{A}'(p)}\overline{[a_{h+1, p}+1, 1]}_v[A+E_{h+1, p}-E_{h, p}]_{\Delta} \\
			& + \sum\limits_{p=j_m, a_{h, p} \geq 1} v^{\beta_{A}'(p)}\overline{[a_{h+1, p}, 1]}_v[A+E_{h+1, p}-E_{h, p}]_{\Delta}.\\
		\end{eq}
		
		\noindent$(d)$ If there exists  $m$ such that $i_m=  h$ and $i_{m+1} = h+1$, we have
		\begin{align*}
			& \relax [C]_\emptyset\ast [A]_{\Delta} \\
			= & \sum\limits_{p=j_{m+1}, a_{h, p} \geq 1}v^{\beta_{A}'(p)}\overline{[a_{h+1, p}, 1]}_v[A+E_{h+1, p}-E_{h, p}]_{\Delta} \\
			& +  \sum\limits_{p \neq j_{m+1}, a_{h, p} \geq 1} v^{\beta_{A}'(p)}\overline{[a_{h+1, p}+1, 1]}_v[A+E_{h+1, p}-E_{h, p}]_{\Delta} \\
			& + \sum\limits_{p=j_{m}, a_{h, p} \geq 1} v^{\beta_{A}'(j_{m+1})-\sum_{j_{m+1}\leq j< p}a_{h, j}}(1-v^{-2})\overline{[a_{h+1, j_{m+1}}, 1]}_v[A+E_{h+1, p}-E_{h, p}]_{\Delta_s} \\ 
			& + \sum\limits_{j_{m+1}<p=t<j_m, a_{h, p} \geq 1} v^{\beta_{A}'(j_{m+1})-\sum_{j_{m+1}\leq j < p}a_{h, j}}(1-v^{-2})\overline{[a_{h+1, j_{m+1}}, 1]}_v[A+E_{h+1, p}-E_{h, p}]_{\Delta_t},
		\end{align*}
		where
		\begin{eq}
			\Delta_s &= \{(i_1, j_1), \cdots, (i_{m-1}, j_{m-1}), (h+1, j_m), (i_{m+2}, j_{m+2}), \cdots, (i_k, j_k)\},\\
			\Delta_t &=\{(i_1, j_1), \cdots, (i_{m-1}, j_{m-1}), (h, j_m), (h+1, t), (i_{m+2}, j_{m+2}), \cdots, (i_k, j_k)\}.
		\end{eq}
	\end{prop}
	
	\begin{proof}
		The proof is similar to Proposition \ref{prop 4.1.4}.
	\end{proof}
	Let $(A,\Delta),(A,\Delta')\in \Xi_{n|n,d}$. for any $(i_p,j_p)\in \Delta$, we define
	\begin{align*}
		\theta_{A}'(i_p, j_p)= \left\{ \begin{array}{ll}
			v^{\gamma_{A}(p)}(1-2v^{-2\sum_{j\leq j_p}a_{i_{p}, j}})  & \text{ if } (i_p, j_p)\in \Delta',\\
			v^{\gamma_{A}(p)}(1-v^{-2a_{i_p, j_p}})  & \text{ otherwise },\\
		\end{array}
		\right.
	\end{align*}
	where $\gamma_{A}(p)=\sum_{i_{p-1}<i\leq i_p,\ j\leq j_p}a_{i, j}-\sum_{i_{p-1}<i\leq i_p, j>j_p}a_{i, j}$.
	\begin{prop}\label{prop 4.1.6}
		Assume $h\in [1, n]$. For $(D,\{h, h\}), (A, \Delta) \in \Xi_{n|n,d}$ satisfied $\ro(A)$$= \co(D)$ and $D$ is a diagonal matrix, 
		where $\Delta = \{(i_1, j_1), \cdots, (i_k, j_k)\}$.
		
		\noindent$(a)$ If there exists $(i_m,j_m)\in \Delta$ such that $i_{m-1}<h<i_m$, 
		\begin{align*}
			\relax [D]_{\{(h,h)\}}\ast [A]_{\Delta}=\sum\limits_{\Delta'} v^{\sum_{\{(i,j)\}\leq \Delta'}a_{i, j}-\sum_{\{(i,j)\}\leq \Delta}a_{i, j}}\prod^{m-1}_{p=1}\theta_{A}'(i_{p}, j_{p})\rho'(A,\Delta') [A]_{\Delta'},
		\end{align*}
		where $\Delta'$ runs over the set
		$$
		\left\{ \Delta'=\{(i'_1, j'_1), \cdots, (i'_{l}, j'_{l})\}\ \bigg| \ \begin{array}{l}
			(1)  (i_{t}, j_{t})\in \Delta'\text{ for all } t \in [m, k] \\
			(2)  \text{ if } (i', j')\in \Delta'\setminus \Delta, \text{ then } i'\leq h
		\end{array}\right\},$$
		\begin{align*}
			\rho'(A,\Delta')= \left\{ \begin{array}{ll}	
				v^{\sum_{i_{m-1}<i\leq h, j\leq j_m} a_{i, j}-\sum_{i_{m-1}<i\leq h, j> j_m} a_{i, j}}  & \text{ if } (h, j')\in \Delta',\\
				v^{\sum_{i_{m-1}<i\leq h, j\leq j_m} a_{i, j}-\sum_{i_{m-1}<i\leq h, j> j_m} a_{i, j}}(1-v^{-2\sum_{j\leq j_m}a_{h, j}})  & \text{ otherwise }.\\
			\end{array}
			\right.
		\end{align*}
		
		\noindent$(b)$ If there is $i_m=h$, then
		\begin{eq}
			\relax [D]_{\{(h,h)\}}\ast [A]_{\Delta}=\sum\limits_{\Delta'}v^{\sum_{\{(i,j)\}\leq \Delta'}a_{i, j}-\sum_{\{(i,j)\}\leq \Delta}a_{i, j}}\prod^{m}_{p=1}\theta_{A}'(i_{p}, j_{p})[A]_{\Delta'},
		\end{eq}
		where $\Delta'$ runs over the set
		$$
		\left\{ \Delta'=\{(i'_1, j'_1), \cdots, (i'_{l}, j'_{l})\}\ \bigg| \ \begin{array}{l}
			(1)  (i_{t}, j_{t})\in \Delta'\text{ for all } t \in [m+1, k] \\
			(2)  \text{ if } (i', j')\in \Delta'\setminus \Delta, \text{ then } i'\leq h
		\end{array}\right\}.$$
		
		\noindent$(c)$ If $i_k<h$, then
		\begin{eq}
			\relax [D]_{\{(h,h)\}}\ast [A]_{\Delta}=\sum\limits_{\Delta'} v^{\sum_{\{(i,j)\}\leq \Delta'}a_{i, j}-\sum_{\{(i,j)\}\leq \Delta}a_{i, j}-\sum_{i_k<i\leq h, j}a_{i j}}\prod^{k}_{p=1}\theta_{A}'(i_{p}, j_{p}) [A]_{\Delta'},
		\end{eq}
		where $\Delta'$ runs over the set $\{\Delta'=\{(i'_1, j'_1), \cdots, (i'_{l}, j'_{l})\}\ |\ i'_{l}\leq h\}$. 
		
		\noindent$(d)$ If $\Delta=\emptyset$, 
		\begin{align*}
			\relax [D]_{\{(h,h)\}}\ast [A]_{\emptyset}=\sum\limits_{t\in [1,n], a_{h, t}>0} v^{-\sum_{i\leq h,j>t}a_{i, j}} [A]_{\{(h,t)\}}.
		\end{align*}
	\end{prop}  
	\begin{proof}
		The proof is similar to Proposition \ref{prop 4.1.4}.
	\end{proof}
	\begin{prop}\label{prop 4.1.7}
		Assume $h\in [1, n-1]$ and $R\geq 0$ is an integer. Take $P_R=\{\mathbf{p}=(q_{12},\cdots, p_n)\in \mathbb{N}^n\ |\ \sum_{i}p_i=R \}$. For $(B, \emptyset), (A,\Delta) \in \Xi_{n|n,d}$ satisfied $\ro(A) = \co(B)$, where $B - RE_{h, h+1}$ is a diagonal matrix and $\Delta = \{(i_1, j_1), \cdots, (i_k, j_k)\}$, let $\kappa_{A}(\mathbf{p})=\sum_{j\geq u}a_{h, j}p_u-\sum_{j>u}a_{h+1, j}p_u+\sum_{u<u'}p_{u}p_{u'}$, $X_{\mathbf{p}}=A+\sum_{u}p_u(E_{h, u}-E_{h+1, u})$ where $\mathbf{p}\in P_R$.
		
		\noindent$(a)$ If for any $t\in [1,k]$, we have $i_t \neq h, h+1$, then
		$$[B]_{\emptyset}\ast [A]_\Delta = \sum\limits_{\mathbf{p}\in P_R, a_{h+1, u} \geq p_u}
		v^{\kappa_{A}(\mathbf{p})}\prod^{n}_{u=1}\overline{[a_{h, u}+p_u, p_u]}_v[X_{\mathbf{p}}]_{\Delta}.
		$$
		
		\noindent$(b)$ If there exists $m$ such that $i_m = h$ and $i_{m+1} \neq h+1$, then
		\begin{eq}
			\relax [B]_\emptyset\ast [A]_{\Delta}  = &\sum\limits_{\mathbf{p} \in P_R, a_{h+1, u} \geq p_u}v^{\kappa_{A}(\mathbf{p})-\sum_{j_{m+1}<u\leq j_m}p_u}\prod^{n}_{u=1}\overline{[a_{h, u}+p_u-\delta_{u j_m}, p_u]}_v[X_{\mathbf{p}}]_{\Delta}.
		\end{eq}
		
		\noindent$(c)$ If there exists $m$ such that $i_{m-1} \neq h$ and $i_{m} = h+1$, then
		\begin{align*}
			& \relax [B]_\emptyset\ast [A]_{\Delta}\\
			= &\sum\limits_{\mathbf{p}\in P_R, a_{h+1, u} \geq p_u}
			v^{\kappa_{A}(\mathbf{p})}\prod^{n}_{u=1}\overline{[a_{h, u}+p_u, p_u]}_v[X_{\mathbf{p}}]_{\Delta} \\
			& + \sum\limits_{\substack{\mathbf{p}\in P_R, p_{j_m}>0,\\ a_{h+1, u} \geq p_u}} v^{\kappa_{A}(\mathbf{p})-\sum_{j_{m+1}< j \leq j_m}(a_{h+1, j}-p_j)}\prod^{n}_{u=1}\overline{[a_{h, u}+p_u-\delta_{uj_m}, p_u-\delta_{uj_m}]}_v[X_{\mathbf{p}}]_{\Delta_c}\\
			& +\sum\limits_{\substack{\mathbf{p}\in P_R, a_{h+1, u} \geq p_u,\\ p_{j_m}>0, j_{m+1}<t< j_m}}v^{\kappa_{A}(\mathbf{p})-\sum_{t< j \leq j_m}(a_{h+1, j}-p_j)}\prod^{n}_{u=1}\overline{[a_{h, u}+p_u-\delta_{uj_m}, p_u-\delta_{uj_m}]}_v[X_{\mathbf{p}}]_{\Delta_t},
		\end{align*}
		where
		\begin{eq}
			\Delta_c&= \{(i_1, j_1), \cdots, (i_{m-1}, j_{m-1}), (h, j_m), (i_{m+1}, j_{m+1}), \cdots, (i_k, j_k)\},\\
			\Delta_t& = \{(i_1, j_1), \cdots, (i_{m-1}, j_{m-1}), (h, j_m), (h+1, t), (i_{m+1}, j_{m+1}), \cdots, (i_k, j_k)\}.
		\end{eq}
		
		\noindent$(d)$ If there exists $m$ such that $i_m=  h$ and $i_{m+1} = h+1$, then
		\begin{align*}
			&    \relax [B]_\emptyset\ast [A]_{\Delta} \\
			= &\sum\limits_{\mathbf{p} \in P_R, a_{h+1, u} \geq p_u}v^{\kappa_{A}(\mathbf{p})-\sum_{j_{m+1}<u\leq j_m}p_u}\prod^{n}_{u=1}\overline{[a_{h, u}+p_u-\delta_{uj_m}, p_u}_v[X_{\mathbf{p}}]_{\Delta}\\
			& + \sum\limits_{\substack{\mathbf{p} \in P_R, p_{j_{m+1}>0}\\a_{h+1, u} \geq p_u}}v^{\kappa_{A}(\mathbf{p})-\sum_{j_{m+2}<j\leq j_{m+1}}(a_{h+1, j}-p_j)}(1-v^{-2p_{j_{m+1}}})\prod^{n}_{u=1}\overline{[a_{h, u}+p_u, p_u]}_v[X_{\mathbf{p}}]_{\Delta_s}\\
			& + \sum\limits_{\substack{\mathbf{p} \in P_R, p_{j_{m+1}>0},\\a_{h+1, u} \geq p_u, j_{m+2}<t<j_{m+1}}} v^{\kappa_{A}(\mathbf{p})-\sum_{t<j\leq j_{m+1}}(a_{h+1, j}-p_j)}(1-v^{-2p_{j_{m+1}}})\prod^{n}_{u=1}\overline{[a_{h, u}+p_u, p_u]}_v[X_{\mathbf{p}}]_{\Delta_t},
		\end{align*}
		where
		\begin{eq}
			\Delta_s&= \{(i_1, j_1), \cdots, (i_{m-1}, j_{m-1}), (h, j_m), (i_{m+2}, j_{m+2}), \cdots, (i_k, j_k)\}, \\
			\Delta_t &= \{(i_1, j_1), \cdots, (i_{m-1}, j_{m-1}), (h, j_m), (h+1, t), (i_{m+2}, j_{m+2}), \cdots, (i_k, j_k)\}.\\
		\end{eq}
	\end{prop}
	\begin{proof}
		We use mathematical induction on $R$ to prove $(c)$, the remaining identities are similar. 
		
		If $R=1$, the formula follows from Proposition \ref{prop 4.1.4}. 
		Suppose for $R=l-1$, the identity holds. For $R=l$. Take $(B,\emptyset)\in \Xi_{n|n,d}$ such that $\co(B)=\ro(A)$ and $B-lE_{i,i+1}$ is a diagonal matrix. Let $(B',\emptyset), (B'',\emptyset)\in \Xi_{n|n,d}$ satisfy $\co(B')=\ro(A)$, $\co(B'')=\ro(B')$ and $B'-(l-1)E_{i,i+1}$, $B''-E_{i,i+1}$ are diagonal matrices. Then we have the following identity.
		\begin{align*}
			& [B]_{\emptyset}\ast [A]_{\Delta} = \frac{v-v^{-1}}{v^{l}-v^{-l}}[B'']_{\emptyset}\ast [B']_{\emptyset}\ast [A]_{\Delta}\\
			= & \frac{v-v^{-1}}{v^{l}-v^{-l}}(\sum_{\mathbf{p}\in P_{l-1}}v^{\kappa_{A}(\mathbf{p})}\prod^{n}_{u=1}\overline{[a_{h,u}+p_u,p_u]}_v[B'']_{\emptyset}\ast [X_{\mathbf{p}}]_{\Delta}\\
			& + \sum_{\mathbf{p}\in P_{l-1}}v^{\kappa_{A}(\mathbf{p})-\sum_{j_{m+1}<j\leq j_m}(a_{h+1,j}-p_j)}\prod^{n}_{u=1}\overline{[a_{h,u}+p_u-\delta_{uj_m}, p_u-\delta_{uj_m}]}_v[B'']_{\emptyset}\ast [X_{\mathbf{p}}]_{\Delta_c}\\
			& + \sum_{\substack{\mathbf{p}\in P_{l-1}\\j_{m+1<t<j_m}}}v^{\kappa_{A}(\mathbf{p})-\sum_{t<j\leq j_m}(a_{h+1,j}-p_j)}\prod^{n}_{u=1}\overline{[a_{h,u}+p_u-\delta_{uj_m}, p_u-\delta_{uj_m}]}_v[B'']_{\emptyset}\ast [X_{\mathbf{p}}]_{\Delta_t}).
		\end{align*}
		
		For any $\mathbf{p}\in P_{l-1}$ and an integer $q>j_{m+1}$, we define $\kappa_{A}(\mathbf{p},q)=\kappa_{A}(\mathbf{p})-\sum_{q<j\leq j_m}(a_{h+1,j}-p_j)$.
		For any fixed $\mathbf{p}\in P_{l-1}$ and $r\in [1,n]$, we define sequences $\mathbf{p}''\in P_{l}$ and $\mathbf{p}'_s\in P_{l-1}$ as follows.
		$$
		(\mathbf{p}'')_{i}=\left\{
		\begin{array}{ll}
			p_r+1   & i=r, \\
			p_i  & \text{otherwise},
		\end{array}
		\right.
		\text{ and }
		(\mathbf{p}'_{s})_{i}=\left\{
		\begin{array}{lll}
			p_r+1   & i=r, \\
			p_s-1  &  i=s, \\
			p_i  & \text{otherwise},
		\end{array}
		\right.
		$$
		where $s\in [1, r-1]\cup[r+1, n]$. 
		The coefficient of $[X_\mathbf{p}+E_{h,r}-E_{h+1,r}]_{\Delta}$ is the same as \cite[Lemma 3.4]{BLM}. So we only need to compute the coefficients of $[X_\mathbf{p}+E_{h,r}-E_{h+1,r}]_{\Delta_c}$(denoted as $\mathbf{c}_c$) and $[X_\mathbf{p}+E_{h,r}-E_{h+1,r}]_{\Delta_t}$(denoted as $\mathbf{c}_t$). 
		
		Firstly, we compute $\mathbf{c}_t$ in the following three cases.
		
		Case 1: $r=j_m$.
		\begin{align*}
			&\mathbf{c}_t =v^{\kappa_{A}(\mathbf{p})}\prod^{n}_{u=1}\overline{[a_{h,u}+p_u-\delta_{uj_m}, p_u-\delta_{uj_m}]}_v
			v^{\beta_{A}(r)-\sum_{t<j\leq j_m}(a_{h+1,j}-p_j)+1}\\
			&+ v^{\kappa_{A}(\mathbf{p},t)}\prod^{n}_{u=1}\overline{[a_{h,u}+p_u-\delta_{uj_m}, p_u-\delta_{uj_m}]}_v
			(v^{\beta_{A}(r)-1}\overline{[a_{h,r}+p_r, 1]}_v)\\
			& +\sum_{t<s<j_m}v^{\kappa_{A}(\mathbf{p'}_s, t)}\prod^{n}_{u=1}\overline{[a_{h,u}+p'_u-\delta_{uj_m}, p'_u-\delta_{uj_m}]}_v(v^{\beta_{A}(s)-1}\overline{[a_{h,s}+p_s, 1]}_v)\\
			& +\sum_{s\notin [t+1,j_m]}v^{\kappa_{A}(\mathbf{p'}_s, t)}\prod^{n}_{u=1}\overline{[a_{h,u}+p'_u-\delta_{uj_m}, p'_u-\delta_{uj_m}]}_v(v^{\beta_{A}(s)}\overline{[a_{h,s}+p_s, 1]}_v)\\
			& +\sum_{t<s< j_m}v^{\kappa_{A}(\mathbf{p'}_s, t)}\prod^{n}_{u=1}\overline{[a_{h,u}+p'_u-\delta_{uj_m}, p'_u-\delta_{uj_m}]}_v
			(v^{\beta_{A}(s)}(1-v^{-2})\overline{[a_{h,r}+p_s, 1]}_v)\\
			= & v^{\kappa_{A}(\mathbf{p''})-\sum_{t<j\leq j_m}(a_{h+1,j}-p''_j)}\prod^{n}_{u=1}\overline{[a_{h,u}+p''_u-\delta_{uj_m}, p''_u-\delta_{uj_m}]}_v\\
			&(v^{\sum_{j\geq r}p_j-\sum_{j< r}p_j}(\frac{v^{-2(p_r+1)}-v^{-2}}{v^{-2}-1}+1)+\sum_{ s\neq r}v^{\sum_{j\geq s}p'_j-\sum_{j< s}p'_j}\frac{v^{-2p_s}-1}{v^{-2}-1})\\
			= & v^{\kappa_{A}(\mathbf{p''})-\sum_{t<j\leq j_m}(a_{h+1,j}-p''_j)}\prod^{n}_{u=1}\overline{[a_{h,u}+p''_u-\delta_{uj_m}, p''_u-\delta_{uj_m}]}_v\frac{v^{-l-1}-v^{l-1}}{v^{-2}-1}.
		\end{align*}
		
		Case 2: $r\notin [t+1, j_m-1]$.
		\begin{align*}
			& \mathbf{c}_t=v^{\kappa_{A}(\mathbf{p'},t)}\prod^{n}_{u=1}\overline{[a_{h,u}+p_u-\delta_{uj_m}, p_u-\delta_{uj_m}]}_v
			v^{\beta_{A}(j_m)-\sum_{t<j\leq j_m}(a_{h+1,j}-p_j)+1}\\
			& + v^{\kappa_{A}(\mathbf{p},t)}\prod^{n}_{u=1}\overline{[a_{h,u}+p_u-\delta_{uj_m}, p_u-\delta_{uj_m}]}_v
			(v^{\beta_{A}(r)}\overline{[a_{h,r}+p_r+1, 1]}_v)\\
			& +\sum_{t<s\leq j_m}v^{\kappa_{A}(\mathbf{p'}_s, t)}\prod^{n}_{u=1}\overline{[a_{h,u}+p'_u-\delta_{uj_m}, p'_u-\delta_{uj_m}]}_v(v^{\beta_{A}(s)-1}\overline{[a_{h,s}+p_s-\delta_{sj_m}, 1]}_v)\\
			& +\sum_{s\notin [t+1,j_m]}v^{\kappa_{A}(\mathbf{p'}_s,t)}\prod^{n}_{u=1}\overline{[a_{h,u}+p'_u-\delta_{uj_m}, p'_u-\delta_{uj_m}]}_v(v^{\beta_{A}(s)}\overline{[a_{h,s}+p_s, 1]}_v)\\
			& +\sum_{t<s< j_m}v^{\kappa_{A}(\mathbf{p'}_s,t)}\prod^{n}_{u=1}\overline{[a_{h,u}+p'_u-\delta_{uj_m}, p'_u-\delta_{uj_m}]}_v
			(v^{\beta_{A}(s)}(1-v^{-2})\overline{[a_{h,r}+p_s, 1]}_v)\\
			= & v^{\kappa_{A}(\mathbf{p''})-\sum_{t<j\leq j_m}(a_{h+1,j}-p''_j)}\prod^{n}_{u=1}\overline{[a_{h,u}+p''_u-\delta_{uj_m}, p''_u-\delta_{uj_m}]}_v\\
			& (v^{\sum_{j\geq r}p_j-\sum_{j< r}p_j}\frac{v^{-2(p_r+1)}-1}{v^{-2}-1}+\sum_{ s\neq r}v^{\sum_{j\geq s}p'_j-\sum_{j< s}p'_j}\frac{v^{-2p_s}-1}{v^{-2}-1})\\
			= & v^{\kappa_{A}(\mathbf{p''})-\sum_{t<j\leq j_m}(a_{h+1,j}-p''_j)}\prod^{n}_{u=1}\overline{[a_{h,u}+p''_u-\delta_{uj_m}, p''_u-\delta_{uj_m}]}_v\frac{v^{-l-1}-v^{l-1}}{v^{-2}-1}.
		\end{align*}
		
		Case 3: $t<r <j_m$.
		\begin{align*}
			& \mathbf{c}_t=v^{\kappa_{A}(\mathbf{p'},t)}\prod^{n}_{u=1}\overline{[a_{h,u}+p_u-\delta_{uj_m}, p_u-\delta_{uj_m}]}_v
			v^{\beta_{A}(j_m)-\sum_{t<j\leq j_m}(a_{h+1,j}-p_j)+1}\\
			& v^{\kappa_{A}(\mathbf{p},t)}\prod^{n}_{u=1}\overline{[a_{h,u}+p_u-\delta_{uj_m}, p_u-\delta_{uj_m}]}_v
			(v^{\beta_{A}(r)-1}\overline{[a_{h,r}+p_r+1, 1]}_v)\\
			& +\sum_{\substack{t<s\leq j_m\\ s\neq r}}v^{\kappa_{A}(\mathbf{p'}_s,t)}\prod^{n}_{u=1}\overline{[a_{h,u}+p'_u-\delta_{uj_m}, p'_u-\delta_{uj_m}]}_v(v^{\beta_{A}(s)-1}\overline{[a_{h,s}+p_s-\delta_{sj_m}, 1]}_v)\\
			& +\sum_{s\notin [t+1,j_m]}v^{\kappa_{A}(\mathbf{p'}_s, t)}\prod^{n}_{u=1}\overline{[a_{h,u}+p'_u-\delta_{uj_m}, p'_u-\delta_{uj_m}]}_v(v^{\beta_{A}(s)}\overline{[a_{h,s}+p_s, 1]}_v)\\
			&+ v^{\kappa_{A}(\mathbf{p}, t)}\prod^{n}_{u=1}\overline{[a_{h,u}+p_u-\delta_{uj_m}, p_u-\delta_{uj_m}]}_v
			(v^{\beta_{A}(r)}(1-v^{-2})\overline{[a_{h,r}+p_r+1, 1]}_v)\\
			& +\sum_{\substack{t<s<j_m\\ s\neq r}}v^{\kappa_{A}(\mathbf{p'}_s, t)}\prod^{n}_{u=1}\overline{[a_{h,u}+p'_u-\delta_{uj_m}, p'_u-\delta_{uj_m}]}_v
			(v^{\beta_{A}(s)}(1-v^{-2})\overline{[a_{h,r}+p_s, 1]}_v)\\
			= & v^{\kappa_{A}(\mathbf{p''})-\sum_{t<j\leq j_m}(a_{h+1,j}-p''_j)}\prod^{n}_{u=1}\overline{[a_{h,u}+p''_u-\delta_{uj_m}, p''_u-\delta_{uj_m}]}_v\\
			&(v^{\sum_{j\geq r}p_j-\sum_{j< r}p_j}\frac{v^{-2(p_r+1)}-1}{v^{-2}-1}+\sum_{ s\neq r}v^{\sum_{j\geq s}p'_j-\sum_{j< s}p'_j}\frac{v^{-2p_s}-1}{v^{-2}-1})\\
			= & v^{\kappa_{A}(\mathbf{p''})-\sum_{t<j\leq j_m}(a_{h+1,j}-p''_j)}\prod^{n}_{u=1}\overline{[a_{h,u}+p''_u-\delta_{uj_m}, p''_u-\delta_{uj_m}]}_v\frac{v^{-l-1}-v^{l-1}}{v^{-2}-1}.
		\end{align*}
		
		Secondly, we compute $\mathbf{c}_c$ is the following three cases.
		
		Case 1: $r=j_m$.
		\begin{align*}
			& \mathbf{c}_c=v^{\kappa_{A}(\mathbf{p},j_m)}\prod^{n}_{u=1}\overline{[a_{h,u}+p_u-\delta_{uj_m}, p_u-\delta_{uj_m}]}_v
			v^{\beta_{A}(r)-\sum_{j_{m+1}<j\leq j_m}(a_{h+1,j}-p_j)+1}\\
			&+ v^{\kappa_{A}(\mathbf{p},j_m)}\prod^{n}_{u=1}\overline{[a_{h,u}+p_u-\delta_{uj_m}, p_u-\delta_{uj_m}]}_v
			(v^{\beta_{A}(r)-1}\overline{[a_{h,r}+p_r, 1]}_v)\\
			& +\sum_{j_m+1<s< j_m}v^{\kappa_{A}(\mathbf{p'}_s, j_m)}\prod^{n}_{u=1}\overline{[a_{h,u}+p'_u-\delta_{uj_m}, p'_u-\delta_{uj_m}]}_v(v^{\beta_{A}(s)-1}\overline{[a_{h,s}+p_s, 1]}_v)\\
			& +\sum_{s\notin [j_{m+1}+1,j_m]}v^{\kappa_{A}(\mathbf{p'}_s, j_m)}\prod^{n}_{u=1}\overline{[a_{h,u}+p'_u-\delta_{uj_m}, p'_u-\delta_{uj_m}]}_v(v^{\beta_{A}(s)}\overline{[a_{h,s}+p_s, 1]}_v)\\
			& +\sum_{j_m+1<s< j_m}v^{\kappa_{A}(\mathbf{p'}_s, j_m)}\prod^{n}_{u=1}\overline{[a_{h,u}+p'_u-\delta_{uj_m}, p'_u-\delta_{uj_m}]}_v
			(v^{\beta_{A}(s)+1}(1-v^{-2})\overline{[a_{h,r}+p_s, 1]}_v)\\
			= & v^{\kappa_{A}(\mathbf{p''})-\sum_{j_{m+1}<j\leq j_m}(a_{h+1,j}-p''_j)}\prod^{n}_{u=1}\overline{[a_{h,u}+p''_u-\delta_{uj_m}, p''_u-\delta_{uj_m}]}_v\\
			&(v^{\sum_{j\geq r}p_j-\sum_{j< r}p_j}(\frac{v^{-2(p_r+1)}-v^{-2}}{v^{-2}-1}+1)+\sum_{ s\neq r}v^{\sum_{j\geq s}p'_j-\sum_{j< s}p'_j}\frac{v^{-2p_s}-1}{v^{-2}-1})\\
			= & v^{\kappa_{A}(\mathbf{p''})-\sum_{j_{m+1}<j\leq j_m}(a_{h+1,j}-p''_j)}\prod^{n}_{u=1}\overline{[a_{h,u}+p''_u-\delta_{uj_m}, p''_u-\delta_{uj_m}]}_v\frac{v^{-l-1}-v^{l-1}}{v^{-2}-1}.
		\end{align*}
		
		Case 2: $r\notin [j_{m+1}+1, j_m-1]$.
		\begin{align*}
			& \mathbf{c}_c=v^{\kappa_{A}(\mathbf{p'},j_m)}\prod^{n}_{u=1}\overline{[a_{h,u}+p_u-\delta_{uj_m}, p_u-\delta_{uj_m}]}_v
			v^{\beta_{A}(j_m)-\sum_{j_{m+1}<j\leq j_m}(a_{h+1,j}-p_j)+1}\\
			& + v^{\kappa_{A}(\mathbf{p},j_m)}\prod^{n}_{u=1}\overline{[a_{h,u}+p_u-\delta_{uj_m}, p_u-\delta_{uj_m}]}_v
			(v^{\beta_{A}(r)}\overline{[a_{h,r}+p_r+1, 1]}_v)\\
			& +\sum_{j_{m+1}<s\leq j_m}v^{\kappa_{A}(\mathbf{p'}_s, j_m)}\prod^{n}_{u=1}\overline{[a_{h,u}+p'_u-\delta_{uj_m}, p'_u-\delta_{uj_m}]}_v(v^{\beta_{A}(s)-1}\overline{[a_{h,s}+p_s-\delta_{sj_m}, 1]}_v)\\
			& +\sum_{s\notin [j_{m+1}+1,j_m]}v^{\kappa_{A}(\mathbf{p'}_s,j_m)}\prod^{n}_{u=1}\overline{[a_{h,u}+p'_u-\delta_{uj_m}, p'_u-\delta_{uj_m}]}_v(v^{\beta_{A}(s)}\overline{[a_{h,s}+p_s, 1]}_v)\\
			& +\sum_{j_{m+1}<s<j_m}v^{\kappa_{A}(\mathbf{p'}_s,j_m)}\prod^{n}_{u=1}\overline{[a_{h,u}+p'_u-\delta_{uj_m}, p'_u-\delta_{uj_m}]}_v
			(v^{\beta_{A}(s)+1}(1-v^{-2})\overline{[a_{h,r}+p_s, 1]}_v)\\
			= & v^{\kappa_{A}(\mathbf{p''})-\sum_{j_{m+1}<j\leq j_m}(a_{h+1,j}-p''_j)}\prod^{n}_{u=1}\overline{[a_{h,u}+p''_u-\delta_{uj_m}, p''_u-\delta_{uj_m}]}_v\\
			& (v^{\sum_{j\geq r}p_j-\sum_{j< r}p_j}\frac{v^{-2(p_r+1)}-1}{v^{-2}-1}+\sum_{ s\neq r}v^{\sum_{j\geq s}p'_j-\sum_{j< s}p'_j}\frac{v^{-2p_s}-1}{v^{-2}-1})\\
			= & v^{\kappa_{A}(\mathbf{p''})-\sum_{j_{m+1}<j\leq j_m}(a_{h+1,j}-p''_j)}\prod^{n}_{u=1}\overline{[a_{h,u}+p''_u-\delta_{uj_m}, p''_u-\delta_{uj_m}]}_v\frac{v^{-l-1}-v^{l-1}}{v^{-2}-1}.
		\end{align*}
		
		Case 3: $j_{m+1}<r <j_m$.
		\begin{align*}
			& \mathbf{c}_c=v^{\kappa_{A}(\mathbf{p'},j_m)}\prod^{n}_{u=1}\overline{[a_{h,u}+p_u-\delta_{uj_m}, p_u-\delta_{uj_m}]}_v
			v^{\beta_{A}(j_m)-\sum_{j_{m+1}<j\leq j_m}(a_{h+1,j}-p_j)+1}\\
			& v^{\kappa_{A}(\mathbf{p},j_m)}\prod^{n}_{u=1}\overline{[a_{h,u}+p_u-\delta_{uj_m}, p_u-\delta_{uj_m}]}_v
			(v^{\beta_{A}(r)-1}\overline{[a_{h,r}+p_r+1, 1]}_v)\\
			& +\sum_{\substack{j_{m+1}<s\leq j_m\\ s\neq r}}v^{\kappa_{A}(\mathbf{p'}_s,j_m)}\prod^{n}_{u=1}\overline{[a_{h,u}+p'_u-\delta_{uj_m}, p'_u-\delta_{uj_m}]}_v(v^{\beta_{A}(s)-1}\overline{[a_{h,s}+p_s-\delta_{sj_m}, 1]}_v)\\
			& +\sum_{s\notin [j_{m+1}+1,j_m]}v^{\kappa_{A}(\mathbf{p'}_s, j_m)}\prod^{n}_{u=1}\overline{[a_{h,u}+p'_u-\delta_{uj_m}, p'_u-\delta_{uj_m}]}_v(v^{\beta_{A}(s)}\overline{[a_{h,s}+p_s, 1]}_v)\\
			&+ v^{\kappa_{A}(\mathbf{p}, j_m)}\prod^{n}_{u=1}\overline{[a_{h,u}+p_u-\delta_{uj_m}, p_u-\delta_{uj_m}]}_v
			(v^{\beta_{A}(r)+1}(1-v^{-2})\overline{[a_{h,r}+p_r+1, 1]}_v)\\
			& +\sum_{\substack{j_{m+1}<s< j_m\\ s\neq r}}v^{\kappa_{A}(\mathbf{p'}_s, j_m)}\prod^{n}_{u=1}\overline{[a_{h,u}+p'_u-\delta_{uj_m}, p'_u-\delta_{uj_m}]}_v
			(v^{\beta_{A}(s)+1}(1-v^{-2})\overline{[a_{h,r}+p_s, 1]}_v)\\
			= & v^{\kappa_{A}(\mathbf{p''})-\sum_{j_{m+1}<j\leq j_m}(a_{h+1,j}-p''_j)}\prod^{n}_{u=1}\overline{[a_{h,u}+p''_u-\delta_{uj_m}, p''_u-\delta_{uj_m}]}_v\\
			&(v^{\sum_{j\geq r}p_j-\sum_{j< r}p_j}\frac{v^{-2(p_r+1)}-1}{v^{-2}-1}+\sum_{ s\neq r}v^{\sum_{j\geq s}p'_j-\sum_{j< s}p'_j}\frac{v^{-2p_s}-1}{v^{-2}-1})\\
			= & v^{\kappa_{A}(\mathbf{p''})-\sum_{j_{m+1}<j\leq j_m}(a_{h+1,j}-p''_j)}\prod^{n}_{u=1}\overline{[a_{h,u}+p''_u-\delta_{uj_m}, p''_u-\delta_{uj_m}]}_v\frac{v^{-l-1}-v^{l-1}}{v^{-2}-1}.
		\end{align*}
		Therefore, we conclude that $(c)$ holds by mathmetical induction. 
	\end{proof}
	
	\begin{prop}\label{prop 4.1.8}
		Assume $h\in [1, n-1]$ and $R\geq 0$ is an integer. Take $P_R=\{\mathbf{p}=(p_1,\cdots, p_n)\in \mathbb{N}^n\ |\ \sum_{i}p_i=R \}$. For $(C, \emptyset), (A,\Delta) \in \Xi_{n|n,d}$ satisfied $\ro(A) = \co(C)$, where $C - RE_{h+1, h}$ is a diagonal matrix and $\Delta = \{(i_1, j_1), \cdots, (i_k, j_k)\}$, let $\kappa_{A}'(\mathbf{p})=\sum_{j\leq u}a_{h+1, j}p_u-\sum_{j<u}a_{h, j}p_u+\sum_{u<u'}p_{u}p_{u'}$ and $X'_{\mathbf{p}}=A+\sum_{u}p_u(E_{h+1, u}-E_{h, u})$ where $\mathbf{p}\in P_R$.
		
		\noindent$(a)$ If for any $t\in [1,n]$, we have $i_t \neq h, h+1$, then
		$$[C]_{\emptyset}\ast [A]_\Delta = \sum\limits_{\mathbf{p}\in P_R, a_{h, u} \geq p_u} v^{\kappa_{A}'(\mathbf{p})}\prod^{n}_{u=1}\overline{[a_{h+1, u}+p_u, p_u]}_v[X_{\mathbf{p}}]_{\Delta}.$$
		
		\noindent$(b)$ If there exists $m$ such that $i_m= h$ and $i_{m+1} \neq h+1$, we have
		\begin{align*}
			& \relax [C]_\emptyset\ast [A]_{\Delta}\\
			= & \sum\limits_{\mathbf{p}\in P_R, a_{h, u} \geq p_u} v^{\kappa_{A}'(\mathbf{p})-\sum_{j_{m+1}<u\leq j_m}p_u}\prod^{n}_{u=1}\overline{[a_{h+1, u}+p_u, p_u]}_v[X'_{\mathbf{p}}]_{\Delta} \\
			& +\sum\limits_{\substack{\mathbf{p}\in P_R, a_{h, u} \geq p_u\\p_{j_m}>0}} v^{\kappa_{A}'(\mathbf{p})-\sum_{j_{m+1}<j\leq j_m}a_{h+1,j}} \prod^{n}_{u=1}\overline{[a_{h+1, u}+p_u-\delta_{u j_m}, p_u-\delta_{u j_m}]}_v[X'_{\mathbf{p}}]_{\Delta_b} \\  
			& + \sum\limits_{\substack{\mathbf{p}\in P_R, a_{h, u} \geq p_u\\ j_{m+1}<t<j_m, p_t>0}}v^{\kappa_{A}'(\mathbf{p})-\sum_{j_{m+1}<j\leq t}a_{h+1,j}}\prod^{n}_{u=1}\overline{[a_{h+1, u}+p_u-\delta_{u t}, p_u-\delta_{u t}]}_v[X'_{\mathbf{p}}]_{\Delta_t},
		\end{align*}
		where
		\begin{eq}
			\Delta_b &= \{(i_1, j_1), \cdots, (i_{m-1}, j_{m-1}), (h+1, j_m), (i_{m+1}, j_{m+1}), \cdots, (i_k, j_k)\}, \\
			\Delta_t &= \{(i_1, j_1), \cdots, (i_{m-1}, j_{m-1}), (h, j_m), (h+1, t), (i_{m+1}, j_{m+1}), \cdots, (i_k, j_k)\}.
		\end{eq}
		
		\noindent$(c)$ If there exists $m$ such that $i_{m-1} \neq h$ and $i_{m} = h+1$, then
		\begin{eq}
			\relax [C]_\emptyset\ast [A]_{\Delta}= & \sum\limits_{\mathbf{p}\in P_R, a_{h, u} \geq p_u} v^{\kappa_{A}'(\mathbf{p})}\prod^{n}_{u=1}\overline{[a_{h+1, u}+p_u-\delta_{u j_m}, p_u]}_v[X'_{\mathbf{p}}]_{\Delta}.
		\end{eq}
		
		\noindent$(d)$ If there exists  $m$ such that $i_m=  h$ and $i_{m+1} = h+1$, we have
		\begin{align*}
			& \relax [C]_\emptyset\ast [A]_{\Delta} \\
			= & \sum\limits_{\mathbf{p}\in P_R, a_{h, u} \geq p_u} v^{\kappa_{A}'(\mathbf{p})}\prod^{n}_{u=1}\overline{[a_{h+1, u}+p_u-\delta_{u j_{m+1}}, p_u]}_v[X_{\mathbf{p}}]_{\Delta} \\
			& + \sum\limits_{\substack{\mathbf{p}\in P_R, a_{h, u} \geq p_u\\p_{j_m}>0}} v^{\kappa_{A}'(\mathbf{p},j_{m+1})}(1-v^{-2a_{h+1,j_{m+1}}})\prod^{n}_{u=1}\overline{[a_{h+1, u}+p_u-\delta_{u j_m}, p_u-\delta_{u j_m}]}_v[X_{\mathbf{p}}]_{\Delta_s} \\ 
			& + \sum\limits_{\substack{\mathbf{p}\in P_R, a_{h, u} \geq p_u\\j_{m+1}<t<j_m, p_t>0}} v^{\kappa_{A}'(\mathbf{p},t)}(1-v^{-2a_{h+1,j_{m+1}}})\prod^{n}_{u=1}\overline{[a_{h+1, u}+p_u-\delta_{u t}, p_u-\delta_{u t}]}_v[X'_{\mathbf{p}}]_{\Delta_t},
		\end{align*}
		where $\kappa_{A}'(\mathbf{p},q)=\kappa_{A}'(\mathbf{p})-\sum_{q<j\leq j_m}a_{h+1, j}$ and 
		\begin{eq}
			\Delta_s &= \{(i_1, j_1), \cdots, (i_{m-1}, j_{m-1}), (h+1, j_m), (i_{m+2}, j_{m+2}), \cdots, (i_k, j_k)\},\\
			\Delta_t &=\{(i_1, j_1), \cdots, (i_{m-1}, j_{m-1}), (h, j_m), (h+1, t), (i_{m+2}, j_{m+2}), \cdots, (i_k, j_k)\}.
		\end{eq}
	\end{prop}
	\begin{proof}
		The proof is similar to Proposition \ref{prop 4.1.7}
	\end{proof}
	
	\subsection{Generators}\ 
	We introduce a partial order on $\Xi_{n|n,d}$.
	\begin{Def}
		Let $(A, \Delta), (A', \Delta')\in \Xi_{n|n,d}$, such that $\co(A)=\co(A')$, $\ro(A)=\ro(A')$. If the following conditions are satisfied, then we say $A'\leq A$.
		\begin{enumerate}[ $(a)$ ]
			\item for any $i<j$ in $[1,n]$ we have
			$$
			\sum\limits_{r \leq i, s\geq j}a'_{r, s}\leq \sum\limits_{r \leq i, s\geq j}a_{r, s},
			$$
			\item for any $i>j$ in $[1,n]$ we have
			$$
			\sum\limits_{r\geq i, s \leq j}a'_{r, s}\leq \sum\limits_{r\geq i, s \leq j}a_{r, s}.
			$$
		\end{enumerate}
		Then we say $(A', \Delta')\leq (A,\Delta)$ if $A'\leq A$ or $A'=A$ and $\Delta'\leq \Delta$.
		We say $(A', \Delta')<(A, \Delta)$ if $(A', \Delta')\leq (A, \Delta)$ and $(A',\Delta') \neq  (A,\Delta)$.
	\end{Def}
	
	Let $[A]_{\Delta}+lower\; term $ denote an element in $\mathcal{MS}_{n,d}$ which is equal to $[A]_{\Delta}$ plus 
	an $\mathcal{A}$-linear combination of elements $[A']_{\Delta'}$ with $(A', \Delta')<(A, \Delta)$. 
	
	\begin{lem}
		Assume $t\in [1, n-1]$. For $(D,\{(t, t)\}), (D,\{(t+1, t+1)\}) \in \Xi_{n|n,d}$, where $D=\textrm{diag}(d_1, \cdots, d_n)$ is a diagonal matrix, then 
		\begin{eq}
			\relax [D]_{\{(t+1, t+1)\}}= & v^{d_t}[D'+E_{t+1, t}]_{\emptyset}\ast [D'+E_{t, t}]_{\{(t, t)\}} \ast [D'+E_{t, t+1}]_{\emptyset}\\
			& - v^{d_t}[D-E_{t, t}+E_{t, t+1}]_{\emptyset}\ast [D-E_{t, t}+E_{t+1, t}]_{\emptyset} \ast [D]_{\{(t, t)\}}\\
			& - v^{d_t}[D]_{\{(t, t)\}} \ast [D'+E_{t+1, t}]_{\emptyset} \ast [D'+E_{t, t+1}]_{\emptyset}\\
			& + v^{d_t}[D-E_{t, t}+E_{t, t+1}]_{\emptyset} \ast [D+E_{t+1, t+1}-E_{t, t}]_{\{(t, t)\}} \ast [D-E_{t, t}+E_{t+1, t}]_{\emptyset}\\
			& - v^{2d_t-d_{t+1}-1}\overline{[d_t, 1]}(2[D]_{\{(t, t)\}}+[D]_{\emptyset}) - v^{d_{t+1}-2}\overline{[d_{t+1}, 1]}[D]_{\emptyset},\\
		\end{eq}
		where $D'=diag(d_1, \cdots, d_{t+1}-1, \cdots, d_n)$ is a diagonal matrix.
	\end{lem}
	
	\begin{proof}\
		We do mathematical induction on $t$. 
		If $t=1$, we have 
		\begin{align*}
			& [D'+E_{2, 1}]_{\emptyset}\ast [D'+E_{1, 1}]_{\{(1, 1)\}} \ast [D'+E_{1, 2}]_{\emptyset}\\
			= & [D'+E_{2, 1}]_{\emptyset} \ast ([D'+E_{1, 2}]_{\{(1, 1)\}}+[D'+E_{1, 2}]_{\{(1, 2)\}})\\
			= & v^{-1}[D'-E_{1, 1}+E_{1, 2}+E_{2, 1}]_{\{(1, 1)\}}+[D'-E_{1, 1}+E_{1, 2}+E_{2, 1}]_{\{(2, 1)\}}+v^{-d_1}[D]_{\{(2, 2)\}}\\
			& +  [D'-E_{1, 1}+E_{1, 2}+E_{2, 1}]_{\{(1, 2)\}} + v^{d_2-d_1-1}\overline{[d_2, 1]}(v^{-1}[D]_{\emptyset}+[D]_{\{(1, 1)\}})\\
			= & v^{-d_1}[D]_{\{(2, 2)\}}+(v^{-1}[D'-E_{1, 1}+E_{1, 2}+E_{2, 1}]_{\{(1, 1)\}}+[D'-E_{1, 1}+E_{1, 2}+E_{2, 1}]_{\{(2, 1)\}})\\
			& + v^{d_2-d_1-1}\overline{[d_2, 1]}([D]_{\{(1, 1)\}}+v^{-1}[D]_{\emptyset})+[D'-E_{1, 1}+E_{1, 2}+E_{2, 1}]_{\{(1, 2)\}};\\
			% \end{align*}
			% \noindent\begin{align*}
			& [D-E_{1, 1}+E_{1, 2}]_{\emptyset}\ast [D-E_{1, 1}+E_{2, 1}]_{\emptyset} \ast [D]_{\{(1, 1)\}}\\
			= & [D-E_{1, 1}+E_{1, 2}]_{\emptyset}\ast ( v^{-1}[D-E_{1, 1}+E_{2, 1}]_{\{(1, 1)\}}+ [D-E_{1, 1}+E_{2, 1}]_{\{(2, 1)\}})\\
			= & v^{d_1-d_2-3}\overline{[d_1-1, 1]}[D]_{\{(1, 1)\}}+v^{-1}[D'-E_{1, 1}+E_{1, 2}+E_{2, 1}]_{\{(1, 1)\}}\\
			& + v^{d_1-d_2-1}\overline{[d_1, 1]}[D]_{\emptyset}+ v^{d_1-d_2-1}[D]_{\{(1, 1)\}}+[D'-E_{1, 1}+E_{1, 2}+E_{2, 1}]_{\{(2, 1)\}}\\
			= & v^{d_1-d_2-1}\overline{[d_1, 1]}([D]_{\{(1, 1)\}}+[D]_{\emptyset}) + v^{-1}[D'-E_{1, 1}+E_{1, 2}+E_{2, 1}]_{\{(1, 1)\}}\\
			& +[D'-E_{1, 1}+E_{1, 2}+E_{2, 1}]_{\{(2, 1)\}};\\
			% \end{align*}
			% \begin{align*}
			& [D]_{\{(1, 1)\}} \ast [D'+E_{2, 1}]_{\emptyset} \ast [D'+E_{1, 2}]_{\emptyset}\\
			= & [D'-E_{1, 1}+E_{1, 2}+ E_{2, 1}]_{\{(1, 1)\}}+ [D'-E_{1, 1}+E_{1, 2}+ E_{2, 1}]_{\{(1, 2)\}}+ 
			v^{d_2-d_1-1}\overline{[d_2, 1]}[D']_{\{(1, 1)\}};\\
			& [D-E_{1, 1}+E_{1, 2}]_{\emptyset} \ast [D+E_{2, 2}-E_{1, 1}]_{\{(1, 1)\}} \ast [D-E_{1, 1}+E_{2, 1}]_{\emptyset}\\
			= & v^{d_1-d_2-1}\overline{[d_1, 1]}[D]_{\{(1, 1)\}} + [D'-E_{1, 1}+E_{1, 2}+E_{2, 1}]_{\{(1, 1)\}}, 
		\end{align*}
		then we obtain the identity for $t=1$. Suppose when $t=n-1$, the identity holds.
		Then by doing the same calculation, we have the identity holds when $t=n$. Therefore, by mathematical induction we prove the lemma.
	\end{proof}
	\begin{prop}\label{prop 4.2.1}
		The algebra $\mathcal{MS}_{n,d}$ can be generated by the elements $[B]_{\emptyset}$, $[C]_{\emptyset}$, $[D]_{\emptyset}$ and $[D']_{\{(1,1)\}}$ such that $D$, $D'$, $B-E_{i, i+1}$ and $C-E_{i+1,i}$ are diagonal matrices for some $i\in [1,n-1]$ and $d'_{11}> 0$.
	\end{prop}
	\begin{proof}
		Let $\co(A)_k=\sum^n_{i=1}a_{i,k}$ and $\ro(A)_k=\sum^n_{j=1}a_{k,j}$ where $k\in [1,n]$.
		We show the procedure by an example.
		Assume $n=4$ and $\Delta=\{(1, 4), (2, 3), (4, 1)\}$, i.e., 
		$$
		A=
		\begin{pmatrix}
			a_{11} & a_{12}& a_{13} & \enc{a_{14}}\\
			a_{21} & a_{22} & \enc{a_{23}} &a_{24}\\
			a_{31} & a_{32} & a_{33} &a_{34}\\
			\enc{a_{41}} & a_{42} & a_{43} & a_{44}\\
		\end{pmatrix}.
		$$
		Let
		\begin{flalign*}
			& D=
			\begin{psmallmatrix}
				\senc{\co(A)_1} & 0 & 0 & 0\\
				0 &\co(A)_2 & 0 & 0\\
				0 & 0 & \co(A)_3& 0\\
				0 & 0 & 0 & \co(A)_4\\
			\end{psmallmatrix},&
			& A_1=
			\begin{psmallmatrix}
				\co(A)_1 & 0 & 0 & 0\\
				0 &\co(A)_2 & 0 & 0\\
				0 & 0 &  \co(A)_3-a_{43}& 0\\
				0 & 0 &  a_{43}& \co(A)_4\\
			\end{psmallmatrix}, \\
			& A_2=
			\begin{psmallmatrix}
				\co(A)_1 & 0 & 0 & 0\\
				0 &\co(A)_2-a_{42} & 0 & 0\\
				0 & a_{42} & \co(A)_3-a_{43}& 0\\
				0 & 0 & 0& \co(A)_4+a_{43}\\
			\end{psmallmatrix}, &
			& A_3=
			\begin{psmallmatrix}
				\co(A)_1 & 0 & 0 & 0\\
				0 &\co(A)_2-a_{42} & 0 & 0\\
				0 & 0 & \co(A)_3-a_{43}& 0\\
				0 & 0 & a_{42}& \co(A)_4+a_{43}\\
			\end{psmallmatrix}, \\
			& A_4=
			\begin{psmallmatrix}
				\co(A)_1-a_{41} & 0 & 0 & 0\\
				a_{41} &\co(A)_2-a_{42} & 0 & 0\\
				0 & 0 &  \co(A)_3-a_{43}& 0\\
				0 & 0 & 0& a_4\\
			\end{psmallmatrix},&
			& A_5=
			\begin{psmallmatrix}
				\co(A)_1-a_{41} & 0 & 0 & 0\\
				0 &\co(A)_2-a_{42} & 0 & 0\\
				0 & a_{41} &  \co(A)_3-a_{43}& 0\\
				0 & 0 & 0& a_4\\
			\end{psmallmatrix},\\
			& A_6=
			\begin{psmallmatrix}
				\co(A)_1-a_{41} & 0 & 0 & 0\\
				0&\co(A)_2-a_{42} & 0 & 0\\
				0 & 0 &  \co(A)_3-a_{43}& 0\\
				0 & 0 & a_{41}& a_4\\
			\end{psmallmatrix},&
			&  A_7=
			\begin{psmallmatrix}
				\co(A)_1-a_{41} & 0 & 0 & 0\\
				0 &a_{12}+a_{22} & 0 & 0\\
				0 & a_{32} & \co(A)_3-a_{43}& 0\\
				0 & 0 & 0 & a_4+a_{41}\\
			\end{psmallmatrix},\\
			& A_8=
			\begin{psmallmatrix}
				a_{11}+a_{21} & 0 & 0 & 0\\
				a_{31} &a_{12}+a_{22} & 0 & 0\\
				0 & 0 & \co(A)_3-a_{43}+a_{32}& 0\\
				0 & 0 & 0 & a_4+a_{41}\\
			\end{psmallmatrix},&
			&  A_9=
			\begin{psmallmatrix}
				a_{11}+a_{21} & 0 & 0 & 0\\
				0 &a_{12}+a_{22} & 0 & 0\\
				0 & a_{31} & \co(A)_3-a_{43}+a_{32}& 0\\
				0 & 0 & 0 & a_4+a_{41}\\
			\end{psmallmatrix},\\
			& A_{10}=
			\begin{psmallmatrix}
				a_{11} & 0 & 0 & 0\\
				a_{21} &a_{12}+a_{22} & 0 & 0\\
				0 & 0 & \co(A)_3+\sum^{2}_{j=1}a_{3,j}-a_{43}& 0\\
				0 & 0 & 0 & a_4+a_{41}\\
			\end{psmallmatrix},&
			& A_{11}=
			\begin{psmallmatrix}
				a_{11} & a_{12} & 0 & 0\\
				0 &a_{21}+a_{22} & 0 & 0\\
				0 & 0 &  \co(A)_3+\sum^{2}_{j=1}a_{3,j}-a_{43}& 0\\
				0 & 0 & 0 & a_4+a_{41}\\
			\end{psmallmatrix},\\
			& A_{12}=
			\begin{psmallmatrix}
				a_{11}+a_{12} & 0 & 0 & 0\\
				0 &a_{21}+a_{22} & a_{13}+a_{23} & 0\\
				0 & 0 &  \ro(A)_3-a_{34}& 0\\
				0 & 0 & 0 & a_4+a_{41}\\
			\end{psmallmatrix},&
			&  A_{13}=
			\begin{psmallmatrix}
				a_{11}+_{12} & a_{13}+a_{23}  & 0 & 0\\
				0 &a_{21}+a_{22} & 0 & 0\\
				0 & 0 &  \ro(A)_3-a_{34}& 0\\
				0 & 0 & 0 & a_4+a_{41}\\
			\end{psmallmatrix},\\
			& A_{14}=
			\begin{psmallmatrix}
				\ro(A)_1+a_{23}-a_{14} &  0 & 0 & 0\\
				0 &a_{21}+a_{22} & 0 & 0\\
				0 & 0 &  \ro(A)_3-a_{34}& a_{14}\\
				0 & 0 & 0 & a'_4\\
			\end{psmallmatrix},&
			& A_{15}=
			\begin{psmallmatrix}
				\ro(A)_1+a_{23}-a_{14} &  0 & 0 & 0\\
				0 &a_{21}+a_{22} & a_{14} & 0\\
				0 & 0 &  \ro(A)_3-a_{34}& 0\\
				0 & 0 & 0 & a'_{4}\\
			\end{psmallmatrix},\\
			& A_{16}=
			\begin{psmallmatrix}
				\ro(A)_1+a_{23}-a_{14} &  a_{14} & 0 & 0\\
				0 &a_{21}+a_{22} & 0 & 0\\
				0 & 0 &  \ro(A)_3-a_{34}& 0\\
				0 & 0 & 0 & a'_{4}\\
			\end{psmallmatrix},&
			& D'=
			\begin{psmallmatrix}
				\senc{\ro(A)_1+a_{23}} & 0 & 0 & 0\\
				0 &a_{21}+a_{22} & 0 & 0\\
				0 & 0 &  \ro(A)_3-a_{34}& 0\\
				0 & 0 & 0 & a'_{4}\\
			\end{psmallmatrix},\\
			& A_{17}=
			\begin{psmallmatrix}
				\ro(A)_4 & 0 & 0 & 0\\
				a_{23} &a_{21}+a_{22} & 0 & 0\\
				0 & 0 & \ro(A)_3-a_{34}& 0\\
				0 & 0 & 0 & a'_{4}\\
			\end{psmallmatrix},&
			& A_{18}=
			\begin{psmallmatrix}
				\ro(A)_4 & 0 & 0 & 0\\
				0 &\ro(A)_2-a_{24} & 0 & 0\\
				0 & 0 & \ro(A)_3-a_{34}& a_{24}\\
				0 & 0 & 0 & \ro(A)_4+a_{34}\\
			\end{psmallmatrix},\\
			& A_{19}=
			\begin{psmallmatrix}
				\ro(A)_4 & 0 & 0 & 0\\
				0 &\ro(A)_2-a_{24} & a_{24} & 0\\
				0 & 0 & \ro(A)_3-a_{34}& 0\\
				0 & 0 & 0 & \ro(A)_4+a_{34}\\
			\end{psmallmatrix},&
			& A_{20}=
			\begin{psmallmatrix}
				\ro(A)_4 & 0 & 0 & 0\\
				0 &\ro(A)_2 & 0 & 0\\
				0 & 0 & \ro(A)_3-a_{34}& a_{34}\\
				0 & 0 & 0 & \ro(A)_4\\
			\end{psmallmatrix},\\
		\end{flalign*}
		where $a_4=\co(A)_4+a_{42}+a_{43}$ and $a'_4=\ro(A)_4+a_{34}+a_{24}$.
		Then
		\begin{align*}
			&v^{a_{23}(a_{11}+a_{12}-a_{21}-a_{22})+a_{22}-a_{11}} \prod^{20}_{i=17}[A_i]_{\emptyset}\ast [D']_{\{(1,1)\}} \ast \prod^{16}_{i=1}[A_i]_{\emptyset} \ast [D]_{\{(1,1)\}}= [A]_{\Delta}+lower\, term.      
		\end{align*}
		For any $(A,\Delta)\in \Xi_{n|n,d}$, where $\Delta=\{(i_1, j_i), \cdots, (i_k, j_k)\}$, let 
		\begin{align*}
			\Delta_{u}&=\{(i_p, j_p)\in \Delta \ |\ i_p<j_p\}, \\
			\Delta_{l}&=\{(i_p, j_p)\in \Delta \ |\ i_p\geq j_p\}.
		\end{align*}
		
		The methodology employed to acquire the elements of $A$ closely follows the approach detailed in \cite[Proposition 3.9]{BLM}, with the exception of $a_{i, j}$ for $(i,j)\in \Delta$.
		For $(i_p, j_p)\in \Delta_{l}$, the procedural steps mirrors how we gain $(4, 1)$ in the example. Likewise, the process for obtaining $(i_p,j_p)\in \Delta_{u}$ is analogous to obtaining $(1,4)$ and $(2,3)$ in the example. It is imperative to verify that the coefficient of the leading term is a power of $v$.
		Assume $h\in [1, n-1]$ and $\in [1,n]$.
		
		$(a)$ Consider $(A,\Delta)\in \Xi_{n|n,d}$ where $\Delta=\{(i_1, j_1), \cdots, (h, j_m), \cdots, (i_k,j_k)\}$ and $a_{h+1, t}=0$ for $j_{m+1}<t\leq j_m$. For $j_{m+1}<t\leq j_m$, let $(M_t, \Delta_t)\in \Xi_{n|n,d}$ such that $m_{h+1, t}=a_{h, t}$ , $m_{h,t}=0$ and $m_{i, j}=a_{i,j}$ for other $i,j$, where 
		$$\Delta_t=\left\{\
		\begin{array}{cc}
			\{(i_1, j_1), \cdots, (i_{m-1}, j_{m-1}), (h+1, j_m), (i_{m+1}, j_{m+1}), \cdots, (i_k, j_k)\}   &  t=j_m, \\
			\{(i_1, j_1), \cdots, (i_{m-1}, j_{m-1}), (h, j_m), (h+1, t), (i_{m+1}, j_{m+1}), \cdots, (i_k, j_k)\   & \text{otherwise}.
		\end{array}
		\right.$$ 
		Let $(C_t,\emptyset)\in \Xi_{n|n,d}$ be defined by $\co(C_t)=\ro(A)$ and $C-a_{h,t}E_{h+1,h}$ is a diagonal matrix. Due to \cite[Lemma 3.7] {BLM} and Proposition \ref{prop 4.1.8} $(b)$, we have the following identity.
		$$
		[C_t]_{\emptyset}\ast [A]_{\Delta}= v^{a}[M_t]_{\Delta_t}+lower\, term
		,$$
		where $a=a_{h, t}(\sum_{j\leq t}a_{h+1, j}-\sum_{j< t}a_{h, j})-\sum_{j_{m+1}<j\leq t}a_{h+1,j}$.
		
		$(b)$ For any fixed integer $r>h$, let $(A,\Delta)\in \Xi_{n|n,d}$ be defined by $a_{i,j}=0$ for $i\leq h, j>r$ and $a_{h, r}>0$, where $\Delta=\{(i_1, j_1), \dotsc, (i_k,j_k)\}$ such that $i_1>h$.
		Take $(D, \{(h,h)\})\in \Xi_{n|n,d}$ such that $\co(D)=\ro(A)$ and $D$ is a diagonal matrix. Then due to Proposition \ref{prop 4.1.6} $(a)$, we have the following identity.
		$$
		[D]_{\{(h, h)\}} \ast [A]_{\Delta} = v^{a'}[A]_{\Delta\cup \{(h, r)\}}+ lower\, term,
		$$
		where $a'=\sum_{i\leq h, j\leq j_1}a_{i, j}-\sum_{i\leq h, j>j_1}a_{i, j}$.
		Therefore, we conclude that the coefficient of the leading term in our procedure is a power of $v$.
	\end{proof}
	
	For $i\in [1, n-1]$ and $a\in [1,n]$, we define
	$$E_i=\sum[B]_{\emptyset},\ F_i=\sum[C]_{\emptyset},\ H_a^{\pm}=\sum v^{\mp d_{a, a}}[D]_{\emptyset},\  L=\sum v^{-2d_{1,1}}[D]_{\emptyset}+\sum v^{-d'_{1,1}}[D']_{\{(1,1)\}},$$
	where $(B,\emptyset)$, $(C,\emptyset)$ $(D,\emptyset)$ and $(D',\{(1,1)\})$ run over all matrices in $\Xi_{n|n,d}$ such that $D$, $D'$, $B-E_{i,i+1}$, and $C-E_{i+1,i}$ are diagonal matrices, respectively. Then we have the following corollary.
	
	\begin{cor}\label{cor 4.2.1}
		The elements $E_i, F_i, H^{\pm}_{a}$ and $L$ are the generators of $\mathbb{Q}(v)\otimes_{\mathcal{A}} \mathcal{MS}_{n,d}$, where $i\in [1, n-1]$ and $a\in [1, n]$.
	\end{cor}
	\begin{proof}
		% The corollary can be shown by a standard Vandermonde-determinant-type argument.
		Let $T$ be the subalgebra of $\mathbb{Q}(v)\otimes_{\mathcal{A}} \mathcal{MS}_{n,d}$ generated by $E_i, F_i, H^{\pm}_{a}$ and $L$, where $i\in [1, n-1]$ and $a\in [1, n]$. For any diaganol matrix $D=diag(d_1,\cdots, d_n)$ such that $(D, \emptyset)\in \Xi_{n|n,d}$.
		Firstly, we show that $[D]_\emptyset \in T$. For any $\mathbf{\mu}=(\mu_1,\cdots, \mu_n)\in \mathbb{Z}^{n}$, we have
		$$
		\prod\limits_{a \in [1,n]}K^{\mu_a}_a=\sum\limits_{([D]_\emptyset)\in \Xi_{n|n, d}}v^{-\mathbf{\mu}\cdot (d_1, \cdots, d_n)}[D]_{\emptyset},
		$$
		where $\mathbf{\mu}\cdot (d_1, \cdots, d_n)$ is the standard inner product of two vectors. Then we only need to show that there is $\mathbf{\mu}$ such that $\mathbf{\mu}\cdot (d_1, \cdots, d_n)\neq \mathbf{\mu}\cdot(d'_1, \cdots, d'_n)$ for any $D\neq D'$. This is because once we have such $\mathbf{\mu}$, we can form a linear system from the identity above by considering the vectors $c\mathbf{\mu}$ for $c\in \mathbb{N}$ with
		$$1=\sum_{D}[D]_\emptyset,$$
		where $D$ runs over all diagonal matrices in $Mat_{n\times n}(\mathbb{N})$ such that $(D,\emptyset)\in \Xi_{n|n,d}$.
		It is clear that the coefficient matrix of this linear system is the Vandermonde matrix and it is invertible due to the choice of $\mathbf{\mu}$ and $c$. Therefore $[D]_\emptyset \in T$.
		
		The existence of such $\mathbf{\mu}$ is equivalent with finding a $\mathbf{\mu}$ such that $\mathbf{\mu}\cdot (d_1-d'_1, \cdots, d_n-d'_n)$ for any $D\neq D'$. Since $\{[D]_\emptyset\ |\ \ (D, \emptyset)\in \Xi_{n|n, d}\}$ is a finite set, the collection of vectors $(d_1-d'_1, \cdots, d_n-d'_n)$ is finite, let $A$ be the matrix that the row vectors runs over this collection. Then $\mathbf{\mu}$ is the vector that not a solution of linear system $Ax=0$ and any element of $\mathbf{\mu}$ is $0$.
		
		The proofs of generators $[B]_\emptyset$, $[C]_\emptyset$ and $[D]_{\{(1,1)\}}\in T$ are now follows from $E_i\ast [D]_\emptyset$, $F_i\ast [D]_\emptyset$ and $L\ast [D]_\emptyset$ for all $(D, \emptyset)\in \Xi_{n|n, d}$.
		% , where $B-E_{i, i+1}$, $C-E_{i+1,i}$ and $\mathfrak{d}$ are diagonal matrices for some $i\in [1,n-1]$. 
		Therefore $T=\mathbb{Q}(v)\otimes_{\mathcal{A}} \mathcal{MS}_{n,d}$.
	\end{proof}
	\begin{prop}\label{prop 4.2.5}
		For $i\in [1, n-1]$ and $a\in [1, n]$, we have the following relations in $\mathbb{Q}(v)\otimes_{\mathcal{A}} \mathcal{MS}_{n,d}$. 
		\begin{enumerate}[ $(a)$ ]
			\item $H_aH_{-a}=1$, $H_aH_b=H_bH_a$;
			\item $E_{i}^{d+1}=0$, $F_{i}^{d+1}=0$;
			\item $E_{i}^2E_{i+1}+E_{i+1}E_{i}^2=(v+v^{-1})E_{i}E_{i+1}E_{i}$;
			\item $E_{i+1}^2E_{i}+E_{i}E_{i+1}^2=(v+v^{-1})E_{i+1}E_{i}E_{i+1}$;
			\item $F_{i}^2F_{i+1}+F_{i+1}F_{i}^2=(v+v^{-1})F_{i}F_{i+1}F_{i}$;
			\item $F_{i+1}^2F_{i}+F_{i}F_{i+1}^2=(v+v^{-1})F_{i+1}F_{i}F_{i+1}$;
			\item $H_a E_{i}=v^{\delta_{ah}-\delta_{a, h+1}}E_{i} H_a$;
			\item $H_a F_{i}=v^{-\delta_{ah}+\delta_{a, h+1}}F_{i} H_a$;
			\item $E_{i}F_{j}-F_{j}E_{i}=\delta_{i j}\frac{H_iH_{i+1}^{-1}-H^{-1}_iH_{i+1}}{v-v^{-1}}$;
			\item $H_a L=L H_a$;
			\item $L^2=L$;
			\item $L E_{i}=L E_{i} L$;
			\item $L F_{i}=L F_{i} L$;
			\item $(v+v^{-1}) E_{i} L E_{i}=v^{-1}E_{i}^{2}L+vLE_{i}^{2}$;
			\item $(v+v^{-1}) F_{i} L F_{i}=vF_{i}^2 L+v^{-1}LF_{i}^{2}$.
		\end{enumerate} 
	\end{prop}
	\begin{proof}
		The relations $(a)-(i)$ are hold due to \cite{BLM}. One can check the remaining relations by using Proposition \ref{prop 4.1.4}, Proposition \ref{prop 4.1.5} and Proposition \ref{prop 4.1.6}. We calculate $(m)$ for $i=1$ and the others are similar. 
		Let $[B]_{\emptyset}\in \Xi_{n|n,d}$ such that $B-E_{1, 2}=diag(b_1, \cdots, b_n)$ is a diagonal matrix, we have
		\begin{align*}
			[B]^{2}_{\emptyset}= \frac{v^{2}-v^{-2}}{v-v^{-1}}[B+E_{1,2}]_{\emptyset}.
		\end{align*}
		Then we have 
		\begin{align*}
			&\frac{L\ast [B]^{2}_{\emptyset}}{(v+v^{-1})}\\
			=& \sum\limits_{Z}v^{-2z_1}[Z]_{\emptyset}\ast [B+E_{1,2}]_{\emptyset}+\sum\limits_{Z',z'_1>0}v^{-z'_1}[Z']_{\{(1,1)\}}\ast [B+E_{1,2}]_{\emptyset}\\
			= & v^{-4-2b_1}[B+E_{1,2}]_{\emptyset}+v^{-4-b_1}[B+E_{1,2}]_{\{(1,1)\}}+v^{-2-b_1}[B+E_{1,2}]_{\{(1,2)\}}.
		\end{align*}
		Besides, 
		\begin{align*}
			&[B]_{\emptyset}\ast L\\
			=&[B]_{\emptyset}\ast \sum\limits_{Z}v^{-2z_1}[Z]_{\emptyset}+[B]_{\emptyset}\ast \sum\limits_{Z',z'_1>0}v^{-z_1}[Z]_{\{(1,1)\}}\\
			= &v^{-2b_1}[B]_{\emptyset}+v^{-b_1}[B]_{\{(1,1)\}}.
		\end{align*}
		So 
		\begin{align*}
			\frac{[B]^{2}_{\emptyset}\ast L}{v+v^{-1}}=v^{-2b_1}[B+E_{1,2}]_{\emptyset}+v^{-b_1}[B+E_{1,2}]_{\{(1,1)\}}.
		\end{align*}
		Finally, we have
		\begin{align*}
			& L \ast [B]_{\emptyset}\\
			=&\sum\limits_{Z, z_1>0}v^{-2z_1}[Z]_{\emptyset}\ast [B]_{\emptyset}\\
			&+\sum\limits_{Z',z'_1>0}v^{-z'_1}[Z']_{\{(1,1)\}}\ast [B]_{\emptyset}\\
			= &v^{-2b_1-2}[B]_{\emptyset}+v^{-b_1-2}[B]_{\{(1,1)\}}+v^{-b_1-1}[B]_{\{(1,2)\}}.
		\end{align*}
		Hence
		\begin{align*}
			&[B]_{\emptyset}\ast L \ast [B]_{\emptyset}\\
			= & v^{-2b_1-2}[B]_{\emptyset}\ast [B]_{\emptyset}+v^{-b_1-2}[B]_{\emptyset}\ast [B]_{\{(1,1)\}}+v^{-b_1-1}[B]_{\emptyset}\ast [B]_{\{(1,2)\}}\\
			= &v^{-2b_1-2}(v+v^{-1})[B+E_{1,2}]_{\emptyset}+v^{-b_1-2}(v+v^{-1})[B+E_{1,2}]_{\{(1,1)\}}+v^{-b_1-1}[B+E_{1,2}]_{\{(1,2)\}}\\
			= & v^{-2b_1-1}[B+E_{1,2}]_{\emptyset}+v^{-b_1-1}[B+E_{1,2}]_{\{(1,1)\}}\\
			& + v^{-2b_1-3}[B+E_{1,2}]_{\emptyset}+v^{-b_1-3}[B+E_{1,2}]_{\{(1,1)\}}+v^{-b_1-1}[B+E_{1,2}]_{\{(1,2)\}}\\
			= & v^{-1}\frac{[B]^{2}_{\emptyset}\ast L}{v+v^{-1}}+v\frac{L\ast [B]^{2}_{\emptyset}}{(v+v^{-1})}.
		\end{align*}
		Then by linearity, we conclude that the relation $(m)$ holds.
	\end{proof}
	Upon comparing the generators and the relations of the algebra $\mathbf{MU}$ as defined in Definition \ref{def 2.1.2} with relations of algebra $\mathbb{Q}(v)\otimes_{\mathcal{A}}\mathcal{MS}_{n, d}$ in Proposition \ref{prop 4.2.5}, we establish the following proposition.
	\begin{prop}\label{prop 4.2.6}
		There is a surjective algebra homomorphism $\mathbf{MU} \to \mathbb{Q}(v)\otimes_{\mathcal{A}}\mathcal{MS}_{n, d}$ which sends $E_i\mapsto E_i$, $F_i\mapsto F_i$, $L\mapsto L$ and $H^{\pm}_a\mapsto H^{\pm}_a$ for any $i\in [1,n-1]$ and $a\in [1,n]$.
	\end{prop}
	\subsection{$\mathcal{MS}_{n,d}$-action on $\mathcal{MV}$}
	Recall the $\mathbf{r}'_p$ and $\mathbf{r}''_p$ in 2.2. 
	
	\begin{prop}\label{prop 4.3.1}
		Assume $i \in [1, n-1]$ and $a\in [1,n]$. The left $\mathcal{MS}_{n,d}$-action on $\mathcal{MV}$ is given as follows.
		
		\noindent$(a)$ For $a\in [1,n]$, we have 
		\begin{eq}
			\relax H^{\pm}_a \ast e_{\mathbf{r}, \mathbf{j}}=v^{\mp \sum_{j\in [1,d]}\delta_{ar_{j}}}e_{\mathbf{r}, \mathbf{j}}.
		\end{eq}
		
		\noindent$(b)$ If $\mathbf{j}=\emptyset$, we have 
		\begin{eq}
			\relax L\ast e_{\mathbf{r}, \emptyset}=& 
			v^{-\sum_{j \in [1, d]}\delta_{1 r_{j}}}(e_{\mathbf{r},\emptyset} +\sum_{p \in [1, d], r_p>0}e_{\mathbf{r}, \{p\}}).
		\end{eq}
		
		\noindent$(c)$ If $r_{j_1}=1$, then 
		\begin{eq}
			\relax L \ast e_{\mathbf{r}, \mathbf{j}}=& 
			v^{-2\sum_{j \in [1, d]}\delta_{1 r_{j}}+2\sum_{j<j_1}\delta_{1 r_j}}(v^{2}-1)e_{\mathbf{r}, \mathbf{j}\setminus{\{j_1\}}}\\
			& + v^{-2\sum_{j \in [1, d]}\delta_{1 r_{j}}}\sum_{p>j_1, \mathbf{r}_p>0}v^{2\sum_{j<j_1}\delta_{1 r_j}}(v^{2}-1)e_{\mathbf{r},\mathbf{j}_p} \\
			& + v^{-2\sum_{j \in [1, d]}\delta_{1 r_{j}}}(v^{2\sum_{j \leq j_1}\delta_{1 r_j}}-1)e_{\mathbf{r}, \mathbf{j}}, 
		\end{eq}
		where $\mathbf{j}_p=(p, j_2, \cdots, j_k)$.
		
		\noindent$(d)$ If $r_{j_1}>1$, we have
		\begin{eq}
			\relax L \ast e_{\mathbf{r}, \mathbf{j}}= &
			v^{-\sum_{j \in [1, d]}\delta_{1 r_{j}}}v^{2\sum_{j \leq j_1}\delta_{1 r_j}}e_{\mathbf{r}, \mathbf{j}} \\
			& +v^{-\sum_{j \in [1, d]}\delta_{1 r_{j}}}\sum_{p>j_1, \mathbf{r}_{p}>0}v^{2\sum_{j \leq j_1}\delta_{1 r_j}}e_{\mathbf{r}, \mathbf{j}_p},
		\end{eq}
		where $\mathbf{j}_p=(p, j_1, j_2, \cdots, j_k)$.
		
		\noindent$(e)$ If for any $j_m$, $r_{j_m}\neq h, h+1$, we have
		\begin{eq}
			\relax E_i \ast e_{\mathbf{r}, \mathbf{j}}& = v^{-\sum_{j \in [1, d]}\delta_{h r_j}}\sum\limits_{p \in [1, d], p \neq j_m, r_p=h+1}v^{2\sum_{j>p}\delta_{h r_{j}}}e_{\mathbf{r}''_p, \mathbf{j}};\\
			\relax F_i \ast e_{\mathbf{r}, \mathbf{j}} & = v^{-\sum_{j \in [1, d]}\delta_{h+1, r_j}}\sum\limits_{p \in [1, d], r_p=h}v^{2\sum_{j<p}\delta_{h+1, r_{j}}}e_{\mathbf{r}'_p, \mathbf{j}}.\\
		\end{eq}
		
		\noindent$(f)$ If there is $j_m\in \mathbf{j}$ such that $r_{j_m}=h$ and $r_{j_{m+1}}\neq h+1$, then
		\begin{align*}
			\relax E_i \ast e_{\mathbf{r}, \mathbf{j}} = & v^{-\sum_{j \in [1, d]}\delta_{h r_j}}\sum\limits_{p \in [1, j_{m+1}]\cup [j_m+1, d], r_p=h+1}v^{2\sum_{j>p}\delta_{h r_{j}}}e_{\mathbf{r}''_p, \mathbf{j}}\\
			& + v^{-\sum_{j \in [1, d]}\delta_{h r_j}}\sum\limits_{p \in [j_{m+1}+1, j_{m}-1], r_p=h+1}v^{2\sum_{j>p}\delta_{h, r_{j}}-1}e_{\mathbf{r}''_p, \mathbf{j}};\\
			%     \end{align*}
			% \begin{align*}
			\relax F_i \ast e_{\mathbf{r}, \mathbf{j}} = & v^{-\sum_{j \in [1, d]}\delta_{h+1, r_j}}\sum\limits_{p \in [1, d], p \neq j_m,  r_p=h}v^{2\sum_{j<p}\delta_{h+1, r_{j}}}e_{\mathbf{r}'_p, \mathbf{j}}\\
			& + v^{-\sum_{j \in [1, d]}\delta_{h+1, r_j}}\sum\limits_{p= j_m, r_{j_m}=h}v^{2\sum_{j<p}\delta_{h+1, r_{j}}}e_{\mathbf{r}'_p, \mathbf{j}\setminus(j_m)}\\
			& + v^{-\sum_{j \in [1, d]}\delta_{h+1, r_j}}v^{2\sum_{j\leq j_{m+1}}\delta_{ir_{j}}}e_{\mathbf{r}'_{j_m}, \mathbf{j}}\\
			& + v^{-\sum_{j \in [1, d]}\delta_{h+1, r_j}}\sum\limits_{t \in [j_{m+1}+1, j_m-1], r_t=h} v^{2\sum_{j\leq j_{m+1}}\delta_{ir_{j}}}e_{\mathbf{r}'_t, \mathbf{j}_t},
		\end{align*}
		where $\mathbf{j}_t=(j_1, \cdots, j_m, t, j_{m+1}, \cdots, j_k)$.
		
		\noindent$(g)$ If there is $j_m\in \mathbf{j}$ such that $r_{j_m}=h+1$ and $r_{j_{m-1}}\neq h$, then
		\begin{eq}
			\relax E_i\ast e_{\mathbf{r}, \mathbf{j}} = & v^{-\sum_{j \in [1, d]}\delta_{h r_j}}\sum\limits_{p \in [1, d], p \neq j_m, r_p=h+1}v^{2\sum_{j>p}\delta_{h r_{j}}}e_{\mathbf{r}''_p, \mathbf{j}}\\
			& + v^{-\sum_{j \in [1, d]}\delta_{h r_j}}\sum\limits_{p=j_m}v^{2\sum_{j>p}\delta_{h r_{j}}}e_{\mathbf{r}''_p, \mathbf{j}\setminus(j_m)}\\
			& + v^{-\sum_{j \in [1, d]}\delta_{h r_j}}\sum\limits_{p = j_m}v^{2\sum_{j>p}\delta_{h r_{j}}}e_{\mathbf{r}''_p, \mathbf{j}}\\
			& + v^{-\sum_{j \in [1, d]}\delta_{h r_j}}\sum\limits_{\substack{p = j_m, \\t \in [j_{m+1}+1, j_m-1], r_t=h+1}} v^{2\sum_{j>p}\delta_{h r_{j}}}e_{\mathbf{r}''_{j_m}, \mathbf{j}_t},\\
		\end{eq}
		where $\mathbf{j}_t=(j_1, \cdots, j_m, t, j_{m+1}, \cdots, j_k)$;
		\begin{eq}
			\relax F_i \ast e_{\mathbf{r}, \mathbf{j}} = & v^{-\sum_{j \in [1, d]}\delta_{h+1, r_j}}\sum\limits_{p \neq j_m, r_p=h}v^{2\sum_{j<p}\delta_{h+1, r_{j}}}e_{\mathbf{r}'_p, \mathbf{j}}.\\
		\end{eq}
		
		\noindent$(h)$ If there are $j_{m}, j_{m+1}\in \mathbf{j}$ such that $r_{j_m}=h$ and $r_{j_{m+1}}=h+1$, 
		\begin{align*}
			\relax E_i \ast e_{\mathbf{r}, \mathbf{j}} = & v^{-\sum_{j \in [1, d]}\delta_{h r_j}}\sum\limits_{p \in [1, j_{m+1}]\cup [j_m+1, d], r_p=h+1}v^{2\sum_{j>p}\delta_{h r_{j}}}e_{\mathbf{r}''_p, \mathbf{j}}\\
			& + v^{-\sum_{j \in [1, d]}\delta_{h r_j}}\sum\limits_{p \in [j_{m+1}+1, j_{m}-1], r_p=h+1}v^{2\sum_{j>p}\delta_{h r_{j}}-1}e_{\mathbf{r}''_p, \mathbf{j}}\\
			& + v^{-\sum_{j \in [1, d]} \delta_{h r_j}}\sum\limits_{p=j_{m+1}}v^{2\sum_{j<j_{m+1}}\delta_{h+1, r_{j}}} (1-v^{-2})e_{\mathbf{r}''_p, \mathbf{j}\setminus{j_{m+1}}}\\
			& + v^{-\sum_{j \in [1, d]} \delta_{h r_j}}\sum\limits_{\substack{p = j_{m+1},\\j_{m+2}<t<j_{m+1}, r_{t}=h+1}} v^{2\sum_{j<j_{m+1}}\delta_{h+1, r_{j}}} (1-v^{-2})e_{\mathbf{r}''_p, \mathbf{j}_t},
		\end{align*}
		where $\mathbf{j}_t=(j_1, \cdots, j_{m+2}, t, j_{m}, \cdots, j_k)$;
		\begin{eq}
			\relax F_i \ast e_{\mathbf{r}, \mathbf{j}} = & v^{-\sum_{j \in [1, d]}\delta_{h+1, r_j}}\sum\limits_{p \neq j_m+1,  r_p=h}v^{2\sum_{j<p}\delta_{h+1, r_{j}}} e_{\mathbf{r}'_p, \mathbf{j}}\\
			& + v^{-\sum_{j \in [1, d]} \delta_{h+1, r_j}}\sum\limits_{p = j_m}v^{2\sum_{j<j_{m+1}}\delta_{h+1, r_{j}}}(v^2-1)e_{\mathbf{r}'_p, \mathbf{j}_{j_m}}\\
			& + v^{-\sum_{j \in [1, d]} \delta_{h+1, r_j}}\sum\limits_{j_{m+1}<t<j_m, r_t=h} v^{2\sum_{j<j_{m+1}}\delta_{h+1, r_{j}}} (v^2-1)e_{\mathbf{r}'_t, \mathbf{j}_t},\\
		\end{eq}
		where $\mathbf{j}_t=(j_1, \cdots, j_{m+2}, t, j_{m}, \cdots, j_k)$.
	\end{prop}
	
	\begin{proof}
		The proposition follows from Proposition \ref{prop 4.1.4}, Proposition \ref{prop 4.1.5} and Proposition \ref{prop 4.1.6}.
	\end{proof}
	
	\subsection{$\mathcal{MH}$-action on $\mathcal{MV}$}
	Recall $\sigma\mathbf{r}$ and $\sigma \mathbf{j}$ in 2.2, where $\sigma$ is an element in symmetric group $S_d$.
	\begin{prop}\label{prop 4.4.1}
		Assume $j\in [1,d-1]$. Let $\mathbf{r}={r_1,\cdots, r_d}$ $\mathbf{j}=(j_1, \cdots, j_k)\in J_{\mathbf{r}}$. The $\mathcal{MH}$-action is given as follows. 
		
		\noindent$(a)$ If $j,j+1\notin \mathbf{j}$, then 
		\begin{align*}
			e_{\mathbf{r}, \mathbf{j}} \tau_{s_j,\emptyset}=\left\{\begin{array}{ll}
				e_{s_j\mathbf{r}, \mathbf{j}} & \text { if } r_j<r_{j+1}, \\
				v^2 e_{\mathbf{r},\mathbf{j}} & \text { if } r_j=r_{j+1}, \\
				\left(v^2-1\right) e_{\mathbf{r}, \mathbf{j}}+v^2 e_{s_j\mathbf{r},\mathbf{j}} & \text { if } r_j>r_{j+1}.
			\end{array}\right.
		\end{align*}
		
		\noindent$(b)$ If $j\in \mathbf{j}$ and $j+1\notin \mathbf{j}$, then
		\begin{align*}
			e_{\mathbf{r}, \mathbf{j}} \tau_{s_j,\emptyset}=\left\{\begin{array}{ll}
				e_{s_j\mathbf{r}, s_j\mathbf{j}}+e_{s_j\mathbf{r}, \mathbf{j}\setminus (j)} & \text { if } r_j<r_{j+1}, \\
				e_{\mathbf{r}, s_j\mathbf{j}} & \text { if } r_j=r_{j+1}, \\
				e_{\mathbf{r}, \mathbf{j}\cup \{j+1\}}+e_{s_j\mathbf{r}, s_j\mathbf{j}}+ e_{s_j\mathbf{r}, \mathbf{j}\setminus (j)} & \text { if } r_j>r_{j+1}.
			\end{array}\right. 
		\end{align*}
		
		\noindent$(c)$ If $j\notin \mathbf{j}$ and $j+1\in \mathbf{j}$, then
		\begin{align*}
			e_{\mathbf{r}, \mathbf{j}} \tau_{s_j,\emptyset}=\left\{\begin{array}{ll}
				2e_{s_j\mathbf{r}, \mathbf{j}\setminus (j+1)}+e_{s_j\mathbf{r}, s_j\mathbf{j}}+e_{s_j\mathbf{r}, \mathbf{j}\cup\{j\}}+e_{\mathbf{r}, \mathbf{j}\setminus (j+1)} & \text { if } r_j<r_{j+1}, \\
				v^2e_{\mathbf{r}, s_j\mathbf{j}}+v^2e_{\mathbf{r}, \mathbf{j}\setminus (j+1)}+(v^2-1)e_{\mathbf{r}, \mathbf{j}} & \text { if } r_j=r_{j+1}, \\
				(v^2-1)e_{\mathbf{r}, \mathbf{j}}+ v^2(2e_{\mathbf{r}, \mathbf{j}\setminus (j+1)}+e_{s_j\mathbf{r}, s_j\mathbf{j}}+e_{s_j\mathbf{r}, \mathbf{j}\setminus (j+1)}) & \text { if } r_j>r_{j+1}.
			\end{array}\right.
		\end{align*}
		
		\noindent$(d)$ If $j, j+1\in \mathbf{j}$, then
		\begin{align*}
			e_{\mathbf{r}, \mathbf{j}} \tau_{s_j,\emptyset} & =
			2(v^2-1)e_{\mathbf{r}, \mathbf{j}}+(2v^2-1)e_{\mathbf{r}, \mathbf{j}\setminus (j+1)}\\
			& + (2v^2-1)e_{s_j\mathbf{r}, \mathbf{j}\setminus(j, j+1)}+(2v^2-1)e_{s_j\mathbf{r}, \mathbf{j}\setminus (j)}.
		\end{align*}
		
		\noindent$(e)$ 
		\begin{align*}
			e_{\mathbf{r}, \mathbf{j}} \tau_{id,\{1\}}=\left\{\begin{array}{ll}
				(v^{2(\sum_{i\leq i_k} \delta_{r_{1} i})}-1)e_{\mathbf{r}, \mathbf{j}}+ \sum_{i'>i_k}\delta_{r_1 i'}e_{\mathbf{r}, \mathbf{j}'},  &  \text{ if } j_k>1,\\
				(v^{2}-2)e_{\mathbf{r}, \mathbf{j}}  &  \text{ if } j_k=1,
			\end{array}\right.
		\end{align*}
		where $\mathbf{j}'=(j_1, j_2, \cdots, j_k, 1)$.
	\end{prop}
	
	\begin{proof}
		The algebra $\mathcal{MH}$ is naturally embedded in $\mathcal{MS}_{d,d}$ as a subalgebra. 
		Then the action of $\mathcal{MH}$ on the space $\mathcal{MV}$ can be induced from the $\mathcal{MS}_{d,d}$-action on $\mathcal{MV}$.
		Specifically, for $1 \leq i \leq d-1$, we have the following identity for $(B_i,\emptyset), (C_i,\emptyset)\in \Xi_{d|d,d}$ such that $B_i-E_{i, i+1}$ and $C_i-E_{i+1, i}$ are diagonal matrices and $\co(B)=\ro(C)$.
		\begin{align*}
			\tau_{s_i,\emptyset}=e_{C_i, \emptyset} \ast e_{B_i, \emptyset}-e_{I, \emptyset},
		\end{align*}
		where $I$ is the identity matrix.
		
		By \cite[Remark2.8]{R18}, we have an anti-automorphism 
		$$\mathcal{MS}_{d,d}\to \mathcal{MS}_{d,d},\ \ e_{A,\Delta}\mapsto e_{^tA, ^t\Delta},$$ 
		where $^tA$ denotes the transpose matrix and $^t\Delta=\{(j_k, i_k), \cdots, (j_1, i_1)\}$, then 
		$$^t(e_{A,\Delta}\ast e_{A', \Delta'})=e_{^tA', ^t\Delta'} \ast e_{^tA, ^t\Delta}.$$ 
		Therefore for $1 \leq j \leq d-1$ and $(A, \Delta)\in \Xi^1$, we have 
		\begin{align*}
			e_{A, \Delta} \tau_{s_j, \emptyset}=& e_{A, \Delta} \ast e_{C_j, \emptyset} \ast e_{B_j, \emptyset}-e_{A, \Delta}\ast e_{I, \emptyset}\\
			& = {}^t(e_{^tB_j, \emptyset} \ast e_{^tC_j, \emptyset}\ast e_{^tA, ^t\Delta})-{}^t(e_{I, \emptyset}\ast e_{^tA, ^t\Delta})\\
			& = {}^t(e_{C_j, \emptyset} \ast e_{B_j, \emptyset}\ast e_{^tA, ^t\Delta})-e_{A, \Delta}.\\
		\end{align*}
		Then by Proposition \ref{prop 4.1.1} and Proposition \ref{prop 4.1.2}, $(a)-(d)$ hold.
		
		For $\tau_{id,\{1\}}$, we have $e_{A,\Delta}\tau_{id,\{1\}}={}^t(e_{I,\{(1,1)\}}\ast e_{^tA, ^t\Delta})$, then by Proposition \ref{prop 4.1.3} the identities hold.
	\end{proof}
	\section{Algebra $\mathcal{K}$}
	\subsection{Stabilization}\
	Let $\tilde{\Xi}=\{(A, \Delta)\ |\ A=(a_{ij})\in Mat_{n\times n}(\mathbb{Z}), a_{i j}\geq 0 \text{ for } i\neq j\}$ where $\Delta$ is similar with that in Definition \ref{def 3.1.1}, the difference is we allow $(i,i)\in \Delta$ for $a_{ii}\leq 0$.
	
	For any $(A, \Delta)\in \tilde{\Xi}$ and an integer $p$, let $(_{p}A, \Delta)=(A+pI, \Delta)$ where $I$ is the identity matrix. 
	Clearly, if $p$ is large enough, we have $({}_pA, \Delta)\in \Xi_{n|n}$.
	
	Let $v'$ be an indeterminate and  $\mathfrak{R} = \mathbb{Q}(v)[v', v'^{-1}]$, we have the following proposition. 
	\begin{prop}\label{prop 5.1.1}
		For $(A_1, \Delta_1), (A_2, \Delta_2), \cdots, (A_k, \Delta_k)\in \tilde{\Xi}$ satisfied $\co(A_i)=\ro(A_{i+1})$ for $i=1, \cdots, k-1$, where $k\geq 2$.
		There exist $(Z_1, \Delta'_1), \cdots, (Z_m, \Delta'_m)\in \tilde{\Xi}$, elements $G_{\Delta'_i}(v, v')\in \mathfrak{R}$
		and an integer $p_{0}\geq 0$ such that for all $p \geq p_0$
		$$
		[_{p}A_{1}]_{\Delta_1} \ast [_{p}A_{2}]_{\Delta_2} \ast \cdots \ast [_{p}A_{k}]_{\Delta_k} =
		\sum_{i=1}^{m} G_{\Delta'_i}(v, v^{-p})[_{p}Z_{i}]_{\Delta'_i}.
		$$
	\end{prop}
	\begin{proof}\ 
		The proof is similar to the one of \cite[Proposition 4.2]{BLM}, we need $p$ large enough such that $a_{ii}+p> 0$ for any $(i,i)\in \Delta$. We first check the stability of the $[_pD]_{\{(1,1)\}} \ast [_pA]_{\Delta}$ where $\Delta=\{(i_1,j_1), \cdots, (i_k,j_k)\}$. 
		By Proposition \ref{prop 4.1.6}, we have the following formulas.
		
		\noindent$(a)$ If $i_1>1$, we have 
		\begin{align*}
			\relax  [_{p}D]_{\{(1,1)\}}\ast [_{p}A]_{\Delta}= & v^{\sum_{j \leq j_1}a_{1, j}+p-\sum_{j> j_1}a_{1, j}}(1-v^{-2(\sum_{j \leq j_1}a_{1, j}+p)})[_{p}A]_{\Delta}\\
			&+ \sum_{j'>j_1,a_{1,j'}>0}v^{\sum_{j \leq j_1}a_{1, j}+p-\sum_{j>j'}a_{1, j}}[_{p} A]_{\Delta_{j'}},
		\end{align*}
		where $\Delta_{j'}=\{(1,j'), (i_1, j_1), \cdots, (i_k, j_k)\}$.
		Thus, we set $G_{\Delta'}(v,v')$ in this case as
		\begin{align*}
			G_{\Delta}(v,v') & = v^{\sum_{1<j\leq j_1}a_{1, j}-\sum_{j> j_1}a_{1, j}}(1-v^{-2\sum_{j\leq j_1}a_{1, j}}v'^{2})v^{a_{11}}v'^{-1},\\
			G_{\Delta_{j'}}(v,v') & = v^{\sum_{1<j\leq j_1}a_{1, j}-\sum_{j>j'}a_{1, j}}v^{a_{11}}v'^{-1},
		\end{align*}
		where $j'>j_1$. It is clear that $\sum_{1<j\leq j_1}a_{1, j}-\sum_{j> j_1}a_{1, j}$ and $\sum_{1<j\leq j_1}a_{1, j}-\sum_{j> j'}a_{1, j}$ do not change when $A$ is replaced by $_{p}A$. Therefore we have 
		\begin{eq}
			[_{p}D]_{\{(1,1)\}}\ast [_{p}A]_{\Delta}=G_{\Delta}(v, v^{-p})[_{p}A]_{\Delta}+\sum_{j'>j_1}G_{\Delta_{j'}}(v, v^{-p})[_{p} A]_{\Delta_{j'}}.
		\end{eq}
		
		\noindent$(b)$ If $i_1=1$, then
		\begin{align*}
			\relax [_{p}D]_{\{(1,1)\}}\ast [_{p}A]_{\Delta}=& v^{\sum_{j \leq j_1}a_{1,j}+p-\sum_{j> j_1}a_{1,j}}(1-2v^{-2(\sum_{j \leq j_1}a_{1,j}+p)})[_{p}A]_{\Delta}\\
			& +v^{\sum_{j \leq j_1}a_{1,j}+p-\sum_{j>j_2}a_{1j}}(1-v^{-2a_{1,j_1}})[_{p}A]_{\Delta\setminus\{(1,j_1)\}}\\
			& +\sum_{j'>j_2, a_{1,j'}>0}v^{\sum_{j \leq j_1}a_{1,j}+p-\sum_{j>j'}a_{1,j}}[_{p}A]_{\Delta_{j'}},
		\end{align*}
		where $\Delta_{j'}=\{(1,j'), (i_2, j_2), \cdots, (i_k, j_k)\}$. Thus we set $G_{\Delta'_i}$ as
		\begin{align*}
			G_{\Delta}(v,v') & = v^{\sum_{1<j\leq j_1}a_{1j}-\sum_{j> j_1}a_{1j}}(1-2v^{-2\sum_{j\leq j_1}a_{1j}}v'^{2})v^{a_{11}}v'^{-1},\\
			G_{\Delta\setminus\{(1,j_1)\}}(v,v') & =v^{\sum_{1<j\leq j_1}a_{1j}-\sum_{j>j_2}a_{1j}}(1-v^{-2a_{1j_1}})v^{a_{11}}v'^{-1},\\
			G_{\Delta_{j'}}(v,v') & = v^{\sum_{1<j\leq j_1}a_{1j}-\sum_{j>j'}a_{1j}-a_{11}}(1-v^{-2\sum_{j\leq j_1}a_{1j}}v'^{2})v^{a_{11} }v'^{-1}. 
		\end{align*}
		Hence, we have
		$$
		[_{p}D]_{\{(1,1)\}}\ast [_{p}A]_{\Delta}= G_{\Delta}(v,v^{-p})[_{p}A]_{\Delta}+G_{\Delta\setminus\{(1,j_1)\}}[_{p} A]_{\Delta\setminus\{(1,j_1)\}}+\sum_{j'>j_2}G_{\Delta_{j'}}(v,v^{-p})[_{p} A]_{\Delta_{j'}}.
		$$
		
		\noindent$(c)$ If $\Delta=\emptyset$, we have 
		\begin{align*}
			\relax [_{p}D]_{\{(1,1)\}}\ast [_{p}A]_{\Delta}=& \sum\limits_{\substack{t\in [1,n],\\ a_{1, t}>0 \text{ for } t>1}}v^{-\sum_{j>t}a_{1,j}}[_{p}A]_{\{(1, t)\}}
		\end{align*}
		Thus we set $G_{\{(1, t)\}}(v,v')=v^{-\sum_{j>t}a_{1,j}}$ and  
		$$
		[_{p}D]_{\{(1,1)\}}\ast [_{p}A]_{\Delta}=\sum\limits_{\substack{t\in [1,n],\\ a_{1, t}>0 \text{ for } t>1}} G_{\{(1, t)\}}(v,v^{-p})[_{p}A]_{\{(1,t)\}}.
		$$
		Therefore $[_{p}D]_{\{(1,1)\}}\ast [_{p}A]_{\Delta}$ is of the required form.
		
		We define $\Psi(A,\Delta)$ as follows. 
		$$
		\Psi(A, \Delta)=\sum_{1\leq i<j \leq n}\frac{(j-i)(j-i+1)}{2}(a_{ij}+a_{ji})+\sum_{(i,j)\in \Delta} \frac{(j-i)(j-i+1)}{2\sharp  \Delta_{i+j}}+\sharp \Delta,
		$$
		where $\Delta_{k}=\{(i,j)\in \Delta | i+j=k\}$. It is clear that
		if $(A',\Delta') < (A,\Delta)$ then $\Psi(A',\Delta')<\Psi(A, \Delta)$. Thus we can do induction on $\Psi(A,\Delta)$ and the remaining proof is the same as the one of \cite[Proposition 4.2]{BLM}.
	\end{proof}
	
	Let $v'=1$, we can obtain an associative $\mathcal{A}$-algebra $\mathcal{K}$ and $\{[A]_{\Delta}\ |\ (A,\Delta)\in \tilde{\Xi}\}$ form its basis. The product of $\mathcal{K}$ is given by
	\begin{align*}
		\prod_{i=1}^{k}[A_i]_{\Delta_i}=\left\{\begin{array}{cc}
			\sum^{m}_{i=1}G_{\Delta'_i}(v, 1)[Z_i]_{\Delta'_i}   &  \co(A_i)=\ro(A_{i+1})\text{ for } i=1, \cdots, k-1,\\
			0   & \text{ otherwise }.
		\end{array}
		\right.
	\end{align*}
	
	\subsection{Multiplication formulas in $\mathcal{K}$}\ 
	By Proposition \ref{prop 4.1.6}, Proposition \ref{prop 4.1.7}, Proposition \ref{prop 4.1.8} and Proposition \ref{prop 5.1.1}, we have the following multiplication formulas in $\mathcal{K}$.
	\begin{cor}\label{cor 5.2.1}
		Assume $h\in [1, n-1]$ and $R\geq 0$ is an integer. Take $P_R=\{\mathbf{p}=(p_1,\cdots, p_n)\in \mathbb{N}^n\ |\ \sum_{i}p_i=R \}$. For $(B, \emptyset), (A,\Delta) \in \tilde{\Xi}$ satisfied $\ro(A) = \co(B)$, where $B - RE_{h, h+1}$ is a diagonal matrix and $\Delta = \{(i_1, j_1), \cdots, (i_k, j_k)\}$, let $\kappa_{A}(\mathbf{p})=\sum_{j\geq u}a_{h, j}p_u-\sum_{j>u}a_{h+1, j}p_u+\sum_{u<u'}p_{u}p_{u'}$, $X_{\mathbf{p}}=A+\sum_{u}p_u(E_{h, u}-E_{h+1, u})$ where $\mathbf{p}\in P_R$.
		
		\noindent$(a)$ If for any $t\in [1,k]$, we have $i_t \neq h, h+1$, then
		$$[B]_{\emptyset}\cdot [A]_\Delta = \sum\limits_{\mathbf{p}}
		v^{\kappa_{A}(\mathbf{p})}\prod^{n}_{u=1}\overline{[a_{h, u}+p_u, p_u]}_v[X_{\mathbf{p}}]_{\Delta},
		$$
		where $\mathbf{p}$ runs over the sequences in $\{\mathbf{p}\in P_R\ |\ \text{for any $u\neq h+1$, } a_{h+1,u}\geq p_u\}$.
		
		\noindent$(b)$ If there exists $m$ such that $i_m = h$ and $i_{m+1} \neq h+1$, then
		\begin{align*}
			\relax [B]_\emptyset\cdot [A]_{\Delta}  = &\sum\limits_{\mathbf{p}}v^{\kappa_{A}(\mathbf{p})-\sum_{j_{m+1}<u\leq j_m}p_u}\prod_{u\neq j_m}\overline{[a_{h, u}+p_u-\delta_{u j_m}, p_u]}_v[X_{\mathbf{p}}]_{\Delta},
		\end{align*}
		where $\mathbf{p}$ runs over the sequences in $\{\mathbf{p}\in P_R\ |\ \text{for any $u\neq h+1$, } a_{h+1,u}\geq p_u\}$.
		
		\noindent$(c)$ If there exists $m$ such that $i_{m-1} \neq h$ and $i_{m} = h+1$, then
		\begin{align*}
			& \relax [B]_\emptyset\cdot [A]_{\Delta}\\
			= &\sum\limits_{\mathbf{p}}
			v^{\kappa_{A}(\mathbf{p})}\prod^{n}_{u=1}\overline{[a_{h, u}+p_u, p_u]}_v[X_{\mathbf{p}}]_{\Delta} \\
			& + \sum\limits_{\mathbf{p}, p_{j_m}>0} v^{\kappa_{A}(\mathbf{p})-\sum_{j_{m+1}< j \leq j_m}(a_{h+1, j}-p_j)}\prod^{n}_{u=1}\overline{[a_{h, u}+p_u-\delta_{uj_m}, p_u-\delta_{uj_m}]}_v[X_{\mathbf{p}}]_{\Delta_c}\\
			& +\sum\limits_{\substack{\mathbf{p}, p_{j_m}>0,\\ j_{m+1}<t< j_m}}v^{\kappa_{A}(\mathbf{p})-\sum_{t< j \leq j_m}(a_{h+1, j}-p_j)}\prod^{n}_{u=1}\overline{[a_{h, u}+p_u-\delta_{uj_m}, p_u-\delta_{uj_m}]}_v[X_{\mathbf{p}}]_{\Delta_t},
		\end{align*}
		where $\mathbf{p}$ runs over the sequences in $\{\mathbf{p}\in P_R\ |\ \text{for any $u \neq h+1$, } a_{h+1,u}\geq p_u\}$ and 
		\begin{eq}
			\Delta_c&= \{(i_1, j_1), \cdots, (i_{m-1}, j_{m-1}), (h, j_m), (i_{m+1}, j_{m+1}), \cdots, (i_k, j_k)\},\\
			\Delta_t& = \{(i_1, j_1), \cdots, (i_{m-1}, j_{m-1}), (h, j_m), (h+1, t), (i_{m+1}, j_{m+1}), \cdots, (i_k, j_k)\}.
		\end{eq}
		
		\noindent$(d)$ If there exists $m$ such that $i_m=  h$ and $i_{m+1} = h+1$, then
		\begin{align*}
			&    \relax [B]_\emptyset \cdot [A]_{\Delta} \\
			= &\sum\limits_{\mathbf{p}}v^{\kappa_{A}(\mathbf{p})-\sum_{j_{m+1}<u\leq j_m}p_u}\prod_{u\neq j_m}\overline{[a_{h, u}+p_u, p_u]}_v\overline{[a_{h, u}+p_{j_m}-1, p_{j_m}]}_v[X_{\mathbf{p}}]_{\Delta}\\
			& + \sum\limits_{\mathbf{p},p_{j_{m+1}}>0}v^{\kappa_{A}(\mathbf{p})-\sum_{j_{m+2}<j\leq j_{m+1}}(a_{h+1, j}-p_j)}(1-v^{-2p_{j_{m+1}}})\prod^{n}_{u=1}\overline{[a_{h, u}+p_u, p_u]}_v[X_{\mathbf{p}}]_{\Delta_s}\\
			& + \sum\limits_{\substack{\mathbf{p} \in P_R, p_{j_{m+1}}>0, \\j_{m+2}<t<j_{m+1}}} v^{\kappa_{A}(\mathbf{p})-\sum_{t<j\leq j_{m+1}}(a_{h+1, j}-p_j)}(1-v^{-2p_{j_{m+1}}})\prod^{n}_{u=1}\overline{[a_{h, u}+p_u, p_u]}_v[X_{\mathbf{p}}]_{\Delta_t},
		\end{align*}
		where $\mathbf{p}$ runs over the sequences in $\{\mathbf{p}\in P_R\ |\ \text{for any $u \neq h+1$, } a_{h+1,u}\geq p_u\}$ and
		\begin{eq}
			\Delta_s&= \{(i_1, j_1), \cdots, (i_{m-1}, j_{m-1}), (h, j_m), (i_{m+2}, j_{m+2}), \cdots, (i_k, j_k)\}, \\
			\Delta_t &= \{(i_1, j_1), \cdots, (i_{m-1}, j_{m-1}), (h, j_m), (h+1, t), (i_{m+2}, j_{m+2}), \cdots, (i_k, j_k)\}.\\
		\end{eq}
	\end{cor}
	
	\begin{cor}\label{cor 5.2.2}
		Assume $h\in [1, n-1]$ and $R\geq 0$ is an integer. Take $P_R=\{\mathbf{p}=(p_1,\cdots, p_n)\in \mathbb{N}^n\ |\ \sum_{i}p_i=R \}$. For $(C, \emptyset), (A,\Delta) \in \Xi_{n|n,d}$ satisfied $\ro(A) = \co(C)$, where $C - RE_{h+1, h}$ is a diagonal matrix and $\Delta = \{(i_1, j_1), \cdots, (i_k, j_k)\}$, let $\kappa_{A}'(\mathbf{p})=\sum_{j\leq u}a_{h, j}p_u-\sum_{j<u}a_{h+1, j}p_u+\sum_{u<u'}p_{u}p_{u'}$, $X'_{\mathbf{p}}=A+\sum_{u}p_u(E_{h+1, u}-E_{h, u})$ where $\mathbf{p}\in P_R$.
		
		\noindent$(a)$ If for any $t\in [1,n]$, we have $i_t \neq h, h+1$, then
		$$[C]_{\emptyset}\cdot [A]_\Delta = \sum\limits_{\mathbf{p}} v^{\kappa_{A}'(\mathbf{p})}\prod^{n}_{u=1}\overline{[a_{h+1, u}+p_u, p_u]}_v[X_{\mathbf{p}}]_{\Delta},$$
		where $\mathbf{p}$ runs over the sequences in $\{\mathbf{p}\in P_R\ |\ \text{for any $u\neq h$, } a_{h,u}\geq p_u\}$.
		
		\noindent$(b)$ If there exists $m$ such that $i_m= h$ and $i_{m+1} \neq h+1$, we have
		\begin{align*}
			& \relax [C]_\emptyset\cdot [A]_{\Delta}\\
			= & \sum\limits_{\mathbf{p}} v^{\kappa_{A}'(\mathbf{p})-\sum_{j_{m+1}<u\leq j_m}p_u}\prod^{n}_{u=1}\overline{[a_{h+1, u}+p_u, p_u]}_v[X'_{\mathbf{p}}]_{\Delta} \\
			& +\sum\limits_{\mathbf{p}, p_{j_m}>0} v^{\kappa_{A}'(\mathbf{p})-\sum_{j_{m+1}<j\leq j_m}a_{h+1,j}} \prod^{n}_{u=1}\overline{[a_{h+1, u}+p_u-\delta_{u j_m}, p_u-\delta_{u j_m}]}_v[X'_{\mathbf{p}}]_{\Delta_b} \\  
			& + \sum\limits_{\substack{\mathbf{p}\\ j_{m+1}<t<j_m, p_t>0}}v^{\kappa_{A}'(\mathbf{p})-\sum_{j_{m+1}<j\leq t}a_{h+1,j}}\prod^{n}_{u=1}\overline{[a_{h+1, u}+p_u-\delta_{u t}, p_u-\delta_{u t}]}_v[X'_{\mathbf{p}}]_{\Delta_t},
		\end{align*}
		where $\mathbf{p}$ runs over the sequences in $\{\mathbf{p}\in P_R\ |\ \text{for any $u \neq h$, } a_{h,u}\geq p_u\}$ and 
		\begin{eq}
			\Delta_b &= \{(i_1, j_1), \cdots, (i_{m-1}, j_{m-1}), (h+1, j_m), (i_{m+1}, j_{m+1}), \cdots, (i_k, j_k)\}, \\
			\Delta_t &= \{(i_1, j_1), \cdots, (i_{m-1}, j_{m-1}), (h, j_m), (h+1, t), (i_{m+1}, j_{m+1}), \cdots, (i_k, j_k)\}.
		\end{eq}
		
		\noindent$(c)$ If there exists $m$ such that $i_{m-1} \neq h$ and $i_{m} = h+1$, then
		\begin{align*}
			\relax [C]_\emptyset\cdot [A]_{\Delta}= & \sum\limits_{\mathbf{p}} v^{\kappa_{A}'(\mathbf{p})}\prod^{n}_{u=1}\overline{[a_{h+1, u}+p_u-\delta_{u j_m}, p_u]}_v[X'_{\mathbf{p}}]_{\Delta},
		\end{align*}
		where $\mathbf{p}$ runs over the sequences in $\{\mathbf{p}\in P_R\ |\ \text{for any } u\neq h, a_{h,u}\geq p_u\}$.
		
		\noindent$(d)$ If there exists  $m$ such that $i_m=  h$ and $i_{m+1} = h+1$, we have
		\begin{align*}
			& \relax [C]_\emptyset\cdot [A]_{\Delta} \\
			= & \sum\limits_{\mathbf{p}} v^{\kappa_{A}'(\mathbf{p})}\prod^{n}_{u=1}\overline{[a_{h+1, u}+p_u-\delta_{u j_{m+1}}, p_u]}_v[X_{\mathbf{p}}]_{\Delta} \\
			& + \sum\limits_{\mathbf{p}, p_{j_m}>0} v^{\kappa_{A}'(\mathbf{p},j_{m+1})}(1-v^{-2a_{h+1,j_{m+1}}})\prod^{n}_{u=1}\overline{[a_{h+1, u}+p_u-\delta_{u j_m}, p_u-\delta_{u j_m}]}_v[X_{\mathbf{p}}]_{\Delta_s} \\ 
			& + \sum\limits_{\substack{\mathbf{p}\\j_{m+1}<t<j_m, p_t>0}} v^{\kappa_{A}'(\mathbf{p},t)}(1-v^{-2a_{h+1,j_{m+1}}})\prod^{n}_{u=1}\overline{[a_{h+1, u}+p_u-\delta_{u t}, p_u-\delta_{u t}]}_v[X'_{\mathbf{p}}]_{\Delta_t},
		\end{align*}
		where $\mathbf{p}$ runs over the sequences in $\{\mathbf{p}\in P_R\ |\ \text{for any } u \neq h, a_{h,u}\geq p_u\}$, $\kappa_{A}'(\mathbf{p},q)=\kappa_{A}'(\mathbf{p})-\sum_{q<j\leq j_m}a_{h+1, j}$ and 
		\begin{eq}
			\Delta_s &= \{(i_1, j_1), \cdots, (i_{m-1}, j_{m-1}), (h+1, j_m), (i_{m+2}, j_{m+2}), \cdots, (i_k, j_k)\},\\
			\Delta_t &=\{(i_1, j_1), \cdots, (i_{m-1}, j_{m-1}), (h, j_m), (h+1, t), (i_{m+2}, j_{m+2}), \cdots, (i_k, j_k)\}.
		\end{eq}
	\end{cor}
	
	Recall the definition of $\theta_{A}'(i_p, j_p)$ and $\rho'(A,\Delta')$ as presented in Proposition \ref{prop 4.1.6}, we deduce the following corollary.
	\begin{cor}\label{cor 5.2.3}
		Assume $h\in [1, n]$. For $(D,\{h, h\}), (A, \Delta) \in \tilde{\Xi}$ satisfied $\ro(A)= \co(D)$ and $D$ is a diagonal matrix, 
		where $\Delta = \{(i_1, j_1), \cdots, (i_k, j_k)\}$. 
		
		\noindent$(a)$ If there is $(i_m,j_m)\in \Delta$ such that $i_{m-1}<h<i_m$, then
		\begin{align*}
			\relax [D]_{\{(h,h)\}}\cdot [A]_{\Delta}=\sum\limits_{\Delta'} v^{\sum_{\{(i,j)\}\leq \Delta'}a_{i, j}-\sum_{\{(i,j)\}\leq \Delta}a_{i, j}}\prod^{m-1}_{p=1}\theta_{A}'(i_{p}, j_{p})\rho'(A,\Delta') [A]_{\Delta'},
		\end{align*}
		where $\Delta'$ runs over the set
		$$
		\left\{ \Delta'=\{(i'_1, j'_1), \cdots, (i'_{l}, j'_{l})\}\ \bigg| \ \begin{array}{l}
			(1)  (i_{t}, j_{t})\in \Delta'\text{ for all } t \in [m, k] \\
			(2)  \text{ if } (i', j')\in \Delta'\setminus \Delta, \text{ then } i'\leq h
		\end{array}\right\}.$$
		
		\noindent$(b)$ If there is $i_m=h$, then
		\begin{align*}
			\relax [D]_{\{(h,h)\}}\cdot [A]_{\Delta}=\sum\limits_{\Delta'}v^{\sum_{\{(i,j)\}\leq \Delta'}a_{i, j}-\sum_{\{(i,j)\}\leq \Delta}a_{i, j}}\prod^{m}_{p=1}\theta_{A}'(i_{p}, j_{p})[A]_{\Delta'},
		\end{align*}
		where $\Delta'$ runs over the set
		$$
		\left\{ \Delta'=\{(i'_1, j'_1), \cdots, (i'_{l}, j'_{l})\}\ \bigg| \ \begin{array}{l}
			(1)  (i_{t}, j_{t})\in \Delta'\text{ for all } t \in [m+1, k] \\
			(2)  \text{ if } (i', j')\in \Delta'\setminus \Delta, \text{ then } i'\leq h
		\end{array}\right\}.$$
		
		\noindent$(c)$ If $i_k<h$, we have 
		\begin{align*}
			\relax [D]_{\{(h,h)\}}\cdot [A]_{\Delta}=\sum\limits_{\Delta'} v^{\sum_{\{(i,j)\}\leq \Delta'}a_{i, j}-\sum_{\{(i,j)\}\leq \Delta}a_{i, j}-\sum_{i_k<i\leq h, j}a_{i j}}\prod^{k}_{p=1}\theta_{A}'(i_{p}, j_{p}) [A]_{\Delta'},
		\end{align*}
		where $\Delta'$ runs over the set $\{ \Delta'=\{(i'_1, j'_1), \cdots, (i'_{l}, j'_{l})\}\ |\ i'_l\leq h\}$.
		
		\noindent$(d)$ If $\Delta=\emptyset$, we have
		\begin{align*}
			\relax [D]_{\{(h,h)\}}\cdot [A]_{\emptyset}=\sum\limits_{t\in [1,n], a_{ht}>0} v^{-\sum_{i\leq h,j>t}a_{i, j}} [A]_{\{(h,t)\}}.
		\end{align*}
	\end{cor} 
	
	By Proposition \ref{prop 4.2.1}, Corollary \ref{cor 5.2.1}, Corollary \ref{cor 5.2.2} and Corollary \ref{cor 5.2.3} we have the following proposition.
	\begin{prop}\label{prop 5.2.4}
		The algebra $\mathcal{K}$ can be generated by elements 
		$[B]_{\emptyset}$, $[C]_{\emptyset}$, $[D]_{\emptyset}$ and $[D]_{\{(1,1)\}}$ such that $D$, $B-E_{i,i+1}$ and $C-E_{i+1,i}$ are diagonal matrices for some $i\in [1,n-1]$.
	\end{prop}
	We can introduce a $\mathcal{A}$-linear map $\eta: \mathcal{K}\to \mathcal{MS}_{n,d}$ as follows.
	\begin{align*}
		\eta([A]_{\Delta})=\left\{\begin{array}{cc}
			[A]_{\Delta}  &  \text{if } (A,\Delta)\in \Xi_{n|n,d},\\
			0 & \text{otherwise}.
		\end{array}
		\right.
	\end{align*}
	\begin{prop}\label{prop 5.2.5}
		$\eta$ is a surjective algebra homomorphism.
	\end{prop}
	\begin{proof}
		It is easy to see that $\eta$ is surjective. So it is enough to prove the following identity
		\begin{align*}
			\eta([B]_{\Delta'}\cdot [A]_{\Delta})=\eta([B]_{\Delta'})\ast \eta([A]_\Delta),
		\end{align*}
		where $[B]_{\Delta'}$ is one of the generators in Proposition \ref{prop 5.2.4} with $\co(B)=\ro(A)$. 
		
		If $\sum_{1\leq i,j\leq n} a_{i, j}\neq d$, then we have $\eta([B]_{\Delta'}\cdot [A]_{\Delta})=0=\eta([B]_{\Delta'})\ast \eta([A]_{\Delta})$.
		
		Firstly, we assume $\sum_{1\leq i,j\leq n} a_{i, j}=d$, $\co(B)=\ro(A)$, $B-E_{i, i+1}$ is a diagonal matrix for some $i\in [1,n-1]$ and $\Delta'=\emptyset$. Let 
		$X_{i, p}=A+E_{i, p}-E_{i+1, p}$ and $\mathfrak{D}_{i,p}=\{\Delta''| (X_{i, p}, \Delta'')\in \tilde{\Xi}\}$, we have 
		$$[B]_{\Delta'}\cdot [A]_{\Delta}=\sum\limits_{p \in [1,n], \Delta''\in \mathfrak{D}_{i,p}}c_{p, \Delta''}[X_{i, p}]_{\Delta''}.$$
		
		$(a)$ If $[B]_{\emptyset}$ and $[A]_{\Delta}\in \Xi_{n|n,d}$, then the identity follows from Proposition \ref{prop 4.1.4} and Corollary \ref{cor 5.2.1}.
		
		$(b)$ If there is $a_{i_0, i_0}< 0$ and $(i_0, i_0)\notin \Delta$, then we have $\eta([B]_{\Delta'})\ast \eta([A]_{\Delta})=0$. On the other hand, if $B-E_{i, i+1}$ is a diagonal matrix for $i\neq i_0$, then for any $p\in [1,n]$ and $\Delta''\in \mathfrak{D}_{i,p}$, $\eta([X_{i,p}]_{\Delta''})=0$. 
		
		If $B-E_{i_0, i_0+1}$ is a diagonal matrix. For any $p\neq i_0$ and $\Delta''\in \mathfrak{D}_{i_0,p}$, we have $\eta([X_{i_0, p}]_{\Delta''})=0$. For $p=i_0$ , if $a_{i_0, i_0}<-1$, then $\eta([X_{i_0, i_0}]_{\Delta''})=0$. If $a_{i_0, i_0}=-1$, then
		$$
		\eta(\overline{[a_{i_0, i_0}+1, 1]}[X_{i_{0}, i_0}]_{\Delta''})=0
		.$$
		Therefore we have $ \eta([B]_{\emptyset}\cdot [A]_{\Delta})=0=\eta([B]_{\emptyset})\ast \eta([A]_\Delta)$ by Corollary \ref{cor 5.2.1}.
		
		$(c)$ If there is $a_{i_0, i_0}\leq 0$ and $(i_0, i_0) \in \Delta$, then we have $\eta([B]_{\Delta'})\ast \eta([A]_{\Delta})=0$. If $B-E_{i_0, i_0+1}$ is a diagonal matrix, then for $p\neq i_0$ and $\Delta'' \in \mathfrak{D}_{i_0,p}$, $\eta([X_{i_0, p}]_{\Delta''})=0$. If $a_{i_0, i_0}\leq-1$, then $\eta([X_{i_0, i_0}]_{\Delta''})=0$. If $a_{i_0, i_0}=0$, then
		$$
		\eta(\overline{[a_{i_0, i_0}, 1]}[X_{i_0, i_0}]_{\Delta''})=0
		.$$
		
		If $i\neq i_0$ and $B-E_{i, i+1}$ is a diagonal matrix. For any  $p\in [1, n]$ and $\Delta''\in \mathfrak{D}_{i,p}$, we have $\eta([X_{i, p}]_{\Delta''})=0$.
		Therefore Corollary \ref{cor 5.2.1} implies $\eta([B]_{\emptyset}\cdot [A]_{\Delta})=0=\eta([B]_{\emptyset})\ast \eta([A]_\Delta)$.
		
		$(d)$ If $(A,\Delta)\in \Xi_{n|n,d}$ and $[B]_{\{1,1\}} \notin \Xi_{n|n,d}$, then there is $i_0$ such that $b_{i_0+1, i_0+1}=-1$ and $b_{i_0, i_0+1}=1$. Since $\co(B)=\ro(A)$, we have $a_{i_0+1, p}=0$ for any $\p\in [1,n]$ and $a_{i_0+1, i_0+1}-1=-1$. Therefore for any $\Delta''\in \mathfrak{D}_{i_0},p$, $\eta([X_{i_0, i_0+1}]_{\Delta''})=0$.
		This together with Corollary \ref{cor 5.2.1} elucidates $\eta([B]_{\emptyset}\cdot [A]_{\Delta})=0=\eta([B]_{\emptyset})\ast \eta([A]_{\Delta})$.
		
		Secondly, if $B-E_{i+1, i}$ is a diagonal matrix for some $i\in [1,n-1]$ and $\Delta'=\emptyset$, then similar to $B-E_{i, i+1}$, we have $\eta([B]_{\emptyset}\cdot [A]_{\Delta})=0=\eta([B]_{\emptyset})\ast \eta([A]_{\Delta})$.
		
		Thirdly, we assume $B$ is a diagonal matrix and $\Delta'=\{(1, 1)\}$. 
		
		$(a)$ If $(B,\{(1,1)\}), (A,\Delta)\in \Xi_{n|n,d}$, the identity follows from Proposition \ref{prop 4.1.6} and Corollary \ref{cor 5.2.3}.
		
		$(b)$ If $a_{i_0, i_0}<0$. For any $\Delta''$ such that $(A, \Delta'')\in \tilde{\Xi}$, we have $\eta([A]_{\Delta''})=0$ and $\eta([B]_{\{(1,1)\}})\ast \eta([A]_{\Delta})=0$. Then by Corollary \ref{cor 5.2.3}, we have $\eta([B]_{\{(1,1)\}}\cdot [A]_{\Delta})=0$. 
		
		$(c)$ If $a_{i_0, i_0}=0$ and $(i_0,i_0)\in \Delta$. Since $\eta([A]_{\Delta}=0$, we have $\eta([B]_{\{(1,1)\}})\ast \eta([A]_{\Delta})=0$. On the other hand, 
		$$
		\eta(\theta_{A}'(1,1)[A]_{\emptyset})=\eta(\theta_{A}'(1,1)[A]_{\{(1, j)\}})=0
		,$$
		where $j\in [2, n]$ such that $a_{1, j}>0$. Due to Corollary \ref{cor 5.2.3}, we have $\eta([B]_{\{1,1\}}\cdot [A]_{\Delta})=0$. 
		
		$(d)$ The last case we need to check is $(A,\Delta)\in \Xi_{n|n,d}$ but $[B]_{\{1,1\}} \notin \Xi_{n|n,d}$, which shows $b_{11}=0$ and $\eta([B]_{\{(1,1)\}})\ast \eta([A]_{\Delta})=0$. Since $\co(B)=\ro(A)$, we have $a_{1, p}=0$ for any $p\in [1, n]$. Then $\eta([B]_{\{1,1\}}\cdot [A]_{\Delta})=\eta(\rho'(A,\Delta')[A]_{\Delta})=0$. 
		
		Finally, if $(B,\emptyset)\notin \Xi_{n|n,d}$ and $\co(B)=\ro(A)$, then $(A,\Delta)\notin \Xi_{n|n,d}$. Therefore $\eta([B]_{\emptyset}\cdot [A]_{\Delta})=\eta([A]_{\Delta})=0=\eta([B]_{\emptyset})\ast \eta([A]_{\Delta})$. 
		Thus $$\eta([B]_{\{1,1\}}\cdot [A]_{\Delta})=0=\eta([B]_{\{(1,1)\}})\ast \eta([A]_{\Delta}).$$
		
		Therefore $\eta$ is a surjective algebra homomorphism.
	\end{proof}
	% We can see that $\eta$ is an algebra homomorphism by comparing the multiplication formulas and generators of these two algebras. 
	
	\subsection{Mirabolic quantum $\mathfrak{gl}_n$}\ 
	Let $\widehat{\mathcal{K}}$ be the vector space over $\mathbb{Q}(v)$ 
	of all formal $\mathbb{Q}(v)$-linear combinations $\sum_{(A, \Delta)\in \tilde{\Xi} }\xi_{(A, \Delta)}[A]_{\Delta}$, 
	where for any $a\in \mathbb{Z}^n$, 
	the set $\{(A, \Delta)\in\tilde{\Xi} | \xi _{(A, \Delta)}\neq 0 \ , \ro(A)=a\}$ 
	and $\{(A, \Delta)\in\tilde{\Xi} | \xi _{(A, \Delta)}\neq 0 \ , \co(A)=a\}$ are finite. 
	Then the product in $\mathcal{K}$ induces a well-defined product in $\widehat{\mathcal{K}}$  
	$$\sum_{(A, \Delta)\in \tilde{\Xi} }\xi_{(A, \Delta)}[A]_{\Delta} \cdot \sum_{(B, \Delta')\in \tilde{\Xi} }\xi_{(B, \Delta')}[B]_{\Delta'}
	= \sum_{(A, \Delta), (B, \Delta')\in \tilde{\Xi} }\xi_{(A, \Delta)}\xi_{(B, \Delta')}[A]_{\Delta}\cdot [B]_{\Delta'}.
	$$
	So $\widehat{\mathcal{K}}$ is an associative algebra over $\mathbb{Q}(v)$ and take the algebra $\mathcal{K}$ as its subalgebra.
	
	Let $\Xi^0$ be the set of all $(A, \Delta)\in \tilde{\Xi}$ such that all diagonal entries of $A$ are zero. 
	For $A\in \Xi^{0}$ and $\mathbf{j}=(j_1,\cdots, j_n)\in \mathbb{Z}^n$, we can define elements $A(\mathbf{j})_{\Delta}$ in $\widehat{\mathcal{K}}$ 
	$$
	A(\mathbf{j})_{\Delta}=\sum_{Z} v^{z_1j_1+ \cdots +z_n j_n}[A+Z]_{\Delta},
	$$
	where $Z$ runs over all diagonal matrices $diag(z_1, \cdots, z_n)$ in $Mat_{n\times n}(\mathbb{Z})$.
	
	\begin{Def}
		Take $\mathbf{j}_h=(0,\cdots, 1,\cdots, 0)$(1 on the $h$-th position). Let $\mathcal{MU}$ be the $\mathbb{Q}(v)$-subalgebra of $\widehat{\mathcal{K}}$ generated by $0(-2\mathbf{j}_1)_{\emptyset}+0(-\mathbf{j}_1)_{\{(1, 1)\}}$, $0(\mathbf{j}_a)_{\emptyset}$, $E_i=E_{i, i+1}(\mathbf{0})_{\emptyset}$ and $F_i=E_{i+1, i}(\mathbf{0})_{\emptyset}$ for $1\leq i\leq n-1$ and $1\leq a\leq n$. 
	\end{Def}
	
	\begin{prop}\label{prop 5.3.1}
		For $i\in [1, n-1]$ and $a\in [1, n]$, we have the following relations in $\mathcal{MU}$. 
		\begin{enumerate}[ $(a)$ ]
			\item $0(\mathbf{j}_a)_{\emptyset}0(-\mathbf{j}_a)_{\emptyset}=1$;
			\item $E_{i}^2E_{i+1}+E_{i+1}E_{i}^2=(v+v^{-1})E_{i}E_{i+1}E_{i}$;
			\item $E_{i+1}^2E_{i}+E_{i}E_{i+1}^2=(v+v^{-1})E_{i+1}E_{i}E_{i+1}$;
			\item $F_{i}^2F_{i+1}+F_{i+1}F_{i}^2=(v+v^{-1})F_{i}F_{i+1}F_{i}$;
			\item $F_{i+1}^2F_{i}+F_{i}F_{i+1}^2=(v+v^{-1})F_{i+1}F_{i}F_{i+1}$;
			\item $0(\mathbf{j}_a)_{\emptyset}E_{i}=v^{\delta_{ah}-\delta_{a, h+1}}E_{i} 0(\mathbf{j}_a)_{\emptyset}$;
			\item $0(\mathbf{j}_a)_{\emptyset}F_{i}=v^{-\delta_{ah}+\delta_{a, h+1}}F_{i} 0(\mathbf{j}_a)_{\emptyset}$;
			\item $E_{i}F_{j}-F_{j}E_{i}=\delta_{i j}\frac{0(\mathbf{j}_i-\mathbf{j}_{i+1})-0(-\mathbf{j}_i+\mathbf{j}_{i+1})}{v-v^{-1}}$;
			\item $0(\mathbf{j}_a)_{\emptyset}(0(-2\mathbf{j}_1)_{\emptyset}+0(-\mathbf{j}_1)_{\{(1, 1)\}})=(0(-2\mathbf{j}_1)_{\emptyset}+0(-\mathbf{j}_1)_{\{(1, 1)\}})0(\mathbf{j}_a)_{\emptyset}$;
			\item $(0(-2\mathbf{j}_1)_{\emptyset}+0(-\mathbf{j}_1)_{\{(1, 1)\}})^2=0(-2\mathbf{j}_1)_{\emptyset}+0(-\mathbf{j}_1)_{\{(1, 1)\}}$;
			\item $(0(-2\mathbf{j}_1)_{\emptyset}+0(-\mathbf{j}_1)_{\{(1, 1)\}})E_{i}=(0(-2\mathbf{j}_1)_{\emptyset}+0(-\mathbf{j}_1)_{\{(1, 1)\}})E_{i}(0(-2\mathbf{j}_1)_{\emptyset}+0(-\mathbf{j}_1)_{\{(1, 1)\}})$;
			\item $(0(-2\mathbf{j}_1)_{\emptyset}+0(-\mathbf{j}_1)_{\{(1, 1)\}})F_{i}=(0(-2\mathbf{j}_1)_{\emptyset}+0(-\mathbf{j}_1)_{\{(1, 1)\}})F_{i}(0(-2\mathbf{j}_1)_{\emptyset}+0(-\mathbf{j}_1)_{\{(1, 1)\}})$;
			\item $(v+v^{-1}) E_{i}(0(-2\mathbf{j}_1)_{\emptyset}+0(-\mathbf{j}_1)_{\{(1, 1)\}})E_{i}=v^{-1}E_{i}^{2}(0(-2\mathbf{j}_1)_{\emptyset}+0(-\mathbf{j}_1)_{\{(1, 1)\}})+v(0(-2\mathbf{j}_1)_{\emptyset}+0(-\mathbf{j}_1)_{\{(1, 1)\}})E_{i}^{2}$;
			\item $(v+v^{-1}) F_{i}(0(-2\mathbf{j}_1)_{\emptyset}+0(-\mathbf{j}_1)_{\{(1, 1)\}})F_{i}=vF_{i}^{2}(0(-2\mathbf{j}_1)_{\emptyset}+0(-\mathbf{j}_1)_{\{(1, 1)\}})+v^{-1}(0(-2\mathbf{j}_1)_{\emptyset}+0(-\mathbf{j}_1)_{\{(1, 1)\}})F_{i}^{2}$.
		\end{enumerate} 
	\end{prop}
	\begin{proof}
		Relations $(a)$-$(h)$ can be checked by the same calculation in \cite[Lemma 5.6]{BLM}. For the remaining relations, we can verify them by using Corollary \ref{cor 5.2.1}, Corollary \ref{cor 5.2.2} and Corollary \ref{cor 5.2.3}. We mainly prove $(l)$ for $i=1$, the calculations for others are similar. 
		
		We first compute $(0(-2\mathbf{j}_1)_{\emptyset}+0(-\mathbf{j}_1)_{\{(1, 1)\}})E_{1}$,
		\begin{align*}
			&(0(-2\mathbf{j}_1)_{\emptyset}+0(-\mathbf{j}_1)_{\{(1, 1)\}})E_{1}\\
			=& \sum\limits_{Z'}v^{-2z'_1}[Z']_{\emptyset}\sum\limits_{Z}[Z+E_{1, 2}]_{\emptyset}+\sum\limits_{Z',z'_1>0}v^{-z'_1}[Z']_{\{(1,1)\}}\sum\limits_{Z}[Z+E_{1, 2}]_{\emptyset}\\
			= & v^{-2}E_{1, 2}(-2\mathbf{j}_1)_{\emptyset}+v^{-2}E_{1, 2}(-\mathbf{j}_1)_{\{(1,1)\}}+v^{-1}E_{1,2}(-\mathbf{j}_1)_{\{(1,2)\}}.
		\end{align*}
		Besieds, 
		\begin{align*}
			& E_{1}(0(-2\mathbf{j}_1)_{\emptyset}+0(-\mathbf{j}_1)_{\{(1, 1)\}})\\
			=&\sum\limits_{Z'}[Z'+E_{1, 2}]_{\emptyset}\sum\limits_{Z}v^{-2z_1}[Z]_{\emptyset}+\sum\limits_{Z'}[Z'+E_{1, 2}]_{\emptyset}\sum\limits_{Z,z_1>0}v^{-z_1}[Z]_{\{(1,1)\}}\\
			= &E_{1, 2}(-2\mathbf{j}_1)_{\emptyset}+E_{1, 2}(-\mathbf{{j}_1})_{\{(1,1)\}}.
		\end{align*}
		So 
		\begin{align*}
			&  (0(-2\mathbf{j}_1)_{\emptyset}+0(-\mathbf{j}_1)_{\{(1, 1)\}})E_{1}(0(-2\mathbf{j}_1)_{\emptyset}+0(-\mathbf{j}_1)_{\{(1, 1)\}})\\
			=& (0(-2\mathbf{j}_1)_{\emptyset}+0(-\mathbf{j}_1)_{\{(1, 1)\}})E_{1, 2}(-2\mathbf{j}_1)_{\emptyset}\\
			& +(0(-2\mathbf{j}_1)_{\emptyset}+0(-\mathbf{j}_1)_{\{(1, 1)\}})E_{1, 2}(-\mathbf{{j}_1})_{\{(1,1)\}}\\
			= & v^{-2}E_{1, 2}(-2\mathbf{j}_1)_{\emptyset}+v^{-2}E_{1, 2}(-\mathbf{j}_1)_{\{(1,1)\}}+v^{-1}E_{1,2}(-\mathbf{j}_1)_{\{(1,2)\}}\\
			= & (0(-2\mathbf{j}_1)_{\emptyset}+0(-\mathbf{j}_1)_{\{(1, 1)\}})E_{1},
		\end{align*}
		then $(l)$ follows.
	\end{proof}
	Comparing the generators and the relations of the algebra $\mathbf{MU}$ in Definition \ref{def 2.1.2} and the algebra $\mathcal{MU}$, we have the following proposition.
	\begin{prop}\label{prop 5.3.3}
		There is a surjective algebra homomorphism from $\mathbf{MU}$ to the algebra $\mathcal{MU}$ which sends $E_i\mapsto E_i$, $F_i\mapsto F_i$, $L\mapsto 0(-2\mathbf{j}_1)_{\emptyset}+0(-\mathbf{j}_1)_{\{(1, 1)\}}$ and $H_a\mapsto 0(\mathbf{j}_a)_{\emptyset}$ for any $i\in [1,n-1]$ and $a\in [1,n]$.
	\end{prop}
	% Then by Proposition \ref{prop 4.2.5}, we have the following corollary.
	% \begin{cor}\label{cor 5.3.4}
	%     There is a surjective algebra homomorphism $\kappa:\mathcal{MU} \to \mathbb{Q}(v)\otimes_{\mathcal{A}}\mathcal{MS}_{n, d}$, which takes $E_{i}$, $F_{i}$, $0(\mathbf{j})_{\emptyset}$ and $0(\mathbf{j})_{\emptyset}+0(\mathbf{\tilde{j}})_{\{(1, 1)\}}$ of $\mathcal{MU}$ to the corresponding generators in $\mathbb{Q}(v)\otimes_{\mathcal{A}}\mathcal{MS}_{n,d}$ for all $i\in [1, n-1]$.
	% \end{cor}
	
	\section{The Geometric Approach of Mirabolic Schur-Weyl Duality}
	\subsection{Geometric approach}
	Let $G'=V\rtimes G$ act on $Y=\mathscr{Y}\times V$ and $X=\mathscr{X}\times V$ as follows.
	
	For $(\omega, g)\in G'$, $(f, \omega_1)\in X$ and $(f', \omega_2)\in Y$, 
	$$
	(\omega, g)\cdot (f, \omega_1)=(g\cdot f, g\cdot \omega_1+\omega), \  (\omega, g)\cdot(f', \omega_2)=(g\cdot f', g\cdot \omega_2+\omega).
	$$
	Let $G'$ diagonally acts on $X\times X$, $X\times Y$ and $Y \times Y$.
	
	\begin{thm}[\cite{P}]\label{thm p}
		Let $M$ and $N$ be two sets equipped with the $G$-action, if $M$ and $N$ satisfy the following conditions.
		\begin{enumerate}
			\item $ N=\bigsqcup_{i \in I} N_i $, where $I$ is a finite set,
			\item for any $i \in I$, there is a surjective $G$-equivariant map, $\phi_i: M \rightarrow N_i$, which has finite fibers of constant cardinal $m_i$, 
			\item there exists a $d \in I$ such that $\phi_{d}$ is a bijection, 
		\end{enumerate}
		then for the algebras $A=\mathbb{C}_G(N \times N)$, $B=\mathbb{C}_G(M \times M)$ and the space $C=\mathbb{C}_G(N \times M)$,
		\begin{align*}
			\operatorname{End}_B(C)\cong A, \ \operatorname{End}_A(C)\cong B.
		\end{align*}
	\end{thm}
	\begin{rem}
		Theorem \ref{thm p} still hold when we replace the complex number field $\mathbb{C}$ with the ring $\mathcal{A}$.
	\end{rem}
	
	The space $MS=\mathcal{A}_{G'}(X\times X)$ is a convolution algebra, where the convolution is defined as follows.
	$$
	h_1\ast h_2(f, \omega, f', \omega')=\sum\limits_{f'', \omega''}h_1(f, \omega, f'', \omega'')h_2(f'', \omega'', f', \omega').
	$$
	Similarly, $MH=\mathcal{A}_{G'}(Y\times Y)$ is also a convolution algebra.
	We see that $MV=\mathcal{A}_{G'}(X\times Y)$ is equipped with a left $MS$-action and a right $MH$-action by the convolution product. Let $\Lambda_{n,d}=\{\underline{a}=(a_1, \cdots, a_n)\in \mathbb{N}^n | \sum^{n}_{i=1}a_i=d\}$ denote the set of compositions of $d$ of length $n$. 
	
	\begin{prop}\label{prop 6.2.1}
		We have the following double centralizer property for $n \geq d$.
		\begin{align*}
			\End_{MH}(MV)\cong MS,\  \End_{MS}(MV)\cong MH.
		\end{align*}
	\end{prop}
	\begin{proof}
		We only need to clarify that $X$ and $Y$ satisfy the conditions outlined in Theorem \ref{thm p} as the $G'$-set.
		
		$(1)$ Since $\mathscr{X}$ has a decomposition 
		$$\mathscr{X}=\bigsqcup\limits_{ \underline{a} \in \Lambda_{n, d}} \mathscr{X}_{\underline{a}},$$ 
		where $\mathscr{X}_{\underline{a}} = \{f\in \mathscr{X} | \dim (V_{i}/V_{i-1})=a_i\}$, we have a decomposition for $X$,
		$$
		X=\bigsqcup\limits_{\underline{n}\in \Lambda_{n, d}} \mathscr{X}_{\underline{a}}\times V.
		$$
		
		$(2)$ We have a canonical surjective map of $G$-equivariant $\phi_{\underline{a}}:\mathscr{Y} \to \mathscr{X}_{\underline{a}}$. 
		Therefore we have a surjective $G'$-equivariant map 
		\begin{align*}
			\phi_{\underline{a}}\times id_{V} : Y \to \mathscr{X}_{\underline{a}}\times V.
		\end{align*}
		
		$(3)$ Since $n\geq d$, there is an isomorphism $\phi_{\underline{d}}$, where $\underline{d}=(\underbrace{1, 1, \cdots, 1}_{d}, 0, \cdots, 0)\in \Lambda_{n, d}$. Thus $\phi_{\underline{d}}\times id_{V}$ is an isomorphism.
		Then the proposition follows.
	\end{proof}
	Viewing $V$ as a normal subgroup of $G'$ by embedding. We can define a $G'$-action on $\mathscr{X}$, $\mathscr{Y}$ and $V$ through the isomorphism $G'/V\to G$, $\overline{(\omega,g)} \mapsto g$.
	% , where $\overline{(\omega,g)}$ is the equivalence class of $(\omega,g)$.
	
	Let $G'$ diagonally acts on $\mathscr{X}\times \mathscr{X} \times V$, $\mathscr{Y}\times \mathscr{Y} \times V$, $\mathscr{X}\times \mathscr{Y} \times V$. 
	There is a surjecitve $G'$-equivarient map 
	\begin{align*}
		\Phi:X\times X & \to \mathscr{X} \times \mathscr{X} \times V\\
		(f, \omega, f_1, \omega_1)&\mapsto (f, f_1, \omega_1-\omega).
	\end{align*}
	Similarly, we have surjective $G'$-equivarient maps $\Phi':Y\times Y\to \mathscr{Y}\times \mathscr{Y}\times V$ and $\tilde{\Phi}:X\times Y\to \mathscr{X}\times \mathscr{Y}\times V$.

	Employ the convolution product of $\mathcal{MS}_{n,d}$, the space $\mathcal{A}_{G'}(\mathscr{X}\times \mathscr{X}\times V)$ admits a convolution algebra stucture and $\mathcal{A}_{G'}(\mathscr{X}\times \mathscr{X}\times V)\cong \mathcal{MS}_{n,d}$.
	Similarly, we have $\mathcal{A}_{G'}(\mathscr{Y}\times \mathscr{Y}\times V)\cong \mathcal{MH}$ and $\mathcal{A}_{G'}(\mathscr{X}\times \mathscr{Y}\times V)\cong \mathcal{MV}$.
	\begin{prop}\label{prop 6.2.3}
		We have 
		\begin{align*}
			MS \cong \mathcal{MS}_{n, d},\ MH \cong \mathcal{MH}, \ MV \cong \mathcal{MV}. 
		\end{align*} 
	\end{prop}
	\begin{proof}
		Since $\Phi$ is a surjective $G'$-quivariant, we define a linear map 
		\begin{align*}
			\Phi^{\ast}: \mathcal{MS}_{n, d} &\to MS\\
			h & \mapsto \Phi^{\ast}h,
		\end{align*}
		where $\Phi^{\ast}$ is the pullback of $\Phi$, i.e., $\Phi^{\ast}h(f, \omega, f', \omega')=h(\Phi(f, \omega, f', \omega'))$. Since $\Phi$ is a surjective map , the linaer map $\Phi^{*}$ is injective. 
		%  such that for any $h\in MS$ and $(f, f', \omega)\in \mathscr{X}\times \mathscr{X}\times V$, $\psi(h)(f, f', \omega)=h(f, 0, f', \omega)$. Since for any $(\mu, g)\in G'$, we have
		% \begin{align*}
		%     \psi(h)((\mu,g)\cdot (f,f',\omega))& = \psi(h)(g\cdot f,g\cdot f',g\cdot \omega)\\
		%     &=h(g\cdot f, 0, g\cdot f',g\cdot \omega)\\
		%     & = h((0,g)\cdot (f,0,f',\omega))\\
		%     & =  \psi(h)(f,f',\omega),
		% \end{align*}
		% the linear map is well-defined.
		% For any $h\in MS$, let $h'$ be an $\mathcal{A}$-valued function on $\mathscr{X}\times \mathscr{X}\times V$ such that 
		% $h'(f, f_1, \omega)=h(f, 0, f_1, \omega)$ for all $(f, f_1, \omega)\in \mathscr{X}\times \mathscr{X}\times V$. For any $g\in G$, we have $h'(g\cdot f, g\cdot f_1, g\cdot \omega)=h(g\cdot f, 0, g\cdot f_1, g\cdot \omega)=h((0,g)\cdot (f, 0, f_1, \omega))=h(f, 0, f_1, \omega)=h'(f, f_1, \omega)$. That is, $h'$ is a $G$-invariant function. Therefore, we can define the following map.
		% \begin{align*}
		%     \psi: MS & \to \mathcal{MS}_{n, d}\\
		%           h & \mapsto h' 
		% \end{align*}
		Then we have $MS \cong \mathcal{MS}_{n, d}$ as linear spaces. 
		
		It remains to prove that $\Psi^{\ast}$ is an algebra homomorphism. 
		We have the following diagram,
		\begin{center}
			\begin{tikzcd}[sep=tiny]
				& X\times X\times X \arrow[ld, "p_{12}"'] \arrow[rd, "p_{23}"] \arrow[dd, "\hat{\Phi}", shift left] \arrow[rr, "p_{13}"]       &                                                    & X\times X \arrow[dd, "\Phi", shift left=5] \\
				X\times X \arrow[dd, "\Phi"]          &                                                                                                                                 & X\times X \arrow[d, "\Phi", no head, shift left=3] &                                            \\
				& \mathscr{X}\times \mathscr{X}\times \mathscr{X}\times V\times V \arrow[ld, "q_{124}"'] \arrow[rd, "q_{235}"] \arrow[rr, "q_{13}"] & {} \arrow[d, shift left=3]                         & \mathscr{X}\times \mathscr{X}\times V      \\
				\mathscr{X}\times \mathscr{X}\times V &                                                                                                                                 & \mathscr{X}\times \mathscr{X}\times V              &                                
			\end{tikzcd}
		\end{center}
		where 
		\begin{align*}
			&q_{13}(f_1,f_2,f_3,v_1,v_2)=(f_1,f_3,v_1+v_2)\\
			&\hat{\Phi}(f_1, v_1, f_2, v_2, f_3,v_3)=(f_1, f_2, f_3, v_2-v_1, v_3-v_2).
		\end{align*}
		and $p_{ij},q_{ijk}$ are the obvious projection to the $(i,j)$, $(i,j,k)$ components, respectively.
		For $h_1, h_2\in \mathcal{MS}_{n, d}$, the convolution product $h_1\ast h_2=q_{13!}(p^{*}_{124}h_1\otimes p^{*}_{235}h_2)$, where $q_{13!}$ is the pushforward along the $q_{13}$ and $p^{*}_{124}$ and $p^{*}_{235}$ are the pullbacks along $p_{124}$ and $p_{235}$. Similarly, for any $h'_1, h'_2\in MS$, $h'_1\ast h'_2=p_{13!}(p^{*}_{12}h'_1\otimes p^{*}_{23}h'_2)$.
		Let 
		$$
		Z=\{(a, b)\ | \ a \in X\times X,  b \in \mathscr{X}\times \mathscr{X}\times \mathscr{X}\times V\times V, q_{13}(b)=\Phi(a)\}.
		$$
		Suppose $a=(f_1,v_1,f_3,v_3)$ and $b=(f'_1,f'_2,f'_3,v'_1,v'_2)$, the condition $q_{13}(b)=\Phi(a)$ implies $f_1=f'_1, f_2=f'_2$ and $v_3-v_1=v'_1+v'_2$, so we can construct an element $(f_1,v_1,f'_2,v'_1+v_1,f_3,v_3)\in X\times X\times X$. Then we define a bijection 
		$$\phi:X\times X\times X\to Z, (f_1,v_1,f_2,v_2,f_3,v_3)\mapsto (p_{13}(f_1,v_1,f_2,v_2,f_3,v_3), \hat{\Phi}(f_1,v_1,f_2,v_2,f_3,v_3)).$$
		Therefore we have the following cartesian square.
		\begin{center}
			\begin{tikzcd}
				X\times X\times X \arrow[rr, "p_{13}"] \arrow[d, "\hat{\Phi}"]                      &  & X\times X \arrow[d, "\Phi"]           \\
				\mathscr{X}\times \mathscr{X}\times \mathscr{X}\times V\times V \arrow[rr, "q_{13}"] &  & \mathscr{X}\times \mathscr{X}\times V.
			\end{tikzcd}
		\end{center}
		Then for any $h_1,h_2\in \mathcal{MS}_{n, d}$, we have
		\begin{align*}
			\Phi^{\ast}q_{13!}(q_{12}^{*}h_1\otimes q_{23}^{*}h_2)& = p_{13!}\hat{\Phi}^{*}(q_{124}^{*}h_1\otimes q_{235}^{*}h_2)\\
			& = p_{13!}(\hat{\Phi}^{*}q_{124}^{*}h_1\otimes \hat{\Phi}^{*}q_{235}^{*}h_2)\\
			& = p_{13!}(p_{12}^{*}\Phi^{*}h_1 \otimes p_{23}^{*}\Phi^{*}h_2).
		\end{align*}
		Hence $\Phi^*$ is an algebra isomorphsim.
		
		The proofs for the algebra isomorphism $\Phi'^{\ast}:\mathcal{MH} \to MH$ and the linear space isomorphism $\tilde{\Phi}^{\ast}:\mathcal{MV} \to MV$ are similar. Then the proposition follows.
	\end{proof}
	The isomorphism $\tilde{\Phi}^{\ast}:\mathcal{MV} \to MV$ induce an isomorphism $\End_{\mathcal{A}}{\mathcal{MV}}\to \End_{\mathcal{A}}{MV}$ by sending $\varphi$ to $\tilde{\Phi}^{*}\varphi(\tilde{\Phi}^*)^{-1}$.
	\begin{prop}\label{prop 6.1.7}
		We have the following commute diagram.
		\begin{center}
			\begin{tikzcd}
				\mathcal{MS}_{n,d} \arrow[r,left] \arrow[d, "\Phi^{*}"'] & \End_{\mathcal{A}}{\mathcal{MV}} \arrow[d, "\varphi"'] & \mathcal{MH} \arrow[l, right] \arrow[d, "\Phi'^{*}"]   \\
				MS \arrow[r, left]          & \End_{\mathcal{A}}{MV}                  & MH \arrow[l, right]         
			\end{tikzcd}
		\end{center}
	\end{prop}
	
	\begin{proof}
		We only prove that the left square is commute, and the right one is similar.
		For any $h_1\in \mathcal{MS}_{n, d}$, $h_2\in \mathcal{MV}$ and $(f,\omega, f', \omega')\in X\times Y$, let $\varphi_{h_1}$ denote the image of $h_1$ in $\End_{\mathcal{A}}{\mathcal{MV}}$, we have
		\begin{align*}
			&\tilde{\Phi}^{*}\varphi_{h_1}(\tilde{\Phi}^{*})^{-1}(h_2)(f, \omega, f', \omega')\\
			& = \tilde{\Phi}^{*}(h_1\ast (\tilde{\Phi}^{*})^{-1}(h_2))(f, \omega, f', \omega')\\
			&= h_1\ast (\tilde{\Phi}^{*})^{-1}(h_2)(f, f', \omega'-\omega)\\
			&=\sum\limits_{(f'', \mu)\in X}h_1(f, f'', \mu)(\tilde{\Phi}^{*})^{-1}(h_2)(f'', f', \omega'-\omega-\mu)\\
			&= \sum\limits_{(f'', \mu)\in X}\Phi^{*}(h_1)(f, 0, f'', \mu)h_2(f'',\mu, f', \omega'-\omega)\\
			&=\Phi^{*}(h_1)\ast (h_2)(f, 0, f', \omega'-\omega)=\Phi^{*}(h_1)\ast (h_2)(f, \omega, f', \omega').
		\end{align*}
	\end{proof}
	Proposition \ref{prop 6.2.3} and Proposition \ref{prop 6.1.7} imply the following theorem.
	\begin{thm}\label{thm 6.2.8}
		The $\mathcal{MS}_{n, d}$-action in Proposition \ref{prop 4.3.1} and the $\mathcal{MH}$-action in  Proposition \ref{prop 4.4.1} 
		on $\mathcal{MV}$ form a double centralizer for $n \geq d$, i.e.,  
		\begin{align*}
			\End_{\mathcal{MS}_{n, d}}(\mathcal{MV}) \cong \mathcal{MH}, \ \End_{\mathcal{MH}}(\mathcal{MV}) \cong \mathcal{MS}_{n, d}.
		\end{align*}
	\end{thm}
	
\end{document}